\documentclass{article}
\usepackage{lmodern}            
\usepackage[T1]{fontenc}         
\usepackage[utf8]{inputenc}      
\usepackage[english]{babel}
\usepackage{bbold}
\usepackage{xspace}
\usepackage{amsmath,amssymb}
\usepackage{amsthm}
\usepackage{newtxmath}
\usepackage{stmaryrd}
\usepackage{dsfont}
\usepackage[margin=3.2cm]{geometry}
\usepackage[unicode=true]{hyperref}
\usepackage{xcolor}
\usepackage{enumitem}
\hypersetup{
    colorlinks,
    linkcolor={black},
    citecolor={black},
    urlcolor={black}}
\usepackage{graphicx}
\usepackage{subcaption}
\graphicspath{ {./images/} }
\input{insbox}

\newtheorem{theorem}{Theorem}[section]
\newtheorem{corollary}[theorem]{Corollary}

\newtheorem{lemme}[theorem]{Lemma}
\newtheorem{proposition}[theorem]{Proposition}
\newtheorem{remark}[theorem]{Remark}
\theoremstyle{definition}
\newtheorem{definition}[theorem]{Definition}

\newcommand{\A}{\mathbb{A}}
\newcommand{\B}{\mathbb{B}}		% boule
\newcommand{\E}{\mathbb{E}}		% espérance
\newcommand{\e}{\mathbf{e}}		% espérance
\newcommand{\K}{\mathbb{K}}		% corps quelconque
\newcommand{\N}{\mathbb{N}}	
\newcommand{\n}{\mathbf{n}}	% entiers naturels

\renewcommand{\P}{\mathbb{P}}	% proba
\newcommand{\Q}{\mathbb{Q}}		% rationnels
\newcommand{\q}{\mathbf{q}}		% rationnels
\newcommand{\R}{\mathbb{R}}		% réels
\newcommand{\U}{\mathbb{U}}		% racines de l'unité
\newcommand{\V}{\mathbb{V}}

\newcommand{\Z}{\mathbb{Z}}		% entiers relatifs

\newcommand{\cA}{\mathcal A}		% tribu
\newcommand{\cC}{\mathcal C}	% tribu

\newcommand{\cE}{\mathcal E}	
\newcommand{\cF}{\mathcal F}		% ensemble des fonctions
\newcommand{\cH}{\mathcal H}	% quaternions, Hilbert
\newcommand{\cI}{\mathcal I}		% idéal
\newcommand{\cL}{\mathcal L}		% applications linéaires
\newcommand{\cM}{\mathcal M}	% matrices
\newcommand{\cQ}{\mathcal Q}	% O de LANDAU
\newcommand{\cP}{\mathcal P}		% propriété
\newcommand{\cS}{\mathcal S}
\newcommand{\cT}{\mathcal T}		% Tribu
\newcommand{\cU}{\mathcal U}	% loi uniforme

\renewcommand{\d}{\mathop{}\!\mathrm{d}}

\newcommand{\dx}{\d x}
\newcommand{\dy}{\d y}

\newcommand{\1}{\mathds{1}}			% indicatrice
\definecolor{Myblue}{rgb}{0.2,0.2,0.7}
\newtheoremstyle{correction}{0.75em}{0.75em}{\color{Myblue}}{}{\bfseries}{.}{ }{\thmname{#1}\thmnumber{ #2}\thmnote{ (#3)}}
\theoremstyle{correction}

\usepackage{biblatex} %Imports biblatex package
\title{Constructing the Brownian sphere from a continuum random unicycle}
\author{Mathieu Mourichoux}
\date{}

\begin{document}
\maketitle
\begin{abstract}
    We give an explicit construction of the Brownian sphere biased by the distance between two distinguished points, which is based on the Miermont bijection for quadrangulations. We then describe various conditionings of this object, which are related to Voronoï cells in the Brownian sphere. In particular, we give a new construction of the Brownian sphere with two distinguished points at a fixed distance. We also use this construction to derive a new representation of the bigeodesic Brownian plane.
\end{abstract}
\tableofcontents
\section{Introduction}

This work is concerned with the Brownian sphere $(\cS,D)$, which is a random compact metric space that appears as the scaling limit of several models of random planar maps (see \cite{uniqueness,convergence,ConvergenceBJM,ConvergenceSimple}). The first construction of the Brownian sphere in \cite{uniqueness,convergence} was based on Brownian motion indexed by the Continuum Random Tree (CRT), and is a continuous counterpart of the CVS bijection \cite{cori_vauquelin_1981,schaeffer}. In this construction, the Brownian sphere also comes with a root $x_0$, a distinguished point $x_*$, and a volume measure $\mu$ of total mass $1$. 

The purpose of this paper is to present another construction of the Brownian sphere. Having multiple constructions for the same object is a common feature of Brownian geometry. For instance, the Brownian disk was introduced and constructed in \cite{Browniandisk}, but three other constructions have been introduced since (see \cite{disque,Diskboundary,DiskMarkov}). Moreover, each of these constructions enables to study different geometric properties of the Brownian disk. However, in the case of the Brownian sphere, all the known constructions are essentially equivalent, since they correspond to rerooting the underlying CRT at specific points (see \cite{bessel,Bigeodesicbrownianplane}). The new construction given in this paper is inspired from the Miermont bijection \cite{tessalations}. A particular case of this bijection yields a correspondence between well-labelled unicycles (i.e. planar a map with two faces) and bi-pointed quadrangulations with an extra parameter, called the delay. Since this bijection generalizes the CVS bijection, it is natural to expect that a continuous version of the Miermont bijection exists. This will be the content of our main result. More precisely, let $(Q_n,x_1,x_2,\delta_n)$ be a uniform random variable in the set of delayed quadrangulations with $n$ faces, which are quadrangulations with two distinguished vertices $x_1,x_2$, and an extra parameter $\delta_n$. We prove that as $n\rightarrow\infty$, once properly renormalized, this random metric space converges for the Gromov-Hausdorff-Prokhorov topology towards a continuum random surface. The limiting space $(\cS_b,x_*,\overline{x}_*,\mathrm{Vol})$ has the law of a standard Brownian sphere biased by the distance between two distinguished points, so we call it the \textbf{biased Brownian sphere}. In Theorem \ref{equiv def}, we give an explicit construction of this random space, which is the continuous counterpart of the Miermont bijection, and can be interpreted as the quotient of the continuum random labelled unicycle (CRLU).

This construction enables us to investigate new geometric properties of the biased Brownian sphere. For instance, given $(\cS_b,x_*,\overline{x}_*)$, consider a random variable $\Delta$ which is uniform on $(-D(x_*,\overline{x}_*),D(x_*,\overline{x}_*))$, and the set
 \[\Theta_\Delta=\{x\in\cS_b,\,D(x,x_*)\geq D(x,\overline{x}_*)+\Delta\}.\]
This set corresponds to a generalization of the Voronoï cells, which would be obtained when $\Delta=0$, and we call it the \textbf{$\Delta$- delayed Voronoï cell of $x_*$ with respect to $\overline{x}_*$}. Among other things, in Proposition \ref{volume cells}, we show that $\mathrm{Vol}(\Theta_\Delta)$ follows a Beta distribution with parameter $(1/4,1/4)$.

Then, we also introduce the free version of the biased Brownian sphere, meaning that the volume is not fixed. It has a construction based on Poisson point measures, which allows us to perform many explicit calculations. In particular, we are able to define several singular conditionings, and to obtain a new construction of the Brownian sphere with two distinguished points at a fixed distance. More precisely, in Theorem \ref{Representation libre}, for every $b\in(0,\infty)$ and $\delta\in(-b,b)$, we give an explicit construction of the Brownian sphere $(\cS,x_*,\overline{x}_*)$ conditioned on $\{D(x_*,\overline{x}_*)=b\}$ (and with random volume) as the quotient of a random unicycle, where the faces of the unicycle correspond to $\Theta_\delta$ and its complement. Using these constructions, in Corollary \ref{dist boundary}, we show that as we move on the boundary of $\Theta_\delta$, the distances to $x_*$ and $\overline{x}_*$ evolve, up to a deterministic additive constant, as a process whose law is absolutely continuous with respect to Ito's measure of positive excursions of linear Brownian motion $\n(\d\e)$, with density function 
\[144b^3\left(\int_0^\sigma\int_0^\sigma\frac{dsdt}{(\e_s+\frac{b+\delta}{2})^3(\e_t+\frac{b-\delta}{2})^3}\right)\exp\left(-3\int_0^\sigma\left(\frac{1}{(\e_u+\frac{b+\delta}{2})^2}+\frac{1}{(\e_u+\frac{b-\delta}{2})^2}\right)du\right).\]

Finally, this new construction of the Brownian sphere has interesting applications to the study of the bigeodesic Brownian plane. This random surface was introduced in \cite{Bigeodesicbrownianplane} as the local limit of the Brownian sphere around a point of a typical geodesic. Using this new representation of the Brownian sphere, we derive a new construction of the bigeodesic Brownian plane. Moreover, this construction enables us to recover some invariance properties of this space, which were quite unnatural  from the point of view of \cite{Bigeodesicbrownianplane}.

Let us discuss some related works. In \cite{tessalations}, the author introduces the Miermont bijection, and uses it to prove that there exists a unique geodesic between two distinguished points of the Brownian sphere. Several results of this paper motivated our work, in particular in Section \ref{section bigeodesique}. We also mention that the Miermont bijection is actually much more general than what we use in our work, and holds for maps of any genus with any fixed number of distinguished vertices. We believe that the methods developed in this paper can be generalized to these cases as well. 

In the paper \cite{Compactbrowniansurfaces}, the authors prove the convergence of uniform quadrangulations of arbitrary genus towards \textit{Brownian surfaces}. To do so, they rely on building blocks, called \textit{composite slices} and \textit{quadrilaterals with geodesic sides}. They also introduce a general framework to glue metric spaces along geodesics. Their results are particularly important for this work, since our main object is obtained as the gluing of the scaling limit of quadrilaterals with geodesic sides. 

In \cite{stablesphere}, the authors prove that random planar maps with large faces converge toward some explicit random metric space. To obtain the uniqueness of geodesics between typical points, they introduce a generalization of the BDG bijection, which is very similar to the Miermont bijection. We believe that the ideas used in this paper can also be useful in this context.

This paper is organized as follows. In Section \ref{section discrete}, we present some links between the discrete objects that we will consider, and we study some properties of uniform large well-labelled unicycles. In Section \ref{Section convergence}, we prove that delayed quadrangulations converge towards the biased Brownian sphere. We also give an explicit construction of the limiting space, and study some of its geometric properties. In Section \ref{section free model}, we introduce the free version of the biased Brownian sphere. We also compute several quantities of this model, and give an explicit construction when we fix the delay. Finally, in Section \ref{section bigeodesique}, we derive a new construction of the bigeodesic Brownian plane, and recover some invariance properties of this model. 

\subsection{Acknowledgements}

I would like to thank Grégory Miermont for his support and for many enlightening discussions. I would also like to thank Armand Riera for interesting conversations about Voronoï cells in the Brownian sphere, and Sasha Bontemps for her assistance in several calculations on the blackboard. Finally, I thank Lou Le Bihan and Simon Renouf for their help in typing this article.

\section{Combinatorial tools}\label{section discrete}

In this section, we present the discrete objects considered in this work, together with several links between them. In particular, we show that delayed quadrangulations are closely related to quadrangles with geodesic sides.

\subsection{The CVS bijection and its generalizations}

We begin this section with some preliminaries.

\paragraph{The Miermont bijection for $n=2$.}

Consider the embedding of a finite connected graph into the sphere $\mathbb{S}^2$, where loops and multiple edges are allowed. A planar map is an equivalence class of such embeddings, modulo orientation-preserving homeomorphisms of the sphere. For a planar map $m$, let $V(m)$ and $E(m)$ denote the set of vertices and edges of $m$. The faces of $m$ are the connected components of the complement of its edges. The degree of a face is the number of oriented edges that are incident to it. A corner of a vertex $v\in V(m)$ is an angular sector formed by two consecutive edges (taken in clockwise order) around $v$. We say that $m$ is a quadrangulation if all its faces have degree $4$. In this paper, we will only consider rooted maps, which are maps with a distinguished oriented edge. This edge is called the root edge, and its origin is called the root vertex. The root corner is the incident to the root edge, and lying at its left. We denote by $\Q_n$ the set of rooted quadrangulations with $n$ faces. Every map $m$ can be considered as a metric space, equipped with the graph distance $d_m$.

A unicycle is a planar map with two faces labelled as $f_1$ and $f_2$. Moreover, these faces are distinguished, and called the \textbf{internal face} and the \textbf{external face}. A well-labelled unicycle $(\mathbf{u},\ell)$ is a unicycle $\mathbf{u}$ equipped with a labelling function $\ell:V(\mathbf{u})\rightarrow\Z$ such that 
\[|\ell(x)-\ell(y)|\leq1,\quad\quad\text{ for all adjacent vertices }x,y,\]
and such that the root vertex has a label $0$.  We denote by $\U_n$ the set of rooted well-labelled unicycles with $n$ edges.

We say that a tuple $(\q,v_0,v_1,d)$ is a \textbf{delayed quadrangulation} if :
\begin{itemize}
    \item $\q$ is a rooted quadrangulations,
    \item $v_0$ and $v_1$ are two distinct vertices of $\q$, 
    \item $d$ is an integer with the same parity as $d_\q(v_0,v_1)$ such that $|d|<d_\q(v_0,v_1)$.
\end{itemize}

It is shown in \cite{tessalations} that given a bi-pointed quadrangulation $(\q,v_0,v_1)$, there exist an admissible delay if and only if $v_0$ and $v_1$ are not neighbors or equal. We denote by $\Q^{(b)}$ the set of delayed quadrangulations, and by $\Q_n^{(b)}$ the set of delayed quadrangulations with $n$ faces.

\begin{figure}
\centering
\begin{subfigure}{.5\textwidth}
  \centering
  \includegraphics[width=.85\linewidth]{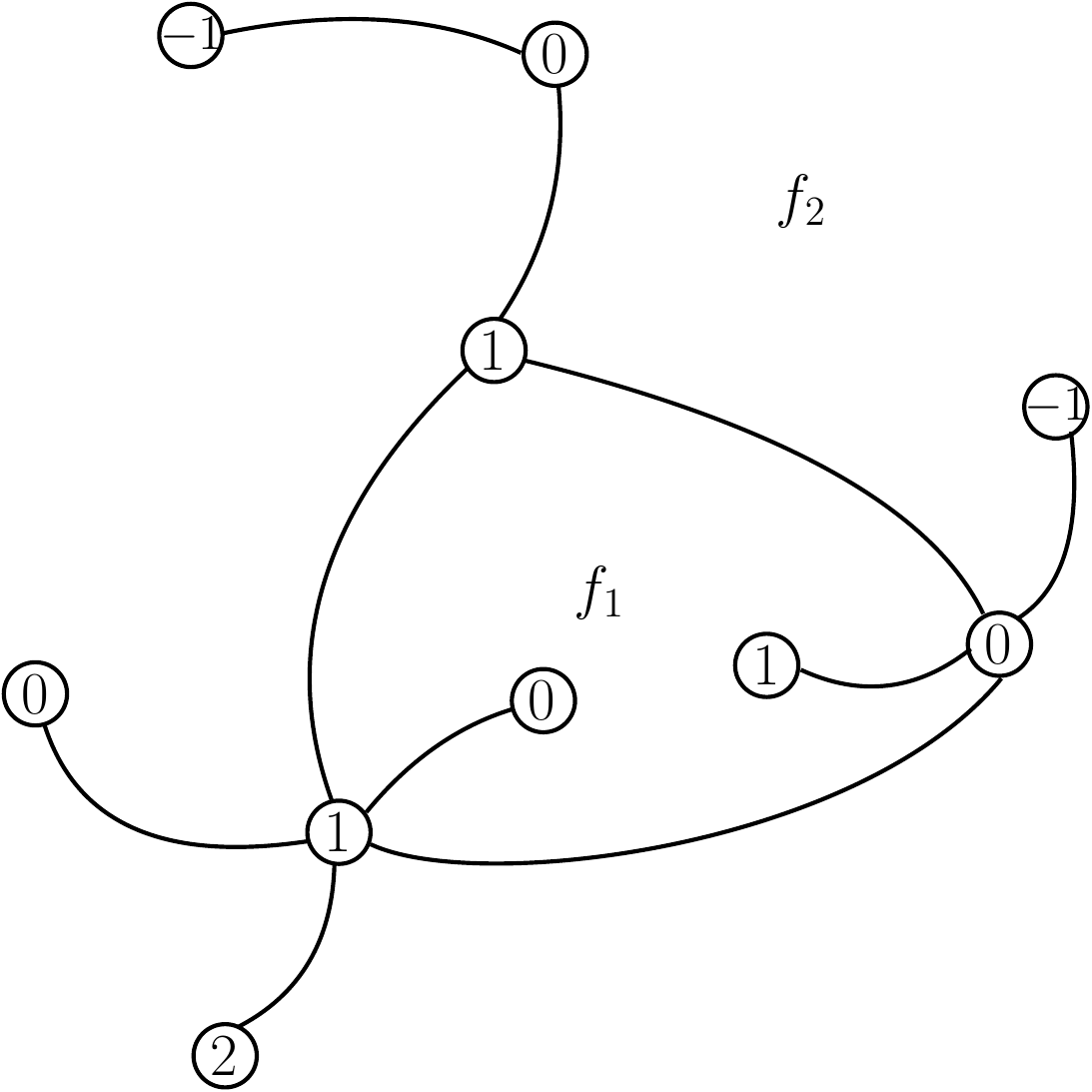}
\end{subfigure}%
\begin{subfigure}{.5\textwidth}
  \centering
  \includegraphics[width=.9\linewidth]{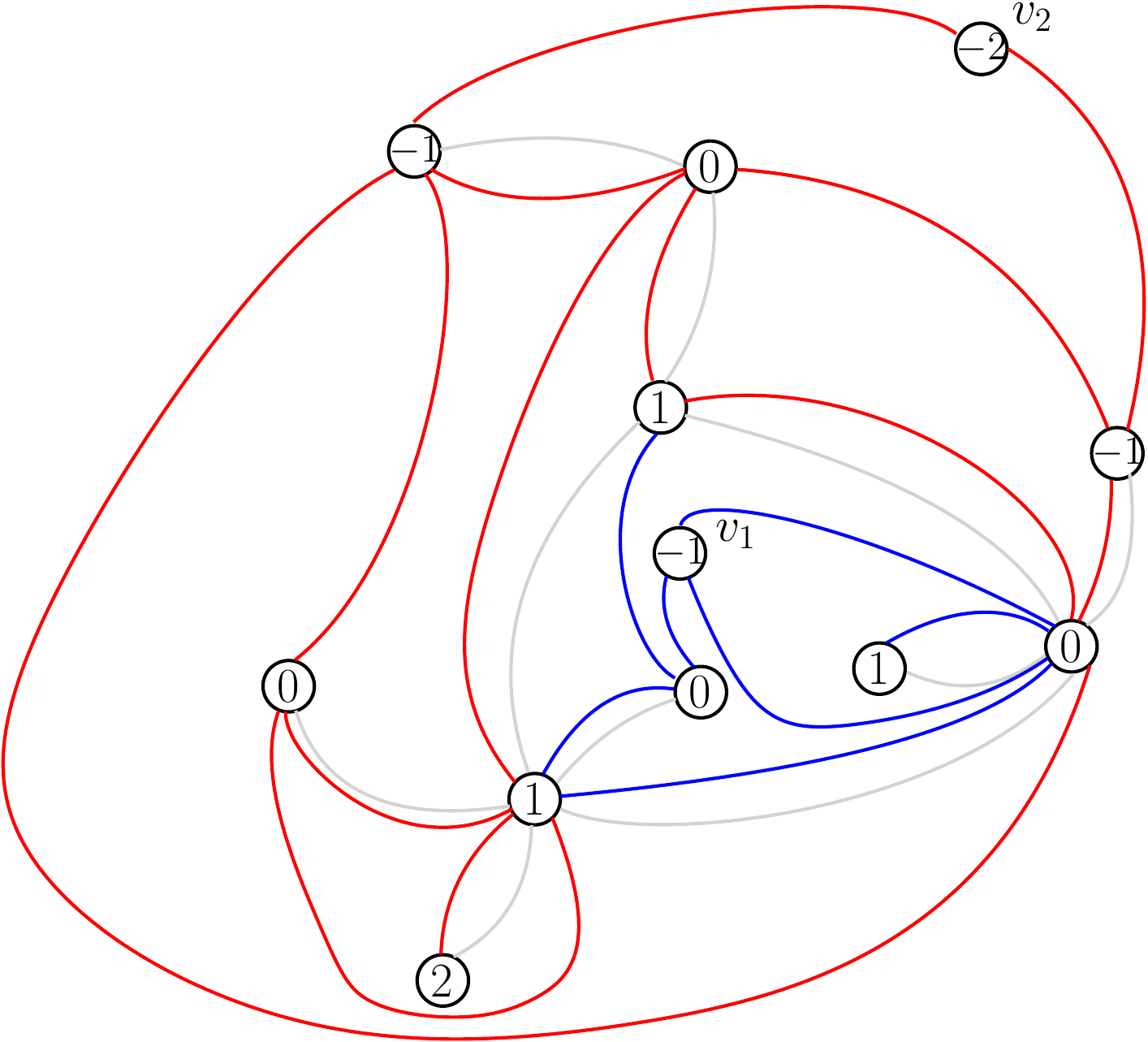}
\end{subfigure}
\caption{Illustration of the Miermont bijection. In this case, the delay is $-1$.}
\label{Miermont bijection}
\end{figure}

The Miermont bijection is a bijection between $\U_n$ and $\Q_n^{(b)}$, which goes as follows (see Figure \ref{Miermont bijection} for an example). Given a well-labelled unicycle $\mathbf{u}\in \U_n$, let $(c_0,...,c_m)$ be the corner sequence of the external face of $u$ starting from any arbitrary corner, that we extend periodically. Note that the labelling function $\ell$ can naturally be extended to the corners of $\mathbf{u}$, by $\ell(c)=\ell(v)$ if $c$ is incident to $v$. For every $0\leq i \leq m$, let
\[\sigma_i=\inf\{j>i,\,\ell(c_j)<\ell(c_i)\}\]
(this quantity can be infinite). We construct a new graph as follows. First, we add a vertex $v_1$ in the external face of $\mathbf{u}$, and we set $\ell(v_1)=\inf_{v\in f_1}\ell(v)-1$. Then, for every $i$ such that $\sigma_i<\infty$, draw an arc between $c_i$ and $c_{\sigma_i}$. On the other hand, for every $i$ such that $\sigma_i=\infty$, draw an arc between $c_i$ and $v_1$. It is possible to draw the arcs in such a way that they do not intersect each other, and do not intersect the edges of $\mathbf{u}$. Then, we repeat the exact same construction in the inner-face of $\mathbf{u}$, and we call $v_2$ extra vertex that we add in this face. Finally, erase all the edges of $\mathbf{u}$. The resulting graph is rooted at the edge between the root corner of $\mathbf{u}$ and its successor (there are two possible choices for the orientation). Moreover, this graph has two distinguished vertices, which are $v_1$ and $v_2$, and a parameter, which is $\ell(v_2)-\ell(v_1)$. We call this graph $\mathrm{CVS}(\mathbf{u})$. The following result is always a restatement of \cite[Theorem 4 and Corollary 1]{tessalations}.

\begin{theorem}
    The mapping $\mathrm{CVS}:\U_n\rightarrow\Q_n^{(b)}$ is one-to-two.
\end{theorem}
\begin{remark}
    The theorem in \cite{tessalations} is in fact much more general, because it deals with $k$-pointed delayed quadrangulations of arbitrary genus $g$. However, since we do not need this level of generality in this paper, we only presented the bijection for $k=2$ and $g=0$.
\end{remark}

This bijection was originally introduced to prove that there exists a unique geodesic between two typical points in the Brownian sphere (see \cite[Theorem 3]{tessalations}). We also mention that this bijection is related to the delayed Voronoï cells of elements of $\Q_n^{(b)}$, which has been used several times in the literature \cite{Guitter_Universal_law_Voronoi,Guitter_Proof_Chapuy, Guitter_statistics}. As we do not need it, we do not give more details, but we will come back to this fact in Section \ref{construction}. Similarly, we do not give the inverse bijection. The interested reader is referred to \cite{tessalations} for more information about this bijection.

\paragraph{Quadrilaterals with geodesic sides.}\label{quadrilateral}

A \textbf{vertebrate} is a plane tree with two distinguished and distinct vertices, called $\rho$ and $\overline{\rho}$. Such a tree $\mathbf{v}$ can be represented in $\R^2$ as follows :
\begin{itemize}[label=\textbullet]
    \item the floor, which is the path of length $h\geq 1$ from $\rho$ to $\overline{\rho}$, corresponds to the points $(i,0)\in\R^2$ for $0\leq i\leq h$,
    \item the upper part, which consists of $h$ trees that are contained in the upper-half plane $\R\times\R_+$, and where the $i$-th tree is grafted to $(i-1,0)$, $1\leq i\leq h$,
    \item the lower part, which consists of $h$ trees that are contained in the lower-half plane $\R\times\R_-$, and where the $i$-th tree is grafted to $(h-i+1,0)$, $1\leq i\leq h$,
\end{itemize}
\begin{figure}
\centering
\includegraphics[scale=0.7]{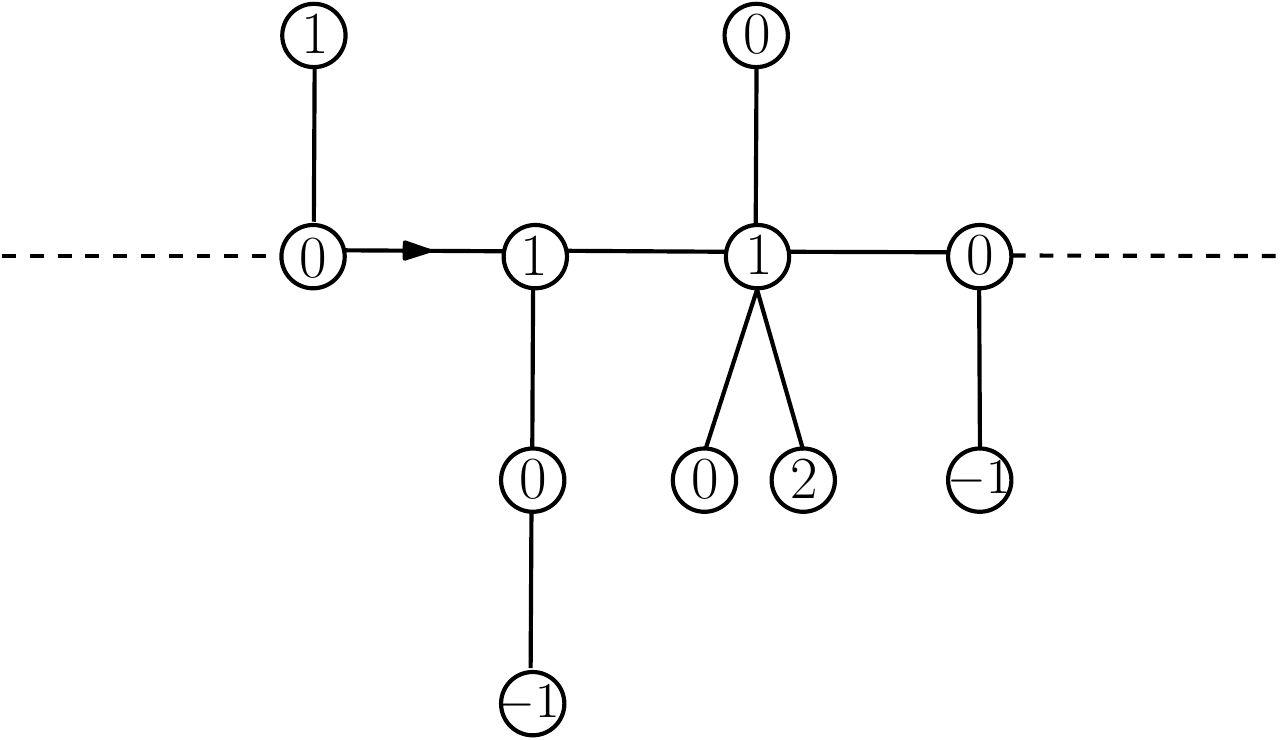}
\caption{A vertebrate with width 3,upper area 2, lower area 5 and tilt 0. The dashed line corresponds to $\R$.}
\label{vertebrate}
\end{figure}

(see Figure \ref{vertebrate} for an example). We will consider well-labelled vertebrates, which are vertebrates equipped with an integer-valued labelling function $\ell$ on their vertices such that $\ell(u)-\ell(v)\in\{-1,0,1\}$ whenever $u$ and $v$ are neighbors, and $\ell(\rho)=0$.

For a well-labelled vertebrate $\mathbf{v}$, let $a$ (resp. $\overline{a})$ stand for the number of edges that are strictly in the upper-half plane (resp. the lower-half plane). The \textit{width}, the \textit{upper area}, the \textit{lower area} and the \textit{tilt} of $\mathbf{v}$ are respectively the numbers 
\[h,\quad\quad a,\quad\quad\overline{a},\quad\quad-\ell(\overline{\rho}).\]
%For every $i,j,k\in\N$ and $l\in\Z$, we denote by $\V(i,j,k,l)$ the set of well-labelled vertebrates with width $i$, upper area $j$, lower area $k$ and tilt $l$.  

We can associate to every well-labelled vertebrate $\mathbf{v}$ a non-rooted quadrangulation of the plane $\mathrm{CVS}(\mathbf{v})$, called quadrilateral with geodesics sides, as follows. Let $I=\{c_0,c_1,...,c_{2a+h}\}$ be a clockwise enumeration of the corners of $\mathbf{v}$ that are incident to the upper-half plane, and set $\ell_*=\inf_{c\in I}\ell(c)-1$. Also, for every $0\leq i\leq 2a+h-1$, we define
\[\sigma_i=\inf\{j\in\{i+1,i+2,...,2a+h\}:\ell(c_j)=\ell(c_i)-1\}.\]
Then : \begin{itemize}[label=\textbullet]
    \item For every $1\leq i\leq \ell(\overline{\rho})-\ell_*+1$, we add a vertex $\xi_i$ in the upper-half plane, with label $\ell(\xi_i)=\ell(\overline{\rho})-i$. By convention, we set $\xi_0=\overline{\rho}$ and $v_*=\xi(\ell(\overline{\rho})-\ell_*+1)$.
    \item  For every  $0\leq i\leq \ell(\overline{\rho})-\ell_*$, draw an edge between $\xi_i$ and $\xi_{i+1}$. 
    \item For every  $0\leq i\leq 2a+h-1$, draw an arc between $c_i$ and $c_{\sigma_i}$ if $\sigma_i<\infty$, or between $c_i$ and $\xi(\ell(\overline{\rho})-\ell(c_i)+1)$ otherwise.
\end{itemize}

The arcs can be drawn in such a way that they remain in the upper half plane, that they do not intersect each other, and that they do not intersect the edges of $\mathbf{v}$.

\begin{figure}
\centering
\includegraphics[scale=0.6]{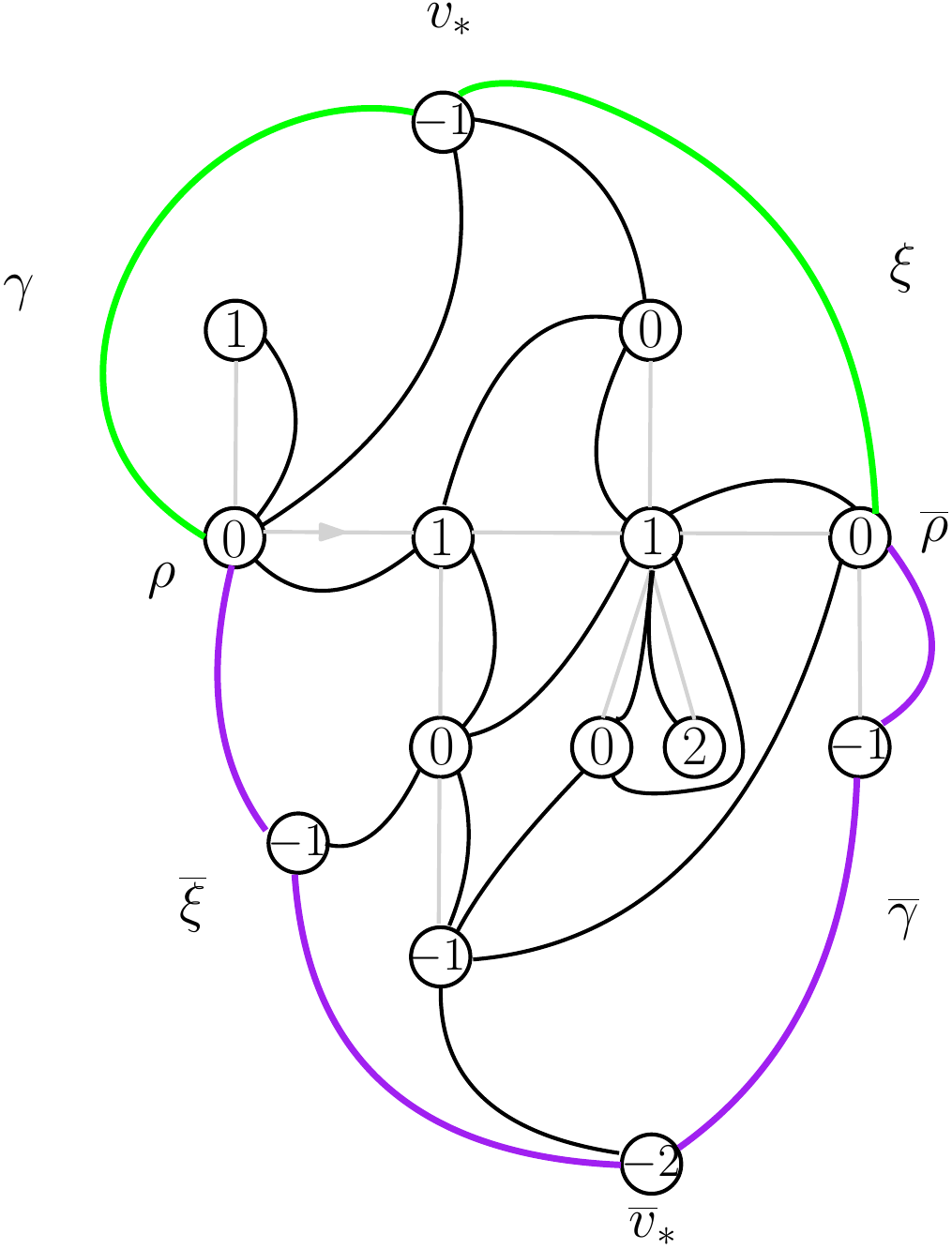}
\caption{The quadrangulation with geodesic boundaries associated to the vertebrate of Figure \ref{vertebrate}.}
\label{quad geodesic}
\end{figure}

We can then perform the same procedure for the lower part of $\mathbf{v}$, by replacing $I$ by $\overline{I}=\{\overline{c}_0,...,\overline{c}_{2\overline{a}+h}\}$, which is a clockwise enumeration of the corners of $\mathbf{v}$ incident to the lower-half plane, starting from $\overline{\rho}$. In this case, the points that we add are called $\overline{\xi}$, and the last vertex is called $\overline{v}_*$.

The resulting has only faces of degree $4$ except one, and is called a \textit{quadrilateral with geodesic sides}. These quadrangulations were introduced in \cite{Compactbrowniansurfaces}, as elementary pieces is the decomposition of quadrangulations of arbitrary genus.

The quadrilateral with geodesic sides $\mathrm{CVS}(\mathbf{v})$ naturally comes with two distinguished vertices, which are $v_*$ and $\overline{v}_*$. Furthermore, it inherits the notions of width, upper area, lower area and tilt coming from $\mathbf{v}$.
The boundary of $\mathrm{CVS}(\mathbf{v})$ is made of four geodesics, which are :
\begin{itemize}
    \item the left-most geodesic between $\rho$ and $v_*$, called $\gamma$,
    \item the left-most geodesic between $\overline{\rho}$ and $\overline{v}_*$, called $\overline{\gamma}$,
    \item the unique geodesic between $\overline{\rho}$ and $v_*$, which is $\xi$,
    \item the unique geodesic between $\rho$ and $\overline{v}_*$, which is $\overline{\xi}$
\end{itemize}
(see Figure \ref{quad geodesic}).
Finally, we equip every quadrilateral $\mathbf{qd}$ with a measure $\overline{\mu}_{\mathbf{qd}}$, which is the counting measure on the vertices of $\mathbf{qd}$ \textbf{that do not belong to $\xi$ or $\overline{\xi}$}.

\subsection{Link between delayed quadrangulations and quadrilaterals}

In this section, we present several links between unicycles, vertebrates, delayed quadrangulations and quadrilaterals. These links are based on combinatorial decompositions, and a comparison of the different variants of the CVS bijection.

First, let us introduce some definitions. Recall that a unicycle is a planar map with two faces, and that $\U_n$ stands for the set of rooted labelled unicycles with $n$ edges. We define $\U_n^\bullet$ to be the set of labelled unicycles with $n$ edges, and with a distinguished edge on their unique cycle. Note that this edge has a natural orientation, which is the one that has the external face on its right. Moreover, every element of $\U_n^\bullet$ also as a distinguished vertex on its cycle, which is the origin of its oriented distinguished edge.

Even though we are primarily interested in studying elements of $\U_n$, we will see that it is often more convenient to work with elements of $\U_n^\bullet$, which admit a simpler combinatorial decomposition. 

Then, let $\V_n$ the set of well-labelled vertebrates with $n$ edges and tilt $0$. There exists a simple bijection between the sets $\U_n^\bullet$ and $\V_n$, which is the following. 

Given a unicycle $\mathbf{u}\in\U_n^\bullet$, let $\Phi(\mathbf{u})$ be the planar graph obtained by cutting $\mathbf{u}$ at its distinguished vertex, as shown in Figure \ref{Unicycle vertebre}. It is clear that the resulting map is a well-labelled rooted vertebrate with $n$ edges and tilt $0$, and that the upper-part (respectively the lower-part) of this vertebrate corresponds to the internal face (respectively external face) of $\mathbf{u}$. Therefore, $\Phi(\mathbf{u})\in\V_n$. Similarly, for every $\mathbf{v}\in\V_n$, it is clear that $\Phi^{-1}(\mathbf{v})$ is the unicycle obtained by gluing the two vertices $\rho$ and $\overline{\rho}$ of $\mathbf{v}$, in such a way that the upper-part of $\mathbf{v}$ becomes the internal face of $\Phi^{-1}(\mathbf{v})$. 

\begin{figure}
\centering
\begin{subfigure}{.5\textwidth}
  \centering
  \includegraphics[width=1\linewidth]{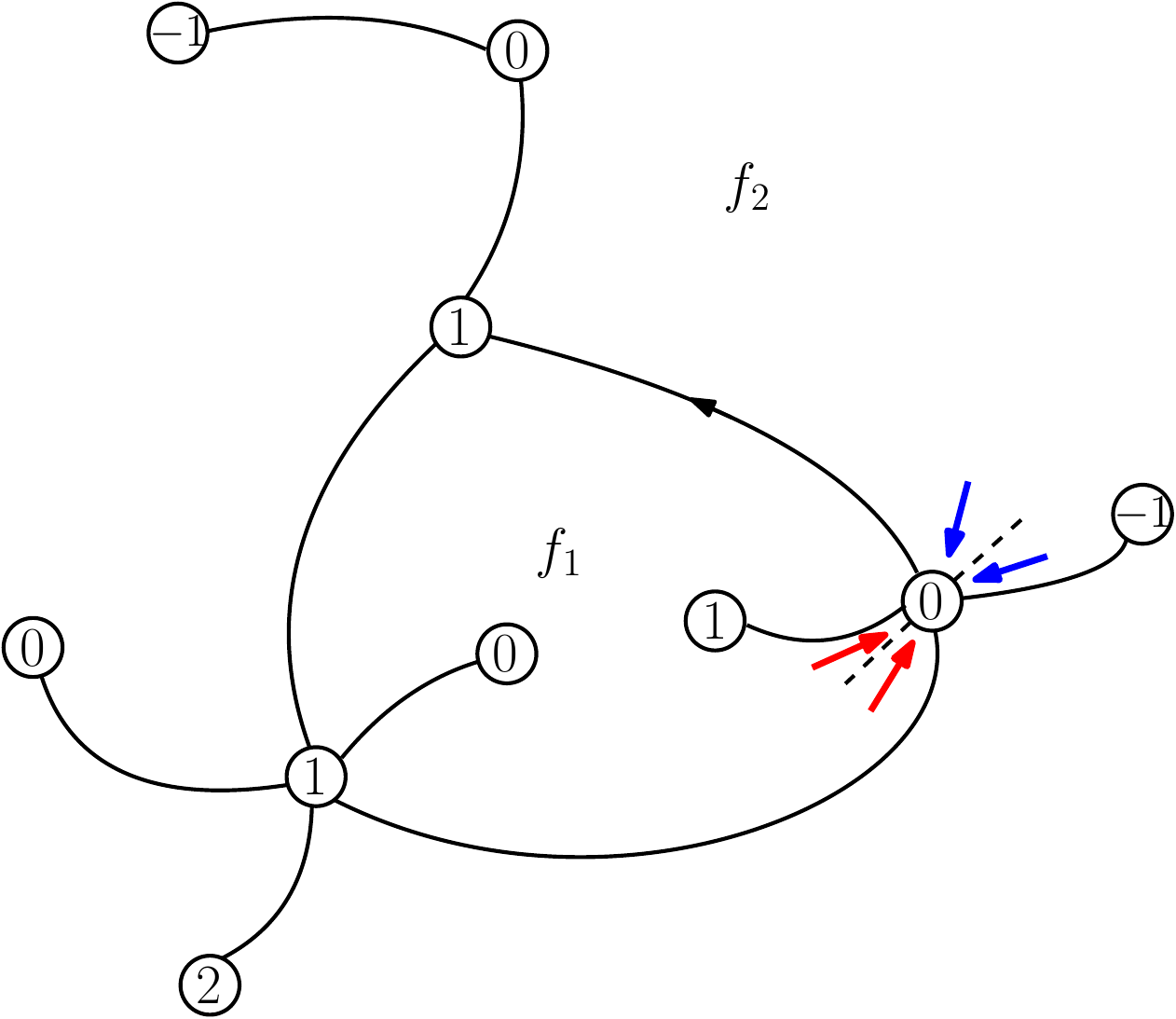}
\end{subfigure}%
\begin{subfigure}{.5\textwidth}
  \centering
  \includegraphics[width=.7\linewidth]{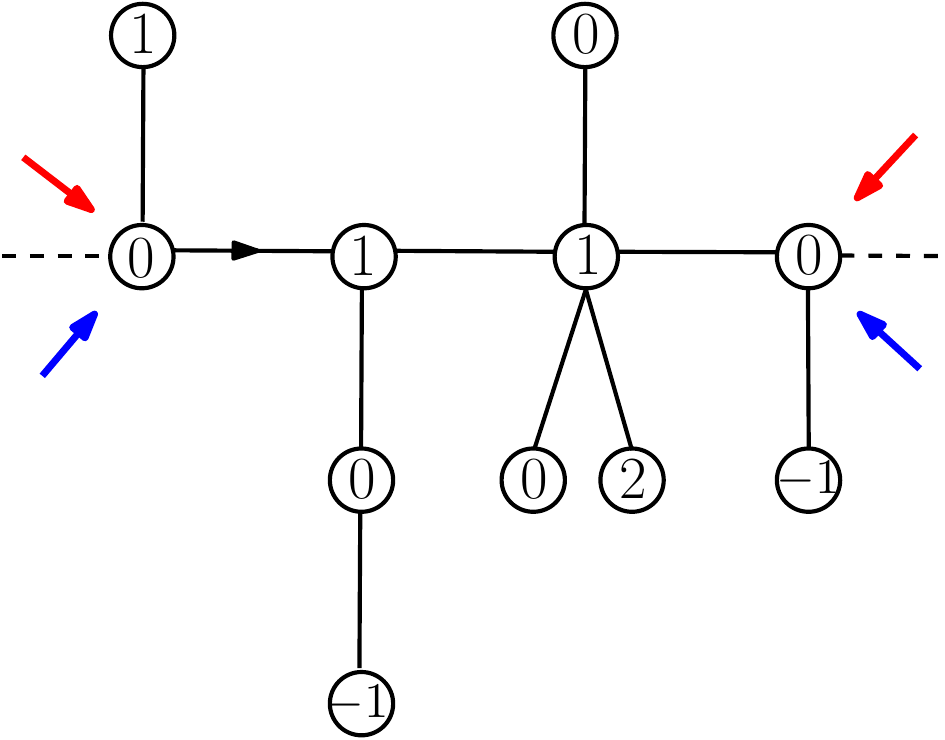}
\end{subfigure}
\caption{A unicycle $\mathbf{u}^\bullet\in U_{10}^\bullet$, and the corresponding vertebrate $\Phi(\mathbf{u}^\bullet)$.}
\label{Unicycle vertebre}
\end{figure}

\begin{remark}\label{correspondance}
    Observe that in this bijection, the length of the spine corresponds to the length of the cycle, and the area of the upper-part (respectively lower-part) corresponds to the area of the internal (respectively external) face. In particular, the law of the length of the spine and the area of the top-face of a uniform element in $\V_n$ is the same as the law of $(L_n^\bullet,A_n^\bullet)$. 
\end{remark}

We will show that the CVS bijection behaves well with respect to this bijection. For every $\mathbf{qd}\in \mathrm{CVS}(\V_n)$, recall that there are four distinguished geodesics $\gamma,\xi,\overline{\gamma},\overline{\xi}$ in $\mathbf{qd}$ between the two distinguished vertices $\rho,\overline{\rho}$ and $v_*,\overline{v}_*$. Moreover, since $\mathbf{qd}$ has tilt $0$, the geodesics $\gamma$ and $\xi$ (respectively $\overline{\gamma}$ and $\overline{\xi}$) have the same length. We define $\mathrm{Glue}(\mathbf{qd})$ as the planar map obtained by gluing $\gamma$ with $\xi$, and $\overline{\gamma}$ with $\overline{\xi}$ (see Figure \ref{Bijection cvs collage}). This planar map is naturally equipped with two distinguished vertices $v_*$ and $\overline{v}_*$ (which are the distinguished vertices of $\mathbf{qd}$), the counting measure on the vertices of $\mathrm{Glue}(\mathbf{qd})$, and a parameter defined as $\ell_{v_*}-\ell_{\overline{v}_*}$. The following proposition will be central in the proofs to come. 
\begin{proposition}\label{CVS collage}
    For every $\mathbf{u}^\bullet\in\U_n^\bullet$, we have 
    \[\mathrm{CVS}(\mathbf{u}^\bullet)=\mathrm{Glue}(\mathrm{CVS}(\Phi(\mathbf{u}^\bullet)))\]
    as elements of $\Q^{(b)}$.
\end{proposition}
\begin{proof}
    Fix $\mathbf{u}^\bullet\in\U_n^\bullet$. Note that we have a correspondence between the corners of $\mathbf{u}^\bullet$ and $\Phi(\mathbf{u}^\bullet)$, as shown in Figure \ref{Unicycle vertebre}. We also identify the first corner of $\rho$ incident to the upper half-plane to the last corner of $\overline{\rho}$ incident to the upper half-plane, and do the same in the lower half-plane (these corners correspond to the arrows of Figure \ref{Unicycle vertebre}). Let $c$ be a corner of $\Phi(\mathbf{u}^\bullet)$, and let $\sigma(c)$ be its successor for the CVS bijection \textit{on $\Phi(\mathbf{u}^\bullet)$}. Without loss of generality, suppose that $c$ belongs to the top face. There are three possibilities :
    \begin{itemize}[label=\textbullet]
        \item If $\sigma(c)$ is a corner of $\Phi(\mathbf{u}^\bullet)$, then one can see that $\sigma(c)$ is also the successor of $c$ for the CVS bijection on $\mathbf{u}^\bullet$.
        \item If $\sigma(c)$ belongs to a shuttle of $\mathrm{CVS}(\Phi(\mathbf{u}^\bullet))$, then $\sigma(c)=\xi(\ell_c-1)$. This also means that the successor of $c$ in the CVS bijection on $\mathbf{u}^\bullet$ is the corner associated to $\gamma(\ell_c-1)$. However, we can see that the edge between $c$ and $\gamma(\ell_c-1)$ in $\mathrm{CVS}(\mathbf{u}^\bullet)$ has the same role as the edge between $c$ and $\xi(\ell_c-1)$ in $\mathrm{Glue}(\mathrm{CVS}(\Phi(\mathbf{u}^\bullet)))$, since we have glued $\gamma$ with $\xi$. 
        \item If $\sigma(c)=v_*$, it is clear that the successor of $c$ in the CVS bijection on $\mathbf{u}^\bullet$ is $v_1$. 
    \end{itemize}
    Either way, the edges between the different corners are the same. Finally, it is straightforward to see that the delay of both maps are the same, which gives the result. 
\end{proof}
\begin{remark}
  Since, for every quadrilateral $\mathbf{qd}$, the counting measure $\overline{\mu}_{\mathbf{qd}}$ does not put any mass on $\xi$ and $\overline{\xi}$, it is easy to see that the push forward v of $\overline{\mu}_{\mathbf{qd}}$ by the gluing operation is the counting measure on $\mathrm{Glue}(\mathbf{qd})$.
\end{remark}

\begin{figure}
\centering
\begin{subfigure}{.5\textwidth}
  \centering
  \includegraphics[width=1\linewidth]{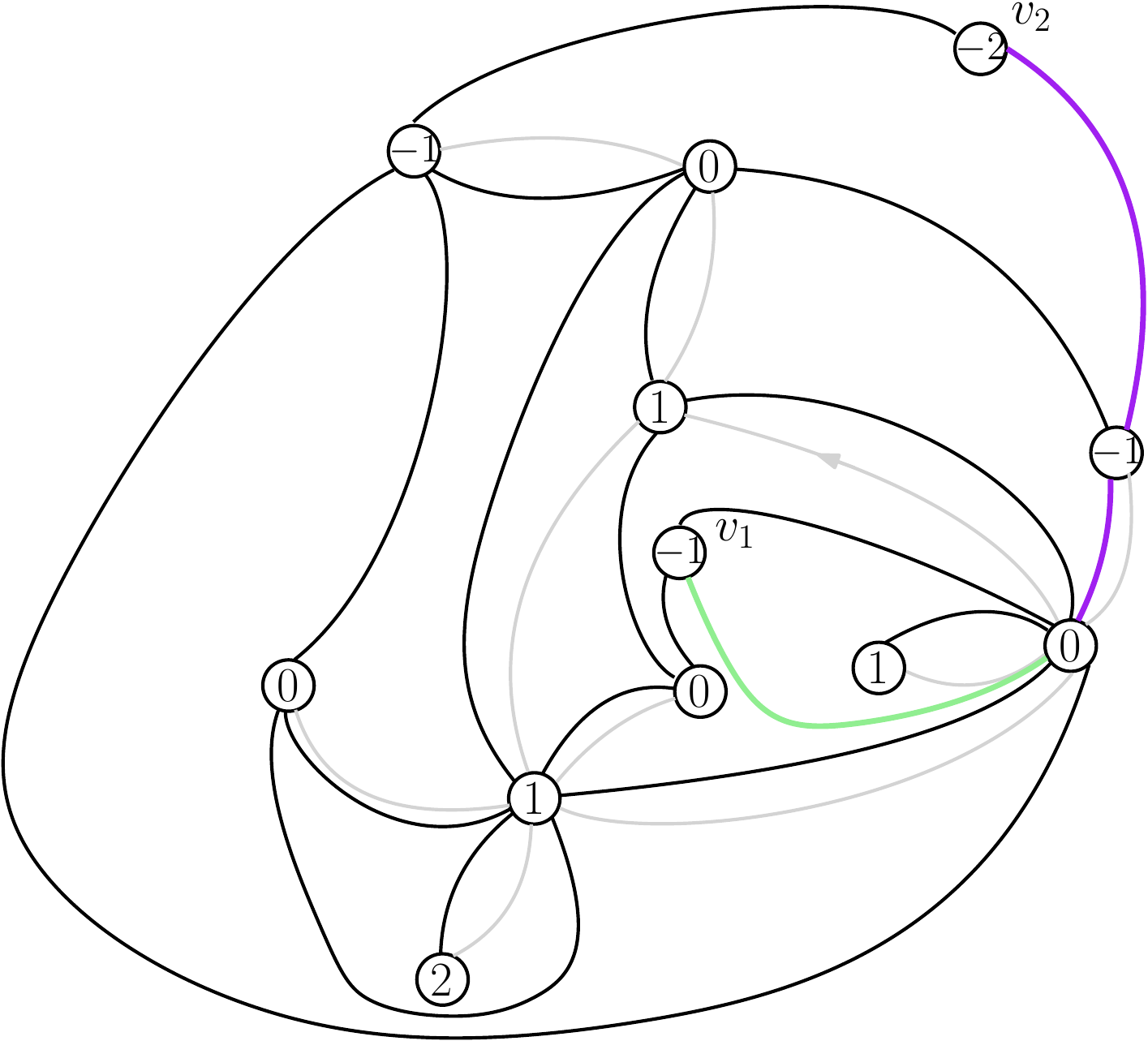}
\end{subfigure}%
\begin{subfigure}{.5\textwidth}
  \centering
  \includegraphics[width=.7\linewidth]{Images/quad_fig2.pdf}
\end{subfigure}
\caption{The correspondence between $\mathrm{CVS}(\mathbf{u}^\bullet)$ and $\mathrm{CVS}(\Phi(\mathbf{u}^\bullet))$, where $\mathbf{u}^\bullet$ is as in Figure \ref{Unicycle vertebre}.}
\label{Bijection cvs collage}
\end{figure}

We introduce some more notation. For every unicycle $\mathbf{u}$, let $L(\mathbf{u})$ denote the length of its unique cycle. If $\mathbf{u}$ has a distinguished oriented edge on its cycle, let $A(\mathbf{u})$ (resp. $\overline{A}(\mathbf{u})$) stand for the number of edges of $\mathbf{u}$ that are only incident to the external face (resp. the internal face). We emphasize that we do not count the edges of the cycle as being included in the faces. Consequently, for every $\mathbf{u}\in\U_n^\bullet$, we have 
\begin{equation}\label{somme}
    L(\mathbf{u})+A(\mathbf{u})+\overline{A}(\mathbf{u})=n.
\end{equation}

Then, let $U_n^\bullet$ be a uniform random variable on $\U_n^\bullet$. We introduce the following random variables
\begin{equation*}
    L_n^\bullet=L(U_n^\bullet),\quad A_n^\bullet=A(U_n^\bullet),\quad  \overline{A}_n^\bullet=\overline{A}(U_n^\bullet).
\end{equation*}

Finally, let $(Q^{(b)}_n,x^{(b)}_n,y^{(b)}_n,\mu_n^{(b)},\Delta_n)$ be a uniform random variable in the set $\Q_n^{(b)}$, and for every $a,a',h>0$, let $Qd_{a,a',h}$ be a uniform random variable in the set of quadrilaterals with geodesics sides with upper-area $a$, lower-area $a'$, width $h$ and tilt $0$. The following corollary is just a probabilistic reformulation of Proposition \ref{CVS collage}. 

\begin{corollary}[Link between biased quadrangulations and quadrangles]\label{lien}
    For every measurable function $F\geq 0$, we have
    \[\E\left[F(Q^{(b)}_n,x^{(b)}_n,y^{(b)}_n,\mu_n^{(b)},\Delta_n)\right]=\E\left[\frac{1}{L_n^\bullet}\right]^{-1}\E\left[\frac{F\left(\mathrm{Glue}\left(Qd_{A_n^\bullet,\overline{A}_n^\bullet,L_n^\bullet}\right)\right)}{L_n^\bullet}\right].\]
\end{corollary}
\begin{proof}
  For every non-negative measurable function $F$, we have 
    \begin{align*}
        \E[F(U_n)]&=\frac{1}{\#\U_n}\sum_{\mathbf{u}\in\U_n}F(\mathbf{u})\\
        &=\frac{1}{\#\U_n}\sum_{\mathbf{u}\in\U_n}\frac{1}{L(\mathbf{u})}\sum_{v\in C(\mathbf{u})}F(\mathbf{u})\\
        &=\frac{\#\U_n^\bullet}{\#\U_n}\frac{1}{\#\U_n^\bullet}\sum_{\mathbf{u}^\bullet\in\U_n^\bullet}\frac{1}{L(\mathbf{u}^\bullet)}F(\mathbf{u}^\bullet)\\
        &=\E\left[\frac{1}{L_n^\bullet}\right]^{-1}\E\left[\frac{F(U_n^\bullet)}{L_n^\bullet}\right].
    \end{align*}
It follows that 
\begin{equation}
    \E[F(Q_n^{(b)})]=\E[F(\mathrm{CVS}(U_n))]=\E\left[\frac{1}{L_n^\bullet}\right]^{-1}\E\left[\frac{F(\mathrm{CVS}(U_n^\bullet))}{L_n^\bullet}\right].
\end{equation}
    By Proposition \ref{CVS collage}, this is also equal to 
    \[\E\left[\frac{1}{L_n^\bullet}\right]^{-1}\E\left[\frac{F(\mathrm{Glue}(Qd_n))}{L_n^\bullet}\right].\]
    Finally, we can decompose the expectation according to the length of the spine and the areas of the faces. By Remark \ref{correspondance}, this gives 
     \[\E[F(Q_n^{(b)})]=\E\left[\frac{1}{L_n^\bullet}\right]^{-1}\E\left[\frac{F\left(\mathrm{Glue}\left(Qd_{A_n^\bullet,\overline{A}_n^\bullet,L_n^\bullet}\right)\right)}{L_n^\bullet}\right],\]
     which concludes the proof. 
\end{proof}

\begin{remark}
    As we will see in Section \ref{continuous quad}, the scaling limit of quadrilaterals with geodesic sides is already established. Therefore, by Corollary \ref{lien}, in order to obtain the scaling limit of delayed quadrangulations, one needs to understand the behavior of the random variables $(L_n^\bullet,A_n^\bullet,\overline{A}_n^\bullet)$.
\end{remark}

\subsection{Study of well-labelled unicycles}\label{section unicycle}

In this section, we study the properties of a uniform random element in $\U_n$ as $n$ gets very large. To do so, we rely on enumerations formulas, which are based into canonical decomposition of well-labelled unicycles into well-known objects. Using the Miermont bijection, the results of this section will allow us to compute asymptotic properties of large biased quadrangulations.

First, let us introduce some definitions and basic results. 
Let $\U_{n,k}$ (respectively $\U_{n,k}^\bullet$) be the subset of $\U_n$ (respectively $\U_n^\bullet$) composed of unicycles with a cycle of length $k$. Observe that given an element $\mathbf{u}\in \U_{n,k}$, one can distinguish an edge on the cycle and then forget the root-edge, which yields an element of $\U_{n,k}^\bullet$. Since there are $k$ possible choices for the edge to distinguish, and $2n$ possible root-edges to forget, this operation is $2n$-to-$2k$. This gives the simple but useful identity :
    \begin{equation}\label{count cycles}
        k\times\#\U_{n,k}=2n\times\#\U^\bullet_{n,k},
    \end{equation}
    which leads to 
    \begin{equation}\label{decompocycle}
        \#\U_n=\sum_{k=1}^n\#\U_{n,k}=\sum_{k=1}^n\frac{2n}{k}\#\U_{n,k}^\bullet,\quad\text{and}\quad\#\U_n^\bullet=\sum_{k=1}^n\#\U_{n,k}^\bullet.
    \end{equation}

Our first result gives asymptotic formulas for $\#\U_n$ and $\#\U_n^\bullet$.

\begin{theorem}\label{counting}
   We have 
    \[\#\U_n\underset{n\rightarrow\infty}{\sim}\frac{\sqrt{3}\times\Gamma(1/4)\times12^n}{2\pi\times n^{1/4}}\]
    and
     \[\#\U_n^\bullet\underset{n\rightarrow\infty}{\sim}\frac{\sqrt{3}\times\Gamma(3/4)\times 12^n}{4\pi\times n^{3/4}}.\]
\end{theorem}
\begin{proof}
  First, let us compute $\#\U_{n,k}^\bullet$. Note that every element of this set can be decomposed as follows. First, the labels on the cycle (started from the root vertex and in the direction of the root edge) form a well-labeled path of length $k$ such that its endpoints have label $0$. Then, the subtrees grafted to the cycle, ordered in clockwise ordered from the root vertex and starting with the external face, form a well-labeled plane forest of $2k$ labeled trees with $n-k$ edges. Finally, shift the labels of each tree by the label of their root (see Figure \ref{decompo unicycle}). This decomposition is clearly a bijection. Therefore, we have :
    \begin{equation}\label{comptage}
        \#\U_{n,k}^\bullet=M_k^{0\rightarrow0}\times3^{n-k}\cF(2k,n-k),
    \end{equation}
    where $M_k^{0\rightarrow0}$ stands for the number of labeled paths from $0$ to $0$ of length $k$, and $\cF(a,b)$ for the number of forests with $a$ trees and $b$ edges. These quantities are well known and can be expressed in a probabilistic form. 
    \begin{figure}
\centering
\begin{subfigure}{.5\textwidth}
  \centering
  \includegraphics[width=.85\linewidth]{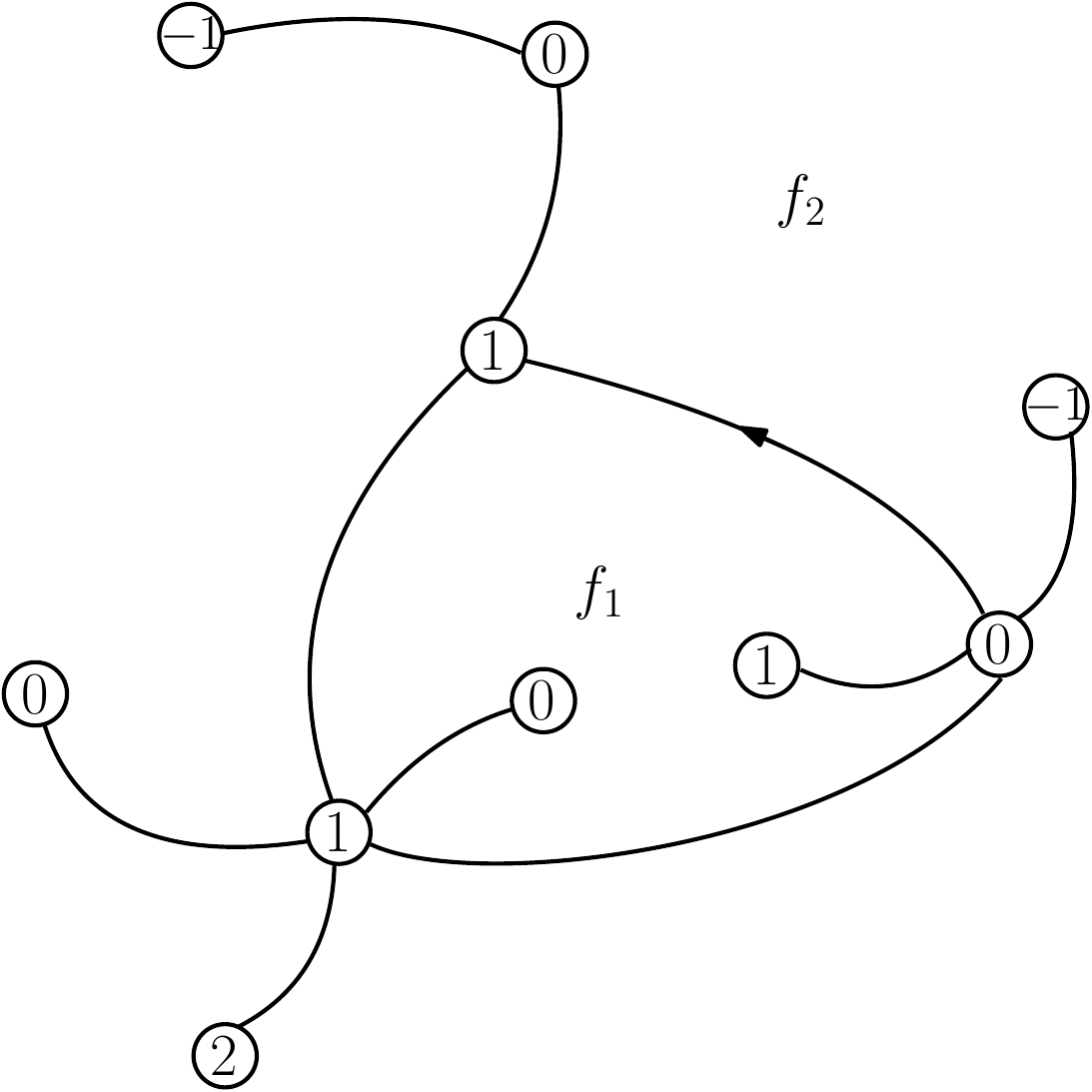}
\end{subfigure}%
\begin{subfigure}{.5\textwidth}
  \centering
  \includegraphics[width=.9\linewidth]{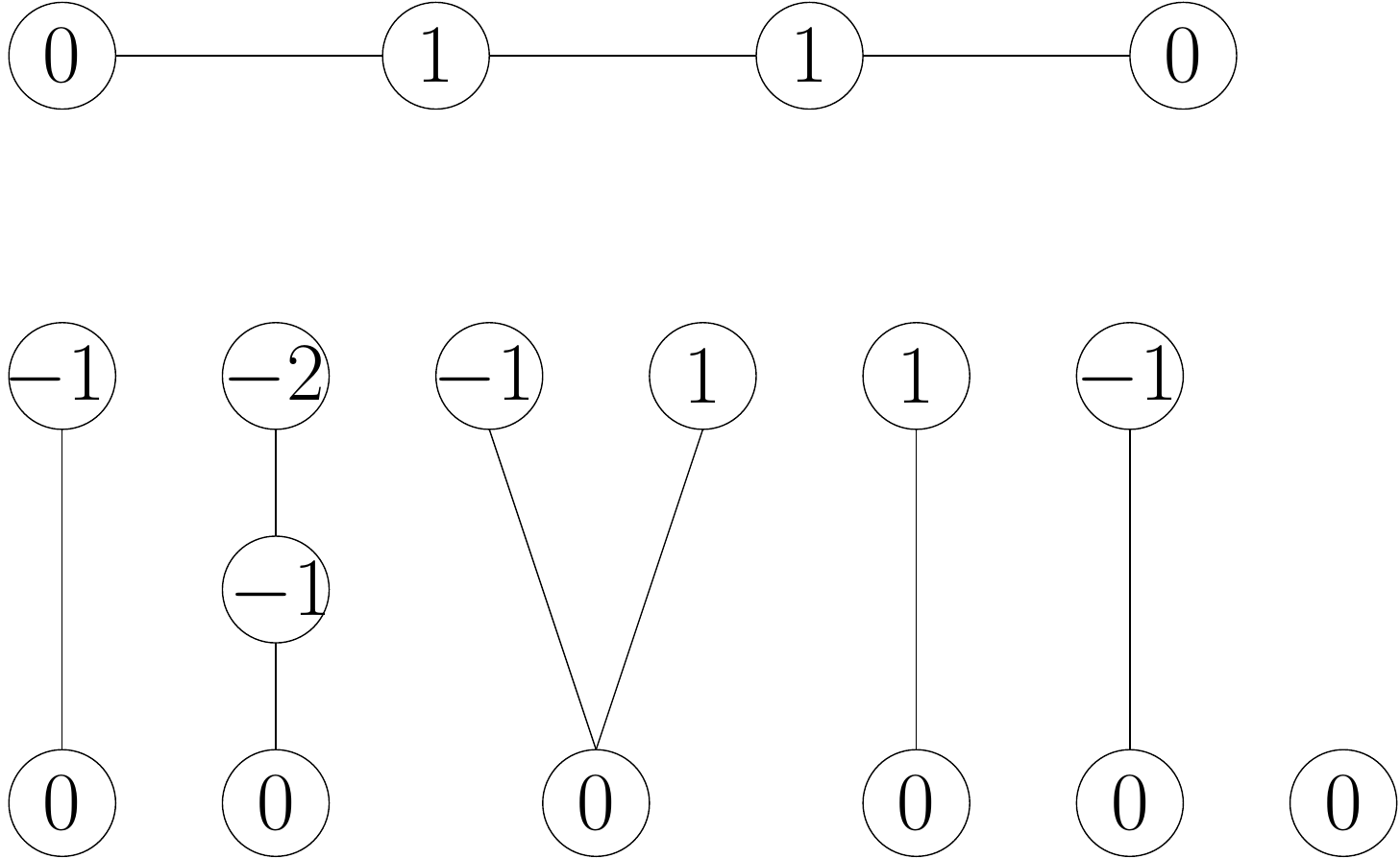}
\end{subfigure}
\caption{A unicycle $\mathbf{u}\in\U^\bullet_{10,3}$, decomposed as a path of length $3$ together with a labelled plane forest with $6$ trees.}
\label{decompo unicycle}
\end{figure}

    Thus, \eqref{comptage} can be rewritten as
    \[\#\U_{n,k}^\bullet=\frac{k}{n}\times12^n\P(M_k=0)\P(S_{2n}=-2k)\]
    where $(M_n)_{n\geq0}$ is a random walk with uniform steps in $\{-1,0,1\}$, and $(S_n)_{n\geq0}$ the simple random walk on $\Z$. 
    Hence, using \eqref{decompocycle}, we have
    \begin{equation}\label{enum1}         \#\U_n=2\times12^n\times\sum_{k=1}^n\P(M_k=0)\P(S_{2n}=-2k)
    \end{equation}  
    and 
    \begin{equation}\label{enum2}
        \#\U_n^\bullet= \frac{12^n}{n}\times\sum_{k=1}^nk\P(M_k=0)\P(S_{2n}=-2k).
    \end{equation}    
    Then, by classical local limit theorems, we know that 
    \begin{equation}\label{local limit pont}
    \sqrt{k}\P(M_k=0)\xrightarrow[k\rightarrow\infty]{}\sqrt{\frac{3}{4\pi}}
    \end{equation}
    and 
    \begin{equation}\label{local limit marche}       
    \sup_{k\in \Z}\left|\sqrt{n}\P(S_{2n}=-2k)-\frac{1}{\sqrt{\pi}}\exp\left(-\frac{k^2}{n}\right)\right|\xrightarrow[n\rightarrow\infty]{}0.
    \end{equation}
    First, let us deal with $\#\U_n$. We can rewrite \eqref{enum1} as
    \[\#\U_n=2\times12^n\times n^{-1/4}\int_0^n\P(M_{\lceil t\rceil}=0)\times n^{1/4}\times\P(S_{2n}=-2\lceil t\rceil)dt.\]
    A change of variable gives
    \[\#\U_n=2\times12^n\times n^{-1/4}\int_0^{\sqrt{n}}\P(M_{\lceil{\sqrt{n}}t\rceil}=0)\times n^{3/4}\times\P(S_{2n}=-2\lceil\sqrt{n}t\rceil)dt.\]
    Using \eqref{local limit pont} and \eqref{local limit marche}, 
an application of the dominated convergence theorem gives 
   \[\int_0^{\sqrt{n}}\P(M_{\lceil\sqrt{n}t\rceil}=0)\times n^{3/4}\times\P(S_{2n}=-2\lceil\sqrt{n}t\rceil)dt\xrightarrow[n\rightarrow\infty]{}\frac{\sqrt{3}}{2\pi}\int_0^\infty\frac{1}{\sqrt{t}}\exp(-t^2)dt.\]
  Therefore, we obtain 
    \[\#\U_n\sim\frac{\sqrt{3}\times12^n\times n^{-1/4}}{\pi}\times\int_0^\infty t^{-1/2}\exp(-t^2)dt=\frac{\sqrt{3}\times\Gamma(1/4)\times12^n}{2\pi\times n^{1/4}},\]
    which gives the desired result for $\#\U_n$. In a similar fashion, we can rewrite \eqref{enum2} as
    \begin{equation}\label{comptage pointé}
         \#\U_n^\bullet= 12^n\times n^{-3/4}\times\int_0^{\sqrt{n}}\frac{\lceil\sqrt{n}t\rceil}{\sqrt{n}}\P(M_{\lceil \sqrt{n}t\rceil}=0)\times n^{3/4}\times\P(S_{2n}=-2\lceil\sqrt{n}t\rceil)dt.
    \end{equation}
    Using the same strategy as previously, we find 
    \[\#\U_n^\bullet\sim\frac{\sqrt{3}\times12^n\times n^{-3/4}}{2\pi}\int_0^\infty t^{1/2}\exp(-t^2)dt=\frac{\sqrt{3}\times\Gamma(3/4)\times 12^n}{4\pi\times n^{3/4}},\] 
    which gives the result.
\end{proof}

Then, we prove that the random variables $\left(L_n^\bullet,A_n^\bullet,\overline{A}_n^\bullet\right)$ admit a scaling limit. In particular, Theorem \ref{counting} suggests that the length of the cycle of a typical element in $\U_n$ (or $\U_n^\bullet$) is of order $\sqrt{n}$. These results will be used in Section \ref{Section convergence} to prove the convergence of $Q_n^{(b)}$ towards a continuous object.

The following proposition identifies the scaling limit of these random variables. 

\begin{proposition}[Scaling limit of length and volume]\label{convergence taille}
   The random variables $\left(\frac{L_n^\bullet}{\sqrt{2n}},\frac{A_n^\bullet}{n}\right)$ converge in distribution as $n\rightarrow\infty$ toward a random variable $(\cL^\bullet,\cA^\bullet)$ with density 
    \[\1_{\{(x,y)\in\R_+\times [0,1]\}}\frac{2^{1/4}}{\Gamma(3/4)\sqrt{\pi}}\frac{x^{3/2}}{(y(1-y))^{3/2}}\exp\left(-\frac{x^2}{2y(1-y)}\right).\]
\end{proposition}
\begin{proof}
    Using the decomposition introduced in the proof of Theorem \ref{counting} (see Figure \ref{decompo unicycle}), for every $h,i$ such that $h+i\leq n$, we have
    \begin{align*}        \P(L_n^\bullet=h,A_n^\bullet=i)&=\frac{3^{n-k}M_h^{0\rightarrow0}\cF(h,i)\cF(h,n-k-i)}{\#U_n^\bullet}\\
    &=\frac{12^n\times h^2\P(M_h=0)\P(S_{h+2i}=-h)\P(S_{2n-2i-h}=-h)}{(h+2i)(2n-2i-h)\times\#U_n^\bullet}\\
        &=\frac{nh^2}{(h+2i)(2n-2i-h)}\times\frac{\P(M_h=0)\P(S_{h+2i}=-h)\P(S_{2n-2i-h}=-h)}{\sum_{k=1}^nk\P(M_k=0)\P(S_{2n}=-2k)}.
    \end{align*}
    Fix $x>0$ and $0\leq y < 1$. For $n$ large enough so that $\lceil x\sqrt{2n}\rceil+\lceil yn\rceil\leq n$, this gives
    \begin{align*}
    n^{3/2}\P(L_n^\bullet=\lceil x\sqrt{2n}\rceil,A_n^\bullet=\lceil yn\rceil)&=n^{1/4}\P(M_{\lceil x\sqrt{2n}\rceil}=0)\frac{n\lceil x\sqrt{2n}\rceil^2}{(2\lceil yn \rceil+\lceil x\sqrt{2n}\rceil)(2n-2\lceil yn\rceil-\lceil x\sqrt{2n}\rceil)}\\
    &\times \frac{ n^{1/2}\P(S_{\lceil x\sqrt{2n}\rceil+2\lceil yn\rceil}=-\lceil x\sqrt{2n}\rceil)n^{1/2}\P(S_{2n-2\lceil yn\rceil-\lceil x\sqrt{2n}\rceil}=-\lceil x\sqrt{2n}\rceil)}{n^{-1/4}\sum_{k=1}^nk\P(M_k=0)\P(S_{2n}=-2k)}.
    \end{align*}
Then, using local limit theorems and Theorem \ref{counting}, we obtain
\[n\sqrt{2n}\P(L_n^\bullet=\lceil x\sqrt{2n}\rceil,A_n^\bullet=\lceil yn\rceil)\xrightarrow[n\rightarrow\infty]{}\frac{2^{1/4}
}{\Gamma(3/4)\sqrt{\pi}}\left(\frac{x}{y(1-y)}\right)^{3/2}\exp\left(-\frac{x^2}{2y(1-y)}\right).\]
In order to conclude with Scheffé's lemma, we need to prove that this function defines a probability density on $\R_+\times [0,1]$. To see this, for every $x\geq0,t\geq0$, we write 
\begin{equation*}
    q_x(t)=\frac{x}{\sqrt{2\pi t^3}}\exp\left(-\frac{x^2}{2t}\right),
\end{equation*}
which is the density of the hitting time of $x$ at time $t$ by Brownian motion. Using the semigroup property of the hitting times, we have 
\[\int_0^1 \frac{1}{(y(1-y))^{3/2}}\exp\left(-\frac{x^2}{2y(1-y)}\right)dy=\frac{2\pi}{x^2}\int_0^1 q_x(y)q_x(1-y)dy=\frac{2\pi}{x^2}q_{2x}(1).\]
Therefore, we obtain 
\begin{align*}
    \int_{\R_+\times [0,1]}\frac{2^{1/4}
}{\Gamma(3/4)\sqrt{\pi}}\left(\frac{x}{y(1-y)}\right)^{3/2}\exp\left(-\frac{x^2}{2y(1-y)}\right)dxdy&=\frac{2^{7/4}}{\Gamma(3/4)}\int_{\R_+}x^{1/2}\exp\left(-2x^2\right)dx\\
    &=\frac{1}{\Gamma(3/4)}\int_{\R_+}u^{-1/4}\exp\left(-u\right)du\\
    &= 1,
\end{align*}
which concludes the proof.
\end{proof}
\begin{remark}
    By \eqref{somme}, we can deduce from Proposition \ref{convergence taille} that the triple $\left(\frac{L_n^\bullet}{\sqrt{2n}},\frac{A_n^\bullet}{n},\frac{\overline{A}_n^\bullet}{n}\right)$ converges in distribution toward $\left(\cL^\bullet,\cA^\bullet,1-\cA^\bullet\right)$.
\end{remark}

Before giving an interpretation for the random variable $(\cL^\bullet,\cA^\bullet)$, we need to introduce some definitions. For every $x>0,\,A>0$, let $B^{(x)}$ denote a standard Brownian motion starting at $0$ and stopped when it reaches $-x$, and let $B^{(x,A)}$ be a first passage bridge of standard Brownian motion from $0$ to $-x$ with duration $A$. We recall that $B^{(x,A)}$ corresponds to a Brownian motion ``conditioned to reach $-x$ for the first time at time $A$'' (see \cite[Section 2.1]{disque} for more details). These random variables are related by the following formula : for every positive measurable function $F$, 
\begin{equation}\label{joint law}
    \E\left[F(B^{(x)})G(T_{-x})\right]=\int_0^\infty G(a)\frac{x}{\sqrt{2\pi a^3}}\exp\left(-\frac{x^2}{2a}\right)\E\left[F(B^{(x,A)}\right]da.    
\end{equation}

We can interpret the random variable $(\cL^\bullet,\cA^\bullet)$ as follows. First, sample a random variable $X$ with density $\1_{\{x\geq 0\}}x^{1/2}\exp(-2x^2)$. Then given $X$, consider a first-passage bridge of Brownian motion hitting $-2X$ for the first time at time $1$, and let $T_{-X}$ be its first hitting time of $-X$. Then, $(\cL^\bullet,\cA^\bullet)\overset{(d)}{=}(X,T_{-X})$ (the proof essentially corresponds to the calculations made at the end of the proof of Proposition \ref{convergence taille}).

We conclude this section with a lemma that deals with the uniform integrability of the family of random variables $\left(\frac{\sqrt{2n}}{L_n^\bullet}\right)_{n\geq 1}$. This result will be used later, in the proof of Theorem \ref{Conv collage}.

\begin{lemme}[Uniform integrability of the inverse of the lengths]\label{uniform integrability}
    The family of random variables $\left(\frac{\sqrt{2n}}{L_n^\bullet}\right)_{n\geq 1}$ is uniformly integrable.
\end{lemme}
\begin{proof}
    For every $n\geq1$ and $M>0$, we have 
    \begin{align}\label{explicit expression}
        \E\left[\frac{\sqrt{2n}}{L_n^\bullet}\1_{\frac{\sqrt{2n}}{L_n^\bullet}>M}\right]&=\E\left[\frac{\sqrt{2n}}{L_n^\bullet}\1_{L_n^\bullet<\frac{\sqrt{2n}}{M}}\right]\nonumber\\
        &=\frac{\sqrt{2n}}{\#\U_n^\bullet}\sum_{\substack{u^\bullet\in\U_n^\bullet,\nonumber\\ L(u^\bullet)<\sqrt{2n}/M}}\frac{1}{L(u^\bullet)}\nonumber\\
        &=\frac{\sqrt{2n}}{\#\U_n^\bullet}\sum_{\substack{u\in\U_n,\\ L(u)<\sqrt{2n}/M}}\frac{1}{2n}\nonumber\\
&=\frac{1}{\sqrt{2n}\times\#\U_n^\bullet}\sum_{k=1}^{\lfloor\sqrt{2n}/M\rfloor}\#\U_{n,k}.
    \end{align}
    Then, a mere adaptation of the proof of Theorem \ref{counting} shows that
    \[\sum_{k=1}^{\lfloor\sqrt{2n}/M\rfloor}\#\U_{n,k}=2\times12^n\times n^{-1/4}\int_0^{\frac{\lfloor\sqrt{2n}/M\rfloor}{\sqrt{n}}}\P(M_{\lceil\sqrt{n}t\rceil}=0)\times n^{3/4}\times\P(S_{2n}=-2\lceil\sqrt{n}t\rceil)dt.\]
    Therefore, using the expression \eqref{comptage pointé}, we obtain 
\begin{equation*}
     \E\left[\frac{\sqrt{2n}}{L_n^\bullet}\1_{\frac{\sqrt{2n}}{L_n^\bullet}>M}\right]=\frac{\int_0^{\frac{\lfloor\sqrt{2n}/M\rfloor}{\sqrt{n}}}\P(M_{\lceil\sqrt{n}t\rceil}=0)\times n^{3/4}\times\P(S_{2n}=-2\lceil\sqrt{n}t\rceil)dt}{\int_0^{\sqrt{n}}\frac{\lceil\sqrt{n}t\rceil}{\sqrt{n}}\P(M_{\lceil \sqrt{n}t\rceil}=0)\times n^{3/4}\times\P(S_{2n}=-2\lceil\sqrt{n}t\rceil)dt}.
\end{equation*}
    Observe that, as we saw in the proof of Theorem \ref{counting}, the denominator converges towards a positive quantity. Consequently, there exists some constant $C>0$ such that for every $n\in\N$ and $M>0$, 
    \begin{equation*}
        \E\left[\frac{\sqrt{2n}}{L_n^\bullet}\1_{\frac{\sqrt{2n}}{L_n^\bullet}>M}\right]\leq C\int_0^{\frac{\lfloor\sqrt{2n}/M\rfloor}{\sqrt{n}}}\P(M_{\lceil\sqrt{n}t\rceil}=0)\times n^{3/4}\times\P(S_{2n}=-2\lceil\sqrt{n}t\rceil)dt.
    \end{equation*}
    Then, by \eqref{local limit pont} and \eqref{local limit marche}, there exists $c>0$ such that for every $n\in\N$ and $t\geq0$, 
    \begin{equation*}
      n^{1/4} \P(M_{\lceil\sqrt{n}t\rceil}=0)\leq\sqrt{\frac{3}{4\pi t}}+c\quad\text{and}\quad n^{1/2} \P(S_{2n}=-2\lceil\sqrt{n}t\rceil)\leq\frac{1}{\sqrt{\pi}}\exp\left(-t^2\right)+c.
    \end{equation*}
  Therefore, we have
    \begin{equation}\label{upper bound}
        \E\left[\frac{\sqrt{2n}}{L_n^\bullet}\1_{\frac{\sqrt{2n}}{L_n^\bullet}>M}\right]\leq C\int_0^{\frac{\sqrt{2}}{M}}\left(\sqrt{\frac{3}{4\pi t}}+c\right)\left(\frac{1}{\sqrt{\pi}}\exp\left(-t^2\right)+c\right)dt.
    \end{equation}
    Note that the integral on the right-hand side tends to $0$ as $M$ goes to infinity. Therefore, for any $\delta>0$, using \eqref{upper bound}, we can find $M>0$ large enough such that for every $n\in\N$, 
    \begin{equation*}
       \E\left[\frac{\sqrt{2n}}{L_n^\bullet}\1_{\frac{\sqrt{2n}}{L_n^\bullet}>M}\right]\leq \delta,
    \end{equation*}
   which concludes the proof.
\end{proof}

\section{The biased Brownian sphere}\label{Section convergence}

In this section, we express by two different means the scaling limit of bi-pointed delayed quadrangulations, that we call the biased Brownian sphere. The first approach essentially only relies on the convergence of uniform quadrangulations toward the Brownian sphere, but does not give an explicit construction of the limiting object. However, this approach makes a connection between the Brownian sphere and the biased Brownian sphere. On the other hand, the second approach is based on the result of Section \ref{section unicycle}, together with \cite[Theorem 14]{Compactbrowniansurfaces} which establishes the scaling limit of random quadrilaterals with geodesic sides. In this case, we can derive an explicit construction of the limiting object.

\subsection{Some preliminaries}

We begin with a number of preliminaries about the different objects that will be used throughout this section.

\subsubsection{The Gromov-Hausdorff-Prokhorov topology}

A \textbf{metric measure space} is a triple $(X,d_X,\mu_X)$, where $(X,d_X)$ is a compact metric space, and $\mu_X$ is a finite Borel measure on $X$. If $(X,d_X,\mu_X)$ and $(Y,d_Y,\mu_Y)$ are two metric measure spaces, the \textbf{Gromov-Hausdorff Prokhorov metric} (GHP metric for short) is defined by 
\begin{equation}\label{GHP}
    d_{\mathrm{GHP}}\left((X,d_X,\mu_X),(Y,d_Y,\mu_Y)\right)=\inf_{\substack{\phi:X\rightarrow Z\\\psi:Y\rightarrow Z}}\{d_Z^{\mathrm{H}}(\phi(X),\psi(Y))\vee d_Z^\mathrm{P}(\phi_*\mu_X,\psi_*\mu_Y)\},
\end{equation}
where the infimum is over all compact metric spaces $(Z,d_Z)$ and all isometries $\phi,\psi$ from $X,Y$ to $Z$, and where $d_Z^\mathrm{H}$ and $d_Z^\mathrm{P}$ are respectively the Hausdorff and Prokhorov metrics, defined as follows. 
For every $\varepsilon>0$ and every closed subset $A\subseteq Z$, set 
\[A^\varepsilon:=\left\{z\in Z\,:\,d_Z(z,A)<\varepsilon\right\}.\]
Then, for every compact subsets $A,B\subseteq Z$, the Hausdorff metric is 
\[d_Z^\mathrm{H}(A,B)=\inf\{\varepsilon>0\,:\,A\subseteq B^\varepsilon\text{ and }B\subseteq A^\varepsilon\},\]
and for every finite Borel measures $\mu,\nu$ on $Z$, the Prokhorov metric is 
\[d_Z^\mathrm{P}(\mu,\nu)=\inf\{\varepsilon>0\,:\text{ for all closed }A\subseteq Z,\,\mu(A)\leq\nu(A^\varepsilon)+\varepsilon\text{ and }\nu(A)\leq\mu(A^\varepsilon)+\varepsilon\}.\]
If $d_{\mathrm{GHP}}\left((X,d_X,\mu_X),(Y,d_Y,\mu_Y)\right)=0$, we say that $(X,d_X,\mu_X)$ and $Y,d_Y,\mu_Y)$ are isometry-equivalent. The formula \eqref{GHP} defines a metric on the set $\K$ of isometry equivalence classes of metric measure spaces. Moreover, equipped with this distance, $\K$ is a complete and separable metric space. 
More generally, for $\ell>0$, we can consider a $\ell$-marked measure metric space $(X,d_X,\A,\mu_X)$, where 
\begin{itemize}
    \item $(X,d_X,\mu_X)$ is a metric measure space,
    \item $\A$ is a $\ell$-tuple of non-empty compact subsets of $X$, called marks.
\end{itemize}
Then, as previously, we can define the $\ell$-marked Gromov-Hausdorff-Prokhorov metric by
\begin{align*}
    d_{\mathrm{GHP}}^{(\ell)}\left((X,d_X,\A,\mu_X),(Y,d_Y,\B,\mu_Y)\right)&=\inf_{\substack{\phi:X\rightarrow Z\\\psi:Y\rightarrow Z}}\left\{d_Z^{\mathrm{H}}(\phi(X),\psi(Y))\vee \max_{1\leq i\leq\ell}d_Z^\mathrm{H}(\phi(A_i),\psi(B_i))\vee d_Z^\mathrm{P}(\phi_*\mu_X,\psi_*\mu_Y)\right\}.
\end{align*}
This formula also defines a metric on the set $\K^{(\ell)}$ of isometry-equivalence classes of $\ell$-marked measured metric space. 
In this paper, we will be interested in two specific cases. The first one is when $\ell=2$ and the two marks are singletons. This subset of $\K^{(2)}$ will be denoted by $\K^{\bullet\bullet}$. The other one is where $(X,d_X)$ is a geodesic space, $\ell=4$, and $\A$ is a $4$-tuple $(\gamma_1,\gamma_2,\gamma_3,\gamma_4)$ of geodesics, and where 
\begin{equation*}
\mathrm{length}(\gamma_1)=\mathrm{length}(\gamma_2),\quad\mathrm{length}(\gamma_3)=\mathrm{length}(\gamma_4). 
\end{equation*}
Let $\K^{(4)}_{\mathrm{geo}}$ be the isometry-equivalence classes of such spaces. For a space $(X,d_X,\gamma_1,\gamma_2,\gamma_3,\gamma_4)\in\K^{(4)}_c$, we can define the gluing of $X$ along its marks as the compact metric space obtained by identifying $\gamma_1$ with $\gamma_2$, and $\gamma_3$ with $\gamma_4$, equipped with the intrinsic distance. We denote by $\mathrm{Glue}(X,d_X,\gamma_1,\gamma_2,\gamma_3,\gamma_4)$ this metric space. The following proposition is a particular case of \cite[Proposition 22]{Compactbrowniansurfaces}
\begin{proposition}\label{gluing geodesics}
    Let $(X_n,d_{X_n},\A_n,\mu_{X_n}),\,n\geq0$ be a sequence of $\K_{\mathrm{geo}}^{[4)}$ that converges toward $(X,d_X,\A,\mu_X)$ for the $4$-marked GHP topology. Then, $(X,d_X,\A,\mu_X)\in\K_{\mathrm{geo}}^{(4)}$, and 
    \[\mathrm{Glue}(X_n,d_{X_n},\A_n,\mu_{X_n})\xrightarrow[n\rightarrow\infty]{}\mathrm{Glue}(X,d_X,\A,\mu_X)\]
    for the GHP-topology. In other words, $\K_{\mathrm{geo}}^{(4)}$ is closed, and the application $\mathrm{Glue}:\K_{\mathrm{geo}}^{(4)}\rightarrow\K^{(2)}$ is continuous.
\end{proposition}

\subsubsection{The Brownian sphere}\label{brownian sphere}

In this section, we give the definition of the Brownian sphere, and recall some of its basic properties.

We start with some deterministic considerations. For every continuous function $f$ defined on a compact interval $[a,b]$ and $s,t\in [a,b]$, we define 
\begin{equation*}
    m_f(s,t)=\left\{
    \begin{array}{ll}
   \inf_{r\in[s,t]}f(r)\quad&\text{ if $s\leq t$} \\
     \inf_{r\in[a,t]\cup[s,b]}f(r)\quad&\text{ if $s\geq t$,}
    \end{array}\right.
\end{equation*}
and set 
\begin{equation}\label{distance arbre}
    d_f(s,t)=f(s)+f(t)-2m_f(s,t).
\end{equation}

Consider a pair $(f,g)$ of continuous function satisfying
\begin{itemize}[label=\textbullet]
    \item Both functions are defined on a common compact interval $[0,a]$ for some $a>0$,
    \item We have $f(0)=f(a)=0$ and $f\geq 0$, 
    \item For every $s,t\in [0,a]$, 
    \begin{equation}\label{Compability}
        d_f(s,t)=0\Longrightarrow d_g(s,t)=0.
    \end{equation}
\end{itemize}
We can associate a metric space to the pair $(f,g)$ as follows. We define
\[D_{f,g}=d_g/\{d_f=0\}.\]

More precisely, for every $s,t\in[0,a]$, we have 
 \begin{equation*}
    {D}_{f,g}(s,t)=\inf_{s_0,t_0,...,s_n,t_n}\sum_{i=0}^n d_g(s_i,t_i),
 \end{equation*}
 where the infimum is taken over all $n\in\N$ and sequences $s_0,t_0,...,s_n,t_n\in I$ such that $s_0=s$, $t_n=t$ and $d_f(t_i,s_{i+1})=0$ for every $1\leq i\leq n-1$. Observe that $D_{f,g}\leq d_g$. On the other hand, we have $D^\circ_{f,g}(s,t)\geq|g(s)-g(t)|$, which gives a simple but powerful bound 
\begin{equation}\label{bound}
    D_{f,g}(s,t)\geq|g(s)-g(t)|.
\end{equation}

Let $\left(S_{f,g},{D}_{f,g}\right)$ denote the quotient space $[0,a]/\{{D}_{f,g}=0\}$ equipped with the distance induced by ${D}_{f,g}$. The metric space 
 \[S_{f,g}=\left(\widehat{M}_{f,g},\widehat{D}_{f,g}\right)\] 
 comes with a canonical projection $p_{f,g}:[0,a]\rightarrow{S}_{f,g}$, and a volume measure $\mu_{f,g}$, which is the push forward of the Lebesgue measure on $[0,a]$ by the canonical projection.

Let us randomize the functions $(f,g)$. We define the snake driven by a deterministic function $f$ as a random centered Gaussian process $(Z_t^f,t\in I)$ with $Z_0^f=0$ and with covariances specified by
\begin{equation}\label{covariance}
    \E\left[\left(Z_t^f-Z_s^f\right)^2\right]=d_f(s,t),\quad s,t\in I.
\end{equation}
We will consider a continuous modification of $Z$, which always exists as soon as $f$ is Hölder continuous. Moreover, \eqref{covariance} implies that $Z_s^f=Z_t^f$ as soon as $d_f(s,t)=0$, so that $(f,Z^f)$ satisfies the condition \eqref{Compability}. 

\begin{definition}
    The \textbf{standard Brownian sphere} is the random metric space $\cS=[0,1]/\{D=0\}$ obtained from  the pair $(\e,Z)$, where $\e$ is a normalized Brownian excursion, and given $\e$, $Z$ is the random snake driven by $\e$. It is equipped with the distance induced by $D$ and a volume measure $\mu$, which is the pushforward of the volume measure on $\cT$ by the canonical projection $\mathbf{p}:[0,1]\rightarrow\cS$.
\end{definition}

This random space appears as the scaling limit of several models of random maps. In particular, it was proved in \cite{uniqueness,convergence} that the Brownian sphere is the scaling limit of uniform quadrangulations with $n$ faces when $n$ goes to infinity, for the Gromov-Hausdorff-Prokhorov topology.

Note that if $s,t\in[0,1]$ are such that $D(s,t)=0$, then $Z_s=Z_t$. Therefore, for every $x\in\cS$ and $t\in[0,1]$ such that $\mathbf{p}(t)=x$, we set $Z_x:=Z_t$. 
Almost surely, there exists a unique $s_*\in[0,1]$ such that $Z_{s_*}=\inf_{s\in[0,1]}Z_s$ (see \cite[Proposition 2.5]{Conditionnedbrowniantrees}). We set $x_*:=\mathbf{p}(s_*)$, and $Z_*:=Z_{s_*}$. We also let $x_0:=\mathbf{p}(0)$. Using the bound \eqref{bound} together with the inequality $D\leq D^\circ$, one can see that for every $x\in\cS$, 
\begin{equation}\label{distance a un point}
    D(x,x_*)=Z_x-Z_*.
\end{equation}

We also define the Brownian sphere of volume $a>0$, which is obtained by replacing the normalized Brownian excursion by a Brownian excursion of duration $a$. Similarly, we define the free Brownian sphere which is obtained by replacing the normalized Brownian excursion by Itô measure of positive excursions of linear Brownian motion, normalized so that the density of the duration is $t\mapsto\left(2\sqrt{2\pi t^3}\right)^{-1}$. Although it is not a random variable, it is often more convenient to work with the free Brownian sphere in several situations.

To conclude, let us mention some properties of the Brownian sphere. First, the following proposition, proved in \cite{TopologicalStructure}, characterizes the elements of $\cT$ that are identified in $\cS$. 

\begin{proposition}\label{identification}
    The following holds almost surely. For every $s,t\in[0,1]$ with $s\neq t$ such that $D(s,t)=0$, we have 
    \[d_\e(s,t)=0\quad\text{ or }\quad d_Z(s,t)=0.\]
    Moreover, the two equalities cannot hold simultaneously.
\end{proposition}

Finally, it was proved in \cite{geodesic1} that almost surely, there exists a unique geodesic $\Gamma$ between $x_0$ and $x_*$ (see \cite[Section 7.2]{tessalations} for another proof). We mention that much more is known about geodesics in the Brownian sphere, but we do not need these results in this work.

\subsubsection{Continuous quadrilaterals}\label{continuous quad}

Here, we define the continuous counterpart of quadrilaterals with geodesic sides. We start with some deterministic considerations. 

We say that a pair $(f,g)$ of continuous functions is a quadrilateral trajectory if :
\begin{itemize}[label=\textbullet]
    \item Both functions are defined on a common compact interval $I$ that contains $0$ in its interior,
    \item We have $f(\inf {I})=f(\sup {I})=\inf_I f$, 
    \item For every $s,t\in I$, 
    \begin{equation}
        d_f(s,t)=0\Longrightarrow d_g(s,t)=0.
    \end{equation}
\end{itemize}
Let $I_+=I\cap\R_+$ and $I_-=I\cap\R_-$. For every quadrilateral trajectory $(f,g)$ and $s,t\in I$, we define 
\begin{equation}
    \widehat{d}_g(s,t)=\left\{
    \begin{array}{ll}
  d_g(s,t)\quad&\text{ if $s,t\in I_+$ or $s,t\in I_-$,} \\   
     +\infty\quad&\text{ otherwise}
    \end{array}\right.
\end{equation}
and 
\[\widehat{D}_{f,g}=\widehat{d}_g/\{d_f=0\}.\]

 Let $\left(\widehat{M}_{f,g},\widehat{D}_{f,g}\right)$ denote the quotient space $I/\{\widehat{D}_{f,g}=0\}$ equipped with the distance induced by $\widehat{D}_{f,g}$. The metric space 
 \[\mathrm{Qd}_{f,g}=\left(\widehat{M}_{f,g},\widehat{D}_{f,g}\right)\] 
 is called the continuous quadrilateral coded by $(f,g)$, and we let $p_{f,g}:I\rightarrow\widehat{M}_{f,g}$ denote the canonical projection.

 The metric space $\mathrm{Qd}_{f,g}$ comes with two distinguished points and four distinguished geodesics, defined as follows. 
 First, let $s_*\in I_-$ and $\overline{s}_*\in I_+$ be such that
 \[g(s_*)=\inf_{x\in I_-}g(x)\quad\text{ and }\quad g(\overline{s}_*)=\inf_{x\in I_+}g(x),\]
 and set
 \[x_*=p_{f,g}(s_*),\quad \overline{x}_*=p_{f,g}(\overline{s}_*)\]
 (note that this does not depend on the choice of $s_*$ and $\overline{s}_*$). For every $t\in I\backslash\{0\}$, we set $I_t=I_-$ if $t<0$ and $I_t=I_+$ if $t>0$. Then, we define 
 \[\Gamma_t(r)=\inf\{s\geq t\,:\,g(s)=g(t)-r\}\quad\text{ for $r\in\R_+$ such that }\inf_{\substack{s\geq t\\s\in I_t}}g(s)\leq g(t)-r,\]
 \[\Xi_t(r)=\sup\{s\leq t\,:\,g(s)=g(t)-r\}\quad\text{ for $r\in\R_+$ such that }\inf_{\substack{s\leq t\\s\in I_t}}g(s)\leq g(t)-r.\]
 We also extend these definitions for $t=0$, by using $I_0=I_+$ in the definition of $\Gamma_0$, and $I_0=I_-$ in the definition of $\Xi_0$. Then, we define the simple geodesics in $\mathrm{Qd}_{f,g}$ by
 \[\gamma_t(r)=p_{f,g}(\Gamma_t(r)),\quad 0\leq r\leq g(t)-m_g(t,\sup I_t),\]
 \[\xi_t(r)=p_{f,g}(\Xi_t(r)),\quad 0\leq r\leq g(t)-m_g(\inf I_t,t),\] 
 where as above, we use $I_0=I_-$ in the definition of $\gamma_0$ and $I_0=I_+$ in the definition of $\Xi_0$. In particular, the geodesics $\gamma=\gamma_0$ and $\overline{\gamma}=\gamma_{\inf I_-}$ are called the maximal geodesics, whereas $\xi=\xi_{\sup I_+}$ and $\overline{\xi}=\xi_0$ are called the shuttles. Observe that $\gamma$ and $\xi$ (resp. $\overline{\gamma}$ and $\overline{\xi}$) are disjoint except at their endpoint, which is $x_*$ (resp. $\overline{x}_*$).

Let us now randomize the quadrilateral trajectory $(f,g)$. For every function $f$ defined on an interval containing $0$, we set $m'_f(t)=m_f(0\wedge t,0\vee t)$.
For every $A,\overline{A},H\in(0,\infty)$, let $\textbf{Quad}_{A,\overline{A},H}$ be the probability distribution under which 
\begin{itemize}[label=\textbullet]
    \item $\left(X_t,\,0\leq t\leq A\right)$ and $\left(X_{-t},\,0\leq t\leq \overline{A}\right)$ are two independent first-passage bridges of standard Brownian motion from $0$ to $-H$ (as defined in \eqref{joint law}), with respective durations $A$ and $\overline{A}$, 
    \item given $X$, the process $W$ has the law of $\left(Z_t+\zeta_{-m'_X(t)},\,-\overline{A}\leq t\leq A\right)$, where $Z$ is the random snake driven by $X-m'_X$, and $\zeta$ is an independent Brownian bridge of duration $H$ from $0$ to $0$.
\end{itemize}

With these definitions, $\textbf{Quad}_{A,\overline{A},H}$ is carried by quadrilateral trajectories on the interval $[-\overline{A},A]$. This allows us to define the quadrilateral with half-areas $A,\overline{A}$, width $H$ and tilt $0$, denoted by $\cQ_{A,\overline{A},H}$, as the random metric space $\mathrm{Qd}_{X,W}$ under $\textbf{Quad}_{A,\overline{A},H}$. Finally, we equip the random space $\cQ_{A,\overline{A},H}$ with a measure $\overline{\mu}_{\cQ_{A,\overline{A},H}}$, which is the pushforward of the Lebesgue measure on $[-\overline{A},A]$ by the canonical projection. These random spaces correspond to a particular case of the quadrilateral with geodesic sides studied in \cite{Compactbrowniansurfaces}.

As mentioned previously, the random variables $\cQ_{A,\overline{A},H}$ appear as the scaling limit of quadrilaterals with geodesic sides, in the following way.
Consider four sequences $(a_n),\,(\overline{a}_n)\in\left(\Z_+\right)^\N,$  $(h_n)\in\N^\N$ and $(\delta_n)\in\Z^\N$ such that when $n\rightarrow\infty$, we have the convergences 
\[\frac{a_n}{n}\rightarrow A>0,\quad\frac{\overline{a}_n}{n}\rightarrow\overline{A}>0,\quad\frac{h_n}{\sqrt{2n}}\rightarrow H>0,\quad\left(\frac{9}{8n}\right)^{1/4}\delta_n\rightarrow 0.\] For every $X=(X,D,\mu,a,b,c,d)\in \K_{\mathrm{geo}}^{(4)}$, we set 
\[\overline{\Omega}_n(X)=\left(X,\left(\frac{8}{9n}\right)^{1/4}D,\frac{1}{n}\mu,a,b,c,d\right).\]
 The following theorem was established in \cite{Compactbrowniansurfaces}. 
\begin{theorem}\label{conv quad}
    Let $Qd_n$ be uniformly distributed among quadrilaterals with half-areas $a_n$ and $\overline{a}_n$, width $h_n$ and tilt ${\delta}_n$. We have the convergence
    \[\overline{\Omega}_n(Qd_n)\xrightarrow[n\rightarrow\infty]{(d)}\cQ_{A,\overline{A},H}\] in distribution in $\K^{(4)}_{\mathrm{geo}}$.
\end{theorem}

\subsection{The scaling limit of delayed quadrangulations}

In this section, we obtain by two different means the scaling limit of random delayed quadrangulations. The first approach does not give an explicit construction of the limiting space, but tells us how it is related to the Brownian sphere. On the other hand, the second approach gives more explicit information on how to construct is the limiting space, but it is not clear that it is related to the Brownian sphere.

\subsubsection{First approach}

Let us begin with the first approach, which only relies on the convergence of uniform quadrangulations towards the Brownian sphere. First, let $(\cS,D,\mu,x,y)$ be a random variable distributed as a standard Brownian sphere $(\cS,D,\mu)$ with two distinguished points sampled according to the measure $\mu$ (see \cite[Section 6.5]{tessalations} for a detailed discussion about randomly marked spaces). Recall that, by the invariance under rerooting of the Brownian sphere, one may take $x=x_0$ and $y=x_*$. We define the biased Brownian sphere $(\cS_{b},D,x,y,\mu,\Delta)$ by the following formula, valid for every non-negative measurable function $f$ :

\begin{equation}\label{def biais sphere}    \E[f(\cS_{b},D,x,y,\mu,\Delta)]=\frac{1}{2\E[D(x,y)]}\E\left[\int_{-D(x,y)}^{D(x,y)}f(\cS,D,x,y,\mu,t)dt\right].
\end{equation}
This formula is well-defined, since by \cite[Corollary 2]{DelmasMoments}, we have 
   \[\E[D(x,y)]=3\frac{2^{1/4}\Gamma(5/4)}{\sqrt{\pi}}.\]
Note that the marginal $(\cS_{b},x,y)$ has the law of the standard Brownian sphere biased by the distance between $x$ and $y$, and that conditionally on $(\cS_{b},x,y)$, the random variable $\Delta$ is uniform on $[-D(x,y),D(x,y)]$. 
   
   The following proposition is just a consequence of the convergence of uniform quadrangulations toward the Brownian sphere proved in \cite{uniqueness,convergence}. For every $n\geq 1$, let $\left(Q^{(b)}_n,d_{Q_n^{(b)}},x^{(b)}_n,y^{(b)}_n,\mu^{(b)}_n,\Delta_n\right)$ be a uniform variable in the set $\Q_n^{(b)}$ of delayed bi-pointed quadrangulation with $n$ faces, equipped with the counting measure on its vertices.
\begin{proposition}[Convergence towards the biased Brownian sphere]\label{convergence biais}
    We have :
    \[\left(Q^{(b)}_n,\left(\frac{9}{8n}\right)^{1/4} d_{Q^{(b)}_n},x^{(b)}_n,y^{(b)}_n,\frac{1}{n}\mu^{(b)}_n,\left(\frac{9}{8n}\right)^{1/4}\Delta_n\right)\xrightarrow[n\rightarrow\infty]{(d)}\left(\cS_{b},D,x,y,\mu,\Delta\right),\] where the convergence is in $\K^{\bullet\bullet}\times\R$.
\end{proposition}
\begin{proof}
By definition, for every non-negative measurable function $F$, we have 
\[\E\left[F(Q^{(b)}_n,x^{(b)}_n,y^{(b)}_n,\mu_n^{(b)},\Delta_n)\right]=\frac{\E\left[\sum_{\delta_n\in A(Q_n,x_n,y_n)}F(Q_n,x_n,y_n,\mu_n,\delta_n)\right]}{\E\left[\left(d_{Q_n}(x_n,y_n)-1\right)_+\right]},\]
where $(Q_n,x_n,y_n)$ is uniformly distributed on the set of bi-pointed quadrangulations with $n$ faces. for every bi-pointed quadrangulation $(Q,x,y)$, $A(Q,x,y)$ is the set of admissible delays of $(Q,x,y)$. In particular, the random variable $(Q^{(b)}_n,x^{(b)}_n,y^{(b)}_n)$ has the law of $(Q_n,x_n,y_n)$ biased by the Radon-Nikodym derivative 
\[G_n=\frac{\left(d_{Q_n}(x_n,y_n)-1\right)_+}{\E\left[\left(d_{Q_n}(x_n,y_n)-1\right)_+\right]}.\]
Then, by \cite[Corollary 3]{Randomplanarlattices}, the sequence $(G_n)_{n\geq 1}$ is uniformly integrable and converges in distribution (jointly with the uniform quadrangulation with the rescaled distance) toward 
\begin{equation}\label{radon}    
\frac{D(x,y)}{\E[D(x,y)]}.
\end{equation}
    Therefore, using the convergence of uniform quadrangulation toward the Brownian sphere \cite{uniqueness,convergence}, we get that $\left(Q_n^{(b)},\left(\frac{8}{9n}\right)^{1/4}d_n^{(b)},x_n^{(b)},y_n^{(b)},\frac{1}{n}\mu_n^{(b)}\right)$ converges in distribution toward the law of $(S,D,x,y,\mu)$ biased by the Radon-Nikodym derivative \eqref{radon}. Finally, recall that conditionally on $(Q^{(b)}_n,x^{(b)}_n,y^{(b)}_n)$, the delay $\Delta_n$ is uniform in the set $\{k\in\Z\,:\,|k|<d_n(x_n^{(b)},y_n^{(b)})\text{ and }k+d_n(x_n^{(b)},y_n^{(b)})\in2\Z\}$. The desired convergence easily follows. 
\end{proof}

\subsubsection{Second approach}

Let us now prove by other means that the random variables $Q_n^{(b)}$ admit a scaling limit. This other approach relies on several results proved in Section \ref{section discrete}. Together with Proposition \ref{convergence biais}, this will give an explicit construction of the standard biased Brownian sphere $\cS_{b}$, that will be detailed in Section \ref{construction}.

First, we need to introduce a pair of random variables $(\cL,\cA)$, taking values in $\R_+\times[0,1]$, with density
\begin{equation}\label{densities}
   \1_{\{(x,y)\in\R_+\times [0,1]\}}  \frac{1}{2^{1/4}\Gamma(1/4)\sqrt{\pi}}\frac{x^{1/2}}{(y(1-y))^{3/2}}\exp\left(-\frac{x^2}{2y(1-y)}\right)
\end{equation}
(one can check that this defines a probability density just like we did at the end of the proof of Proposition \ref{convergence taille}). Observe that these random variables are related to $\left(\cL^\bullet,\cA^\bullet\right)$, introduced in Proposition \ref{convergence taille}, by the following formula :
\begin{equation*}
    \E\left[F(\cL,\cA)\right]=  \E\left[\frac{1}{\cL^\bullet}\right]^{-1} \E\left[\frac{F(\cL^\bullet,\cA^\bullet)}{\cL^\bullet}\right].
\end{equation*}
Finally, observe that for any $a,b,c>0$, the random metric measured space $\mathrm{Glue}(\cQ_{a,b,c})$ is naturally equipped with two distinguished points $x_*$ and $\overline{x}_*$ (which were the distinguished points of $\cQ_{a,b,c}$ before the gluing), together with a real parameter $W_{\overline{x}_*}-W_{x_*}$. Therefore, in what follows, we will consider $\mathrm{Glue}(\cQ_{a,b,c})$ as an element of $\K^{\bullet\bullet}\times \R$.
\begin{proposition}[Scaling limit of biased quadrangulations]\label{Conv collage}
    We have :
     \begin{equation*}
         \left(Q^{(b)}_n,\left(\frac{9}{8n}\right)^{1/4} d_{Q^{(b)}_n},x^{(b)}_n,y^{(b)}_n,\frac{1}{n}\mu^{(b)}_n,\left(\frac{9}{8n}\right)^{1/4}\Delta_n\right)\xrightarrow[n\rightarrow\infty]{(d)}\mathrm{Glue}\left(\cQ_{\cA,1-\cA,\cL}\right)
     \end{equation*}
     where the convergence takes place in $\K^{\bullet\bullet}\times \R$.
\end{proposition}
\begin{proof}
   To simplify notations, we set
   \[\Omega_n\left(Q_n^{(b)}\right)=\left(Q^{(b)}_n,\left(\frac{9}{8n}\right)^{1/4} d_{Q^{(b)}_n},x^{(b)}_n,y^{(b)}_n,\frac{1}{n}\mu^{(b)}_n,\left(\frac{9}{8n}\right)^{1/4}\Delta_n\right).\]
   By Corollary \ref{lien}, for every continuous bounded function $F$, we have
    \[\E\left[F(\Omega_n(Q^{(b)}_n)\right]=\E\left[\frac{\sqrt{2n}}{L_n^\bullet}\right]^{-1}\times\E\left[\frac{\sqrt{2n}\times F\left(\mathrm{Glue}\left(\overline{\Omega}_n\left(Qd_{A_n^\bullet,\overline{A}_n^\bullet,L_n^\bullet}\right)\right)\right)}{L_n^\bullet}\right].\]
By Proposition \ref{convergence taille} and Skorohod's representation theorem, we can suppose that the random variables $\left(\frac{L_n^\bullet}{\sqrt{2n}},\frac{A_n^\bullet}{n}\right)$ converge almost surely toward $(\cL^\bullet,\cA^\bullet)$. Hence, we can construct a probability space $(\Omega\times\tilde{\Omega},\cF,\P\otimes\tilde{\P})$ on which all the random variables considered are defined, and such that for every $\omega\in\Omega$, we have the convergence
\[\left(\tilde{\omega}\mapsto\overline{\Omega}_n\left(Qd^{(\tilde{\omega)}}_{A_n^\bullet(\omega),\overline{A}_n^\bullet(\omega),L_n^\bullet(\omega)}\right)\right)\xrightarrow[n\rightarrow\infty]{(d)}\cQ_{\cA^\bullet(\omega),1-\cA^\bullet(\omega),\cL^\bullet(\omega)},\]
using Theorem \ref{conv quad}. By integrating over $\tilde{\Omega}$ and Proposition \ref{gluing geodesics}, this implies that 
\begin{equation}\label{conv separe}
   \left( \frac{\sqrt{2n}}{L_n^\bullet},\mathrm{Glue}\left(\overline{\Omega}_n\left(Qd_{A_n^\bullet,\overline{A}_n^\bullet,L_n^\bullet}\right)\right)\right)\xrightarrow[n\rightarrow\infty]{(d)}\left(\frac{1}{\cL^\bullet},\mathrm{Glue}(\cQ_{\cA^\bullet,1-\cA^\bullet,\cL^\bullet})\right).
\end{equation}
Then, Lemma \ref{uniform integrability} together with \eqref{conv separe} gives that for every continuous bounded function $F$,
    \begin{equation}\label{conv glue}
        \E\left[\frac{\sqrt{2n}\times F\left(\overline{\Omega}_n\left(Qd_{A_n^\bullet,\overline{A}_n^\bullet,L_n^\bullet,
}\right)\right)}{L_n^\bullet}\right]\xrightarrow[n \rightarrow{\infty}]{}\E\left[\frac{F(\mathrm{Glue}(\cQ_{\cA^\bullet,1-\cA^\bullet,\cL^\bullet}))}{\cL^\bullet}\right]
    \end{equation}
    Furthermore, using Theorem \ref{counting}, we have 
    \[\E\left[\frac{\sqrt{2n}}{L_n^\bullet}\right]=\frac{ \#\U_n}{\sqrt{2n}\times\U_n^\bullet}\xrightarrow[n\rightarrow\infty]{}\sqrt{2}\frac{\Gamma(1/4)}{\Gamma(3/4)}.\]
    Together with \eqref{conv glue}, and using the explicit relation between $\cL^\bullet$ and $\cL$, this gives 
    \[\E\left[F(\Omega_n(Q_n^{(b)}))\right]\xrightarrow[n\rightarrow\infty]{}\E\left[F(\mathrm{Glue}(\cQ_{\cA,1-\cA,\cL}))\right],\] which concludes the proof. 
\end{proof}

\subsection{The biased Brownian sphere as a quotient of a continuum unicycle}\label{construction}

In this section, we use Proposition \ref{Conv collage} to derive a geometric construction of the biased Brownian sphere. This construction is based on the construction of quadrilaterals given in Section \ref{continuous quad}.

Once again, we start with some deterministic considerations. We say that a pair $(f,g)\in\cC^2$ of function is a labelled unicycle trajectory if :
\begin{itemize}[label=\textbullet]
    \item $(f,g)$ is a quadrilateral trajectory as defined in Section \ref{continuous quad},
    \item Moreover, 
    \begin{equation}\label{condition}
        g(0)=g(\inf I)=g(\sup I).
    \end{equation}
\end{itemize}

Then, for every $s,t\in I$ with $s\geq t$, we set
\[[s,t]=[s,\sup I]\cup [\inf I,t],\]
which naturally allows us to define $m_f(s,t)$ when $s>t$.
Then, for every $s,t\in I$ let
\begin{equation}\label{distance cycle}
    d'_f(s,t)=\inf\left(d_f(s,t)\,, d_f(s,\inf I)+d_f(0,t), \,d_f(t,\inf I)+d_f(0,s)\right).
\end{equation}
and
\begin{equation}
    \widehat{d}'_g(s,t)=\left\{
    \begin{array}{ll}
  d'_g(s,t)\quad&\text{ if $s,t\in I_+$ or $s,t\in I_-$,} \\   
     +\infty\quad&\text{ otherwise}
    \end{array}\right.
\end{equation}
Observe that if $d_f(0,\inf I)>0$, the space $\cT'_f=I/\{d'_f=0\}$ is a unicycle, in the sense that $\left(\cT'_f,d'_f\right)$ is obtained from $\left(\cT_f,d_f\right)$ by identifying the equivalence class of $0$ and $\inf I$ in $\cT_f$. Moreover, the condition \eqref{condition} implies that $g$ defines a continuous function on $\cT'_f$.  Then, we set
\[\widehat{D}'_{f,g}=\widehat{d}'_g/\{d_f=0\}.\]
Finally, let $\left(\widehat{M}'_{f,g},\widehat{D}'_{f,g}\right)$ denote the quotient space $I/\{\widehat{D}'_{f,g}=0\}$ equipped with the distance induced by $\widehat{D}'_{f,g}$. To lighten notations, we write 
 \[\cU_{f,g}=\left(\widehat{M}_{f,g},\widehat{D}_{f,g}\right)\] for the metric space coded by $(f,g)$. Note that this metric space is naturally equipped with a measure $\mathrm{Vol}_{f,g}$, which is the pushforward of the Lebesgue measure on $[\inf I,\sup I]$ by the canonical projection. 
 
Let us now randomize the labelled unicycle trajectory $(f,g)$ in order to define the candidate for the biased Brownian sphere.
Let $\textbf{Unic}^{(a)}$ be the probability distribution under which the law of $(X,W)$ is the following :
\begin{itemize}[label=\textbullet]
\item First, we define a pair of random variables $(\cL^{(a)},\cA^{(a)})$ whose distribution is given by the probability density 
 \begin{equation}\label{densité jointe longueur}
     \1_{\{(x,y)\in\R_+\times [0,a]\}} \frac{a^{5/4}}{2^{1/4}\Gamma(1/4)\sqrt{\pi}}\frac{x^{1/2}}{(y(a-y))^{3/2}}\exp\left(-\frac{ax^2}{2y(a-y)}\right),
 \end{equation}
    \item Given $(\cL^{(a)},\cA^{(a)})$, $\left(X_t,\,0\leq t\leq \cA^{(a)}\right)$ and $\left(X_{-t},\,0\leq t\leq a-\cA^{(a)}\right)$ are two independent first-passage bridges of standard Brownian motion from $0$ to $-a$, with respective durations $\cA^{(a)}$ and $a-\cA^{(a)}$, 
    \item given $(\cL^{(a)},\cA^{(a)})$ and $X$, the process $W$ has the law of $\left({Z}_t+\zeta_{-m'_X(t)},\,-\overline{A}\leq t\leq A\right)$, where ${Z}$ is the random snake driven by $X-m'_X$, and $\zeta$ is an independent Brownian bridge of duration $\cL^{(a)}$ from $0$ to $0$.
\end{itemize}
\begin{remark}
    The distribution of $(\cL^{(a)},\cA^{(a)})$ is the same as the one of $(\sqrt{a}\cL,a\cA)$.
\end{remark}
It is easy to check that for every $s,t\in I$ such that $d_X(s,t)=0$, we also have $W_s=W_t$. Therefore, the process $W$ induces a labelling function on the random unicycle $\cT_X'$. Even though we do not use it directly, under $\textbf{Unic}^{(a)}$ the pair $(X,W)$ encodes a \textit{continuous random labelled unicycle} (CRLU), which is the scaling limit of uniform labelled unicycles. This point of view will be more explicit in Section \ref{coding unicycles with point measures}

We denote by $\left(\cS_{b}^{(a)},D\right)$ the random metric space $\cU_{X,W}$ under the probability measure $\textbf{Unic}^{(a)}$. Let $p_{\cS^{(a)}_{b}}:I\rightarrow\cS_{b}^{(a)}$ denote the canonical projection. These random spaces come with two distinguished points, which are defined as follows.
As in Section \ref{brownian sphere}, almost surely, there exists a unique $s_*\in I_-$ and $\overline{s}_*\in I_+$ such that 
\begin{equation*}\label{minimumm cote}
    W_{s_*}=\inf_{s\in I_-}W_s\quad \text{ and }\quad W_{\overline{s}_*}=\inf_{s\in I_+}W_s.
\end{equation*}

The two distinguished points of $\cS_b^{(a)}$ are defined as $x_*=p_{\cS_b^{(a)}}(s_*)$ and $\overline{x}_*=p_{\cS_b^{(a)}}(\overline{s}_*)$. 
We also define the delay of $\cS_b^{(a)}$ as the random variable $\Delta:=W_{\overline{s}_*}-W_{s_*}$. Finally, we write $\mathrm{Vol}^{(a)}$ for the measure associated to $\cS_b^{(a)}$, and note that $\mathrm{Vol}^{(a)}\left(\cS_b^{(a)}\right)=a$. 
\begin{theorem}\label{equiv def}
    The random metric space $\cS_b^{(1)}$ is distributed as the biased Brownian sphere.
\end{theorem}
\begin{proof}
    By Proposition \ref{Conv collage}, we need to show that 
    \begin{equation}\label{recollement}
        \left(\cS_b^{(1)},D,\mathrm{Vol}^{(1)},x_*,\overline{x}_*,\Delta\right)\overset{(d)}{=}\mathrm{Glue}\left(\cQ_{\cA,1-\cA,\cL}\right).
    \end{equation}

    We will prove that for every unicycle trajectory $(f,g)$, the metric spaces $\cU_{f,g}$ and $\mathrm{Glue}(\mathrm{Qd}_{f,g})$ are the same. First, note that these two spaces are constructed as the quotient of the same interval $I$. Therefore, we need to show that the equivalence relations defining these spaces are the same. 

    Recall that $\mathrm{Qd}_{f,g}$ is the quotient of $I$ by the equivalence relation $\{\widehat{D}_{f,g}=0\}$, equipped with the distance $\widehat{D}_{f,g}=\widehat{d}_g/\{d_f=0\}$. Moreover, the gluing operation consists in identifying the geodesics $\Gamma_0$, $\Xi_0$ with $\Xi_{\sup I}$, $\Gamma_{\inf I}$. Therefore, $\mathrm{Glue}(\mathrm{Qd}_{f,g})$ corresponds to $I/R$, where $R$ is the coarsest equivalence relation that contains $\{\widehat{D}_{f,g}=0\}$ and  
    \[\mathrm{Geo}:=\left\{\left(\Gamma_0(r),\Xi_{\sup I}(r)\right),\,0\leq r\leq \inf_{I_+}g\right\}\cup\left\{\left(\Xi_0(r),\Gamma_{\inf I}(r)\right),\,0\leq r\leq \inf_{I_-}g\right\},\]
    and the distance on $\mathrm{Glue}(\mathrm{Qd}_{f,g})$ is
    \[\widehat{D}_{f,g}/\mathrm{Geo}=\left(\widehat{d}_g/\{d_f=0\}\right)/\mathrm{Geo}=\widehat{d}_g/R_f,\]
    where $R_f$ is the coarsest equivalence relation that contains $\{d_f=0\}$ and $\mathrm{Geo}$.

    On the other hand, $\cU_{f,g}$ corresponds to $I/\{\widehat{D}'_{f,g}=0\}$, and $\widehat{D}'_{f,g}=\widehat{d}'_g/\{d_f=0\}$. 
  However, by comparing \eqref{distance arbre} with \eqref{distance cycle}, it is easy to see that $\widehat{d}'_{g}=\widehat{d}_g/\mathrm{Geo}$. Moreover, since we are gluing alongside geodesics, if $\widehat{d}_g'(s,t)=0$, then either $\widehat{d}_g(s,t)=0$ or $(s,t)\in\mathrm{Geo}$. Therefore, the equivalence relation $R$ and $\widehat{D}_{f,g}$ are the same, and we have
    \[\widehat{D}'_{f,g}=\widehat{d}'_g/\{d_f=0\}=\left(\widehat{d}_g/\mathrm{Geo}\right)/\{d_f=0\}=\widehat{d}_g/R_f.\]
    
Therefore, the metric spaces $\cU_{f,g}$ and $\mathrm{Glue}(\mathrm{Qd}_{f,g})$ are the same. Using the same ideas, one can show that the measures also coincide. Moreover, it is clear that the distinguished points of $\mathrm{Qd}_{f,g}$ correspond to those of $\cU_{f,g}$ after the gluing, and that the delay is preserved. To conclude, note that by definition, the trajectories encoding each side of $\eqref{recollement}$ have the same distribution, which concludes the proof. 
\end{proof}

Let us make an important remark. First, if $\zeta$ is a Brownian bridge of duration $a$ between $0$ and $0$, it is known that the Vervaat transform of $\zeta$ is a Brownian excursion of duration $a$ \cite{Vervaat}. Then, note that replacing $\zeta$ by its Vervaat transform just induces an additive shift of the labelling process $W$ in the definition of $\textbf{Unic}^{(a)}$, which does not change the resulting metric space. Therefore, by combining these observations, we see that the distribution of 
 $\cS_b$ does not change if we replace the Brownian bridge $\zeta$ by a Brownian excursion in the definition of the probability measure $\textbf{Unic}^{(a)}$ (we used the fact that the Brownian bridge is independent of the Brownian snake).

 \textbf{For the rest of this paper, we will replace the Brownian bridge by a Brownian excursion in the definition of $\textbf{Unic}^{(a)}$.}

To conclude this section, let us mention some results of the Brownian sphere that also hold for the biased Brownian sphere. We omit the proofs, since they are adaptations of the original ones. First, we give the counterpart of Proposition \ref{identification}.

\begin{proposition}\label{identification2}
    The following holds almost surely. For every $s,t\in[-\cA,1-\cA]$ with $s\neq t$ such that $D(s,t)=0$, we have 
    \[d_\e(s,t)=0\quad\text{ or }\quad \widehat{d}'_Z(s,t)=0.\]
    Moreover, the two equalities cannot hold simultaneously.
\end{proposition}

Then, we introduce the notion of simple geodesics in $\cS_b$. Fix $s\in I_-$. If $s\leq s_*$, we define 

\[\Gamma_s(r)=\inf\{t\geq s\,:\,W_t=W_s-r\}\quad\text{ for $r\in\R_+$ such that }W_{s_*}\leq W_s-r,\]
and if $s\geq s_*$, we define
\[\Gamma_s(r)=\sup\{t\leq s\,:\,W_t=W_s-r\}\quad\text{ for $r\in\R_+$ such that }W_{s_*}\leq W_s-r.\]
Then, for every $r\in [0,W_s-W_{s_*}]$, we set
\[\gamma_s(r)=p_\cS(\Gamma(r)).\]
Using \eqref{bound}, one can check that for every $s\in I_-$, the curve $\gamma_s$ is a geodesic from $p_{\cS_b}(s)$ to $x_*$. These geodesics are called the \textit{simple geodesics to $x_*$}. Similarly, for every $s\in I_+$, one can define the simple geodesic $\gamma_s$ from $p_{\cS_b}(s)$ to $\overline{x}_*$. The following result is the counterpart of \cite[Theorem 7.6]{geodesic1} for the biased Brownian sphere. 

\begin{proposition}\label{simple geodesics}
    Almost surely, every geodesic from an element of $p_{\cS_b}(I_-)$ to $x_*$ (respectively from an element of $p_{\cS_b}(I_+)$ to $\overline{x}_*$) are simple geodesics. 
\end{proposition}
\begin{proof}
    We only sketch the proof, since this is an adaptation of known result for the Brownian sphere. Since the law of the biased Brownian sphere is absolutely continuous with respect to the law of the standard Brownian sphere \eqref{def biais sphere}, by \cite[Corollary 7.5]{geodesic1}, we know that for $\mu$-almost every $x\in p_{\cS_b}(I_-)$, there is a unique geodesic between $x$ and $x_*$, which must be the simple geodesic. Then, we can use the same strategy an in \cite[Proposition 23]{geodesicsbettinelli} to show that every geodesic from any $x\in p_{\cS_b}(I_-)$ to $x_*$ are simple geodesics. 
\end{proof}

\subsection{Delayed Voronoï cells}

In this section, we present the link between the construction of the biased Brownian sphere described in Section \ref{construction} and the delayed Voronoï cells.

For any metric space $(X,d)$ with two distinguished points $x$ and $y$, and any parameter $\delta\in (-d(x,y),d(x,y))$, we define 
\begin{align*}
    \Theta_\delta&:=\{z\in X: d(x,z)\leq d(y,z)+\delta\},\\
    \overline{\Theta}_\delta&:=\{z\in X: d(x,z)\geq d(y,z)+\delta\}.
\end{align*}

\begin{definition}
    The sets $\Theta_\delta$ and $\overline{\Theta}_\delta$ are called the $\delta$-delayed Voronoï cells of $X$ with respect to $x$ and $y$. 
\end{definition}

This notion was already considered with any number of points in \cite{tessalations} for maps of arbitrary genus. Note that when $\delta=0$, we recover the usual notion of Voronoï cells. Delayed Voronoï cells can be seen as typical Voronoï cells, but where we give an advantage to either $x$ or $y$ (depending on the parameter $\delta$).

Let us see how the explicit construction of the biased Brownian sphere gives information about delayed Voronoï cells. Let $I$ be the (random) interval used to construct the biased Brownian sphere, that is $I=[-\cA,1-\cA]$. Recall that if $s,t\in I$ are such that $p_{\cS_b}(s)=p_{\cS_b}(t)$, then $W_s=W_t$. Therefore, for any $x\in\cS_b$ we can unambiguously set $W_x:=W_s$, where $s\in I$ satisfies $p_{\cS_b}(s)=x$. We define 
 \[W_*:=\inf_{t\in I_+}W_t=W_{x_*}\quad\text{ and }\quad\overline{W}_*=\inf_{t\in I_-}W_t=W_{\overline{x}_*}.\]

The classical bound \eqref{bound} shows that for every $x\in\cS_b$, $D(x,x_*)\geq W_x-W_*$ and $D(x,\overline{x}_*)\geq W_x-\overline{W}_*$. On the other hand, by considering the simple geodesics as in Section \ref{continuous quad}, one can see that 
 \begin{equation}\label{egalite distance}
 D(x,x_*)=W_x-W_*\quad\text{ for every $x\in p_{\cS_b}(I_+)$, and }\quad D(x,\overline{x}_*)=W_x-\overline{W}_*\quad\text{ for every $x\in p_{\cS_b}(I_-)$}.
\end{equation}
Then, for every $s\in[0,\cL]$, set 
\[\tau_s=\inf\{t\in I_+:X_t=-s\}\quad\text{ and }\quad\overline{\tau}_s=\sup\{t\in I_-:X_t=-s\}.\]

It is easy to see that for every $s\in[0,\cL]$, $p_{\cS_b}(\tau_s)=p_{\cS_b}(\overline{\tau}_s)$, and that $p_{\cS_b}(\tau_0)=p_{\cS_b}(\tau_\cL)$. Set $(\eta_s)_{s\in[0,\cL]}:=p_{\cS_b}((\tau_s)_{s\in[0,\cL]})$, which is a closed path in $\cS_b$. 
\begin{proposition}\label{frontier}
    The path $\eta$ separates $p_{\cS_b}(I_+)$ from $p_{\cS_b}(I_-)$ in $\cS_b$.
\end{proposition}
\begin{proof}
    Fix $x\in p_{\cS_b}(I_+)$, $y\in p_{\cS_b}(I_-)$, and a continuous path $\gamma:[0,1]\rightarrow\cS_b$ such that $\gamma(0)=x,\gamma(1)=y$. We will prove that $\gamma$ must intersect $\eta$. Set
    \[\alpha=\inf\{t\in[0,1],\,\gamma(t)\in p_{\cS_b}(I_-)\}.\]
    By definition of $x$ and $y$, we know that $\alpha<\infty$. Consider $s\in[0,1]$ such that $p_{\cS_b}(s)=\gamma(\alpha)$. Since  $p_{\cS_b}(I_-)$ is a closed set, we can choose $s$ so that $s\in I_-$. On the other hand, for every $n\in\N$, consider $s_n\in I_+$ such that  $p_{\cS_b}(s_n)=\gamma((\alpha-1/n)\vee 0)$. Since $I_+$ is compact, we can extract a subsequence of $(s_n)_{n\in \N}$ that converges towards some $s_\infty\in I_+$. Moreover, by continuity of $\gamma$, it holds that $p_{\cS_b}(s_\infty)=\gamma(\alpha)$, which means that $D(s,s_\infty)=0$. By Proposition \ref{identification2}, this means that either $d_X(s,s_\infty)=0$ or $\widehat{d}'_{W}(s,s_\infty)=0$. If $s=0$ or $s_\infty=0$, then $\gamma(\alpha)=\eta(0)$, which gives the result. Otherwise, we have $\widehat{d}'_{W}(s,s_\infty)=\infty$, so that $d_X(s,s_\infty)=0$. However, it is easy to see that the image by the canonical projection of the elements $s_-\in I_-,s_+\in I_+$ satisfying $d_X(s-,s_+)=0$ is exactly $\eta$, which concludes the proof.
\end{proof}

Consequently, using the bound \eqref{egalite distance}, we obtain that for every $x\in p_{\cS_b}(I_-)$, 
\begin{equation} \label{ineq dist 1}
D(x,x_*)=\inf_{s\in[0,\cL]}D(x,\eta_s)+D(\eta_s,x_*)\geq\inf_{s\in[0,\cL]}|W_x-W_{\eta_s}|+|W_{\eta_s}-W_*|\geq W_x-W_*,
\end{equation}

and for every $x\in p_{\cS_b}(I_+)$, 
 \begin{equation}\label{ineq dist 2}
D(x,\overline{x}_*)=\inf_{s\in[0,\cL]}D(x,\eta_s)+D(\eta_s,\overline{x}_*)\geq\inf_{s\in[0,\cL]}|W_x-W_{\eta_s}|+|W_{\eta_s}-\overline{W}_*|\geq W_x-\overline{W}_*.
 \end{equation}
 Combining these inequalities with \eqref{egalite distance}, we have 
 \begin{equation}\label{delayed cells}
 p_{\cS_b}(I_+)=\{x\in\cS_b,\,D(x,x_*)\leq D(x,\overline{x}_*)-W_*+\overline{W}_*\},
 \end{equation}
 and 
 \[p_{\cS_b}(I_-)=\{x\in\cS_b,\,D(x,x_*)\geq D(x,\overline{x}_*)-W_*+\overline{W}_*\}.\]

Therefore, $p_{\cS_b}(I_-)$ and $ p_{\cS_b}(I_+)$ correspond the $\Delta$-delayed Voronoï cells of $\cS_{b}$ with respect to $x_*$ and $\overline{x}_*$.
Moreover, the inequalities \eqref{egalite distance} show that for every $s\in[0,\cL]$,
\begin{equation}\label{chepa}
D(\eta_s,x_*)=D(\eta_s,\overline{x}_*)-W_*+\overline{W}_*.
\end{equation}
The following proposition shows that these are the only points that satisfy this equality. 
\begin{proposition}\label{frontiere}
    Almost surely, we have 
    \[\{\eta_s,\,s\in[0,\cL]\}=\{x\in\cS_b,\,D(x,x_*)= D(x,\overline{x}_*)-W_*+\overline{W}_*\}.\]
\end{proposition}
\begin{proof}
    Fix $x\in\cS_b\backslash\eta([0,\cL])$. By \eqref{chepa}, we need to check that
    \[D(x,x_*)\neq D(x,\overline{x}_*)-W_*+\overline{W}_*.\]
    Without loss of generality, suppose that $x\in p_{\cS_b}(I_+)$. By \eqref{delayed cells}, we just need to show that $D(x,x_*)< D(x,\overline{x}_*)-W_*+\overline{W}_*.$ First, as in \eqref{ineq dist 2}, observe that 
    \[D(x,x_*)\geq \inf_{s\in[0,\cL]}D(x,\eta_s)-W_*>-W_*.\]
    Using \eqref{egalite distance}, this proves the result when $W_x\leq 0$.
  In the case $W_x>0$, we argue by contradiction. Suppose that $D(x,x_*)=D(x,\overline{x}_*)-W_*+\overline{W}_*$, and consider a geodesic $\gamma$ between $x$ and $x_*$. First, note that for every $t\in[0,D(x,x_*)]$, it holds that $W_{\gamma(t)}=W_x-t$ (this is just a consequence of \eqref{bound} and the fact that $\gamma$ is a geodesic). Then, we define 
  \[\tau=\inf\left\{t\geq0,\,\gamma(t)\in\eta([0,\cL])\right\}.\]
  By Proposition \ref{frontier}, we have $\tau<\infty$.
  Then, consider a geodesic $\tilde{\gamma}$ between $\gamma(\tau)$ and $\overline{x}_*$. An adaptation of the main result of \cite{geodesic1} shows that for any $x'\in p_{\cS_b}(I_+)$, all the geodesics from $x'$ to $\overline{x}_*$ are simple geodesics. In particular, for every $t\in[0,W_{\gamma(\tau)}-W_*]$, we have $W_{\tilde{\gamma}(t)}=W_{\gamma(\tau)}-t$. Consequently, the concatenation of $\gamma|_{[0,\tau]}$ and $\tilde{\gamma}$ is also a geodesic between $x$ and $\overline{x}_*$. As mentioned previously, it must be a geodesic from $x$ and $\overline{x}_*$. However, by Proposition \ref{simple geodesics}, it is a simple geodesic, and Proposition \ref{identification2} implies that simple geodesics do not hit the range of $\eta$, which gives a contradiction. This concludes the proof.
\end{proof}

Hence, $\eta$ corresponds to the boundary of $\Theta_\Delta$ and $\overline{\Theta}_\Delta$. Therefore, even though the curve $\eta$ is fractal and of infinite length, we can interpret the random variable $\cL$ as the perimeter of the frontier of the $\Delta$-delayed Voronoï cells. In particular, with this interpretation, the joint distribution of the perimeter and the volume of $\Theta_\Delta$ has density \eqref{densité jointe longueur}.

\begin{remark}
    One may consider another notion of perimeter for the curve $\eta$, which is its Minkowski content. We believe that these two notions are the same, up to a deterministic multiplicative constant. Note that similar results are known for the boundary of a Brownian disk \cite[Theorem 9]{Diskboundary}, of hulls \cite[Proposition 1.1]{Hullprocess2016}, and for a portion of the boundary of the Voronoï cells with respect to a boundary segment in the Brownian disk \cite[Section 8.3]{rierathese}.
\end{remark}

\begin{corollary}\label{distance frontiere}
    As one moves along the boundary of $\Theta_\Delta$ with the parametrization described above, the distances to $x_*$ and $\overline{x}_*$ evolve, up to random additive constants, as a Brownian bridge of length $\cL$. 
\end{corollary}
\begin{proof}
    This immediately follows from the explicit construction of Section \ref{construction}, together with \eqref{frontiere}
\end{proof}

The following proposition identifies the law of $\mathrm{Vol}(\Theta_\Delta)$.
\begin{proposition}[Law of the volumes]\label{volume cells}
    Let $\Theta_\Delta$, $\overline{\Theta}_\Delta$ be the $\Delta$-delayed Voronoï cells of $x_*$ and $\overline{x}_*$ in $\left(\cS_b,D,x_*,\overline{x}_*,\Delta\right)$. Then, $\left(\mathrm{Vol}(\Theta_\Delta),\mathrm{Vol}(\overline{\Theta}_\Delta)\right)$ follows a Dirichlet distribution with parameter $(1/4,1/4)$. 
\end{proposition}
\begin{proof}
    We just need to show that $\mathrm{Vol}(\Theta_\Delta)$ is a Beta random variable of parameter $(1/4,1/4)$. As mentioned above, the law of $\mathrm{Vol}(\Theta_\Delta)$ is the same as the law of $\cA$ under $\textbf{Unic}^{(1)}$. Therefore, by \eqref{densité jointe longueur}, the density of $\mathrm{Vol}(\Theta_\Delta)$ at $y\in[0,1]$ is given by 
    \begin{align*}
       &\frac{1}{2^{1/4}\Gamma(1/4)\sqrt{\pi}}\int_{R_+}\frac{x^{1/2}}{(y(1-y))^{3/2}}\exp\left(-\frac{x^2}{2y(1-y)}\right)dx\\
       &=\frac{1}{2^{1/4}\Gamma(1/4)\sqrt{\pi}}\frac{1}{(y(1-y))^{3/2}}\int_{R_+}\frac{(y(1-y))^{3/4}}{(2u)^{1/4}}\exp(-u)du\\
        &=\frac{\Gamma(3/4)}{\Gamma(1/4)\sqrt{2\pi}}\frac{1}{(y(1-y))^{3/4}}.
    \end{align*}
    By Euler's reflection formula, we have 
    \[\frac{\Gamma(3/4)}{\Gamma(1/4)\sqrt{2\pi}}\frac{1}{(y(1-y))^{3/4}}=\frac{\sqrt{\pi}}{\Gamma(1/4)^2\sqrt{2}\sin(\pi/4)}\frac{1}{(y(1-y))^{3/4}}=\frac{1}{\mathrm{B}(1/4,1/4)}\frac{1}{(y(1-y))^{3/4}},\]
    
    which gives the result.
\end{proof}
We conclude this section by giving a geometric property of the curve $\eta$.
\begin{proposition}\label{simple curve}
    Almost surely, the curve $\eta$ is simple. 
\end{proposition}
\begin{proof}
By definition, the times $(\tau_s)_{s\in[0,\cL]}$ correspond to one-sided local minima of the process $X$ used to construct $\cS_b$. Then, an adaptation of \cite[Lemma 3.2]{sphericity} shows that the times $(\tau_s)_{s\in[0,\cL]}$ cannot be one-sided local minima of the process $W$. The result follows from Proposition \ref{identification2}.
\end{proof}

\section{The free model}\label{section free model}

In this section, we introduce the free version of the biased Brownian sphere, and present some of its basic properties. As we will see, this model is more flexible than the one with fixed volume. In particular, we will show how to condition the free model to have a fixed delay.

\subsection{Snake trajectories and the Brownian snake}\label{snake}

In this subsection, we recall some basic notions about snake trajectories, and the Brownian snake excursion measure (we refer to \cite{serpent} for more details).

A finite path is a continuous function $w:[0,\zeta]\rightarrow\R$, where $\zeta=\zeta(w)\geq0$ is the lifetime of $w$, and we let $\widehat{w}=w(\zeta)$ be the terminal value of $w$. We denote by $\mathfrak{W}$ the set of all finite paths in $\R$, and for every $x\in\R$, we set $\mathfrak{W}_x:=\{w\in\mathfrak{W}\,:w(0)=x\}$. We also identify the point $x\in\R$ with the element of $\mathfrak{W}_x$ with $0$ lifetime.

 Fix $x\in\R$. A \textbf{snake trajectory} is a continuous function $t\mapsto\omega_t$ from $\R_+$ to $\mathfrak{W}_x$ such that 
\begin{itemize}[label=\textbullet]
    \item $\omega_0=x$ and $\sigma(\omega)=\sup\{s\geq0\,:\,\omega_s\neq x\}<\infty$,
    \item for every $0\leq s\leq s'$, $\omega_s(t)=\omega_{s'}(t)$ for every $t\in[0,\min_{s\leq r\leq s'}\zeta(\omega_r)]$. 
\end{itemize}
The quantity $\sigma(\omega)$ is called the duration of the snake trajectory. We denote by $\mathfrak{S}_x$ the set of snake trajectories starting from $x$, and $\mathfrak{S}$ the set of all snake trajectories. It will be useful to use the notation $W_s(\omega)=\omega_s$, $\zeta_s(\omega)=\zeta(\omega_s)$ and $W_*(\omega)=\inf_{t\geq0}\widehat{W}_t(\omega)$. We mention that a snake trajectory $\omega$ is determined by its tip function $s\mapsto\widehat{W}_s(\omega)$ and its lifetime function $s\mapsto\zeta_s(\omega)$ (see \cite{Refserpent} for a proof).

The lifetime function $\zeta(\omega)$ of a snake trajectory $\omega$ encodes a compact $\R$-tree, denoted by $\cT_\omega$. More precisely, we introduce a pseudo-distance on $[0,\sigma(\omega)]$ with the formula 
\[d_{(\omega)}(s,t)=\zeta_s(\omega)+\zeta_t(\omega)-2\min_{s\wedge t\leq r\leq s\vee t}\zeta_r(\omega),\]
and $\cT_\omega$ is the quotient space $[0,\sigma(\omega)]/\{d_{(\omega)}=0\}$, equipped with the distance induced by $d_{(\omega)}$. We denote by $p_{\cT_\omega}:[0,\sigma(\omega)]\rightarrow\cT_\omega$ the canonical projection, and root the tree at $\rho_{\cT_\omega}:=p_{\cT_\omega}(0)=p_{\cT_\omega}(\sigma(\omega))$. This space also has a volume measure, which is the pushforward of the Lebesgue measure on $[0,\sigma(\omega)]$ by the canonical projection $p_{\cT_\omega}$. Moreover, note that because of the snake property, $W_s(\omega)=W_t(\omega)$ if $p_{\cT_\omega}(s)=p_{\cT_\omega}(t)$. Therefore, the mapping $s\rightarrow\widehat{W}_s(\omega)$ can be viewed as a function on $\cT_\omega$. In this paper, for $u\in\cT_\omega$, and $s\in[0,\sigma(\omega)]$ such that $p_{\cT_\omega}(s)=u$, we will often use the notation $\ell_u:=\widehat{W}_s(\omega)$. 

We also define intervals on $\cT_\omega$ as follows.
For every $s,t\in[0,\sigma(\omega)]$ with $s<t$, we use the convention $[t,s]=[t,\sigma(\omega)]\cup[0,s]$. Observe that for every $u,v\in\cT_\omega$, there exists a smallest interval $[s,t]$ such that $p_{\cT_\omega}(s)=u$ and $p_{\cT_\omega}(t)=v$, and we set 
\[[u,v]:=\{p_{\cT_\omega}(r):r\in[s,t]\}.\]

Now, we can define the Brownian snake excursion measure. For every $x\in\R$, we define a $\sigma$-finite measure on $\mathfrak{S}_x$, denoted by $\N_x$, as follows. 
Under $\N_x$ :
\begin{itemize}
    \item the lifetime function $(\zeta_s)_{s\geq0}$ is distributed according to the Itô measure of positive excursions of linear Brownian motion, normalized so that the density of $\sigma$ is $t\mapsto\left(2\sqrt{2\pi t^3}\right)^{-1}$, 
    \item conditionally on $(\zeta_s)_{s\geq0}$, the tip function $\left(\widehat{W}_s\right)_{s\geq0}$ is a Gaussian process with mean $x$ and covariance function 
    \[K(s,t)=\min_{s\wedge t\leq r\leq s\vee t}\zeta_r.\]    
\end{itemize}
    The measure $\N_x$ appears as an excursion measure away from $x$ for the Brownian snake, which is a Markov process in $\mathfrak{W}_x$. For every $t>0$, we can define the probability measure $\N_x^{(t)}=\N_x(\cdot\,|\,\sigma=t)$, which can be constructed by replacing the Itô measure used to define $\N_x$ by a Brownian excursion with duration $t$.

    For every $y<x$, we have the formula 
    \begin{equation}\label{inf}
        \N_x(W_*<y)=\frac{3}{2(x-y)^2}
    \end{equation}
    (see \cite{serpent} for a proof). Hence, we can define the conditional probability measure $\N_x(\cdot\,|\,W_*<y)$. Moreover, under $\N_x$ or $\N_x^{(t)}$, a.e, there exists a unique $s_*\in[0,\sigma]$ such that $\widehat{W}_{s_*}=W_*$ (see \cite[Proposition 2.5]{Conditionnedbrowniantrees}).
Finally, for every $\lambda>0$ and $\omega\in\mathfrak{S}_x$, we define $\Theta_\lambda(\omega)\in\mathfrak{S}_{x\sqrt{\lambda}}$ by $\Theta_\lambda(\omega)=\omega'$, with 
\[\omega'_s(t)=\sqrt{\lambda}\omega_{s/\lambda^2}(t/\lambda),\quad\text{ for }s\geq0\text{ and }0\leq t\leq \zeta_s(\omega'):=\lambda\zeta_{s\lambda^2}(\omega).\]
Then, the pushforward of $\N_x$ by $\Theta_\lambda$ is $\lambda\N_{x\sqrt{\lambda}}$, and the pushforward of $\N_x^{(t)}$ is $\N_{x\sqrt{\lambda}}^{(\lambda^2t)}$.

Finally, observe that under $\N_0$ (resp. $\N_0^{(1)}$), the lifetime function and the tip function are distributed as the pair of processes that encode the free Brownian sphere (resp. the standard Brownian sphere) in Section \ref{brownian sphere}. Therefore, in what follows, we will often use these measures to construct the Brownian sphere.

\subsection{Definition and connection with the biased Brownian sphere}

Let $\mathbf{Unic}$ be a $\sigma$-finite measure under which:
\begin{itemize}[label=\textbullet]
    \item $\e$ is distributed according to Itô's excursion measure of Brownian motion,
    \item given $\e$, $\left(X_t,\,0\leq t\leq T_{-\sigma}\right)$ and $\left(X_{-t},\,0\leq t\leq \overline{T}_{-\sigma}\right)$ are two independent Brownian motion started at 0, and stopped when they reach $-\sigma$, where $\sigma$ is the duration of $\e$,
    \item given $\e$ and $X$, the process $W$ has the law of $\left({Z}_t+\e_{-m'_X(t)},\,-\overline{T}_{-\sigma}\leq t\leq T_{-\sigma}\right)$, where ${Z}$ is the random snake driven by $X-m'_X$.
\end{itemize}

\begin{definition}
    The free biased Brownian sphere is the metric space $\cU_{X,W}$ under the measure $\mathbf{Unic}$.
\end{definition}

This space is also equipped with a volume measure $\mathrm{Vol}$. Note that the total volume of the free biased Brownian sphere is distributed as $T_{-\sigma}+\overline{T}_{-\sigma}$. Our first result about this model is to relate it two the standard biased Brownian sphere. More precisely, we show that under $\mathbf{Unic}$, we can decompose the measure according to its volume, and that $\mathbf{Unic}^{(a)}$ corresponds to the measure $\mathbf{Unic}$ ``conditionally on $\{\sigma=a\}$''. 

\begin{proposition}[Link with the biased Brownian sphere]\label{decomposition volume}
    For every positive measurable $F$ and $G$, we have 
\begin{equation*}   
    \mathbf{Unic}\left(F(\e,(X_t)_{t\geq 0},(X_{-t})_{t\geq 0},W)G(\sigma)\right)=\int_{\R_+}\frac{da}{a^{5/4}}\frac{\Gamma(1/4)}{2^{9/4}\pi}G(a)\mathbf{Unic}^{(a)}\left(F(\e,(X_t)_{t\geq 0},(X_{-t})_{t\geq 0},W)\right).
    \end{equation*}
\end{proposition}
\begin{proof}
    First, we prove the result for functions that do not depend on $W$. Fix $F$ and $G$ two positive measurable functions. By definition, we have :
    \begin{align*}
     \mathbf{Unic}\left(F(\e,(X_t)_{t\geq 0},(X_{-t})_{t\geq 0})G(\sigma)\right)&=\n\left(F(\e,B^{(\sigma)},\overline{B}^{(\sigma)})G(T_{-\sigma}+\overline{T}_{-\sigma})\right)\\
     &=\int_0^\infty\frac{dx}{2\sqrt{2\pi x^3}}\E\left[F(\e^{(x)},B^{(x)},\overline{B}^{(x)})G(T_{-x}+\overline{T}_{-x})\right]    
    \end{align*}
    where $\e^{(x)}$ is a Brownian excursion of duration $x$, and $B,\overline{B}$ are independent Brownian motion stopped at their first hitting time of $-x$. Then, using the explicit densities \eqref{joint law}, we can decompose the integral according to the values of $T_{-x}$ and $\overline{T}_{-x}$ which gives     
    \[ 
 \n\left(F(\e,B^{(\sigma)},\overline{B}^{(\sigma)})G(\sigma)\right)=\int_{\R_+^3}\frac{dxdydz}{2\sqrt{2\pi x^3}}\frac{x^2}{2\pi(yz)^{3/2}}\exp\left(-\frac{x^2}{2y}-\frac{x^2}              {2z}\right)G(y+z)\E\left[F(\e^{(x)},B^{(x,y)},\overline{B}^{(x,z)})\right].\]
    This can be rewritten as 
    \begin{equation}\label{Volume décompo}        
   \int_{\R_+}daG(a)\int_{\R_+\times[0,a]}\frac{dxdy}{2(2\pi)^{3/2}}\frac{x^{1/2}}{(y(a-y))^{3/2}}\exp\left(-\frac{ax^2}{2y(a-y)}\right)\E\left[F(\e^{(x)},B^{(x,y)},\overline{B}^{(x,a-y)})\right]. \end{equation}
    To conclude, we need to renormalize the second integral in order to have a probability measure. As in the proof of Proposition \ref{convergence taille}, we get
    \begin{align*}
        \int_{\R_+\times[0,a]}\frac{x^{1/2}}{(y(a-y))^{3/2}}\exp\left(-\frac{ax^2}{2y(a-y)}\right)dxdy&=\frac{2^{1/4}\Gamma(1/4)\sqrt{\pi}}{a^{5/4}}.
    \end{align*}
    Therefore, \eqref{Volume décompo} is equal to 
    \begin{multline*}
        \int_{\R_+}\frac{da}{a^{5/4}}\frac{\Gamma(1/4)}{2^{9/4}\pi}G(a)\int_{\R_+\times[0,a]}dxdy\frac{a^{5/4}}{2^{1/4}\Gamma(1/4)\sqrt{\pi}}\\\frac{x^{1/2}}{(y(a-y))^{3/2}}\exp\left(-\frac{ax^2}{2y(a-y)}\right)\E\left[F\left(\e^{(x)},B^{(x,y)},\overline{B}^{(x,a-y)}\right)\right].
    \end{multline*}
    Finally, according to the definition of the measure $\mathbf{Unic}^{(a)}$, we obtain 
    \begin{equation}\label{intermediate}     \n\left(F(\e,B^{(\sigma)},\overline{B}^{(\sigma)})G(\sigma)\right)=\int_{\R_+}\frac{da}{a^{5/4}}\frac{\Gamma(1/4)}{2^{9/4}\pi}G(a)\mathbf{Unic}^{(a)}\left(F(\e,(X_t)_{t\geq 0},(X_{-t})_{t\geq 0})\right),
    \end{equation} which concludes the proof in this case. Finally, note that we can deduce the statement of the proposition from this result, since for every non-negative function $F$, 
    \begin{align*}
        \n(F(\e,(X_t)_{t\geq 0},(X_{-t})_{t\geq 0},(W_t)_{t\in\R}))&=\n\left(\E\left[F(\e,(X_t)_{t\geq 0},(X_{-t})_{t\geq 0},(W_t)_{t\in\R})\,|\,\e,X\right]\right)\\
        &:=\n\left(F'(\e,(X_t)_{t\geq 0},(X_{-t})_{t\geq 0},(W'_t)_{t\in\R})\right)
    \end{align*}
    where given $\e$ and $X$, $W'$ has the law of $\left({Z}_t+\e_{-m'_X(t)},\,-\overline{T}_{-\sigma}\leq t\leq T_{-\sigma}\right)$, where ${Z}$ is the random snake driven by $X-m'_X$. Therefore, we can apply the identity \eqref{intermediate} to $F'$, which gives the result.
\end{proof}

We still denote by $\cS_b$ the random metric space associated to the measure $\textbf{Unic}$. From now on, to avoid any confusions, we will always denote by $\cS_b^{(a)}$ the random space obtained under the measure $\textbf{Unic}^{(a)}$. Finally, we can use Proposition \ref{decomposition volume} to relate the measure $\mathbf{Unic}$ to the free Brownian sphere.

\begin{proposition}[Link with the free Brownian sphere]\label{link free model}
   For every non-negative function $F$, we have 
   \[\N_0\left(\int_{-D(x,y)}^{D(x,y)}F(\cS,D,x,y,\mu,t)dt\right)=3\cdot\mathbf{Unic}(F(\cS_b,D,x,y,\mathrm{Vol},\Delta))\]
   and 
   \[\N_0(D(x,y)F(\cS,x,y))=\frac{3}{2}\mathbf{Unic}(F(\cS_b,D,x,y,\mathrm{Vol})).\]
\end{proposition}
\begin{proof}
To lighten notation, we omit the notation of the distance $D$ and the volume measure $\mathrm{Vol}$ in the proof. Using a standard decomposition of the measure $\N_0$, we have 
   \[\N_0\left(\int_{-D(x,y)}^{D(x,y)}F(\cS,x,y,t)dt\right)=\int_{\R_+}\frac{1}{2\sqrt{2\pi a^3}}\N^{(a)}_0\left(\int_{-D(x,y)}^{D(x,y)}F(\cS,x,y,t)dt\right)da.\]
   Using the definition of the biased Brownian sphere \eqref{def biais sphere}, this can be rewritten as
   \begin{equation}\label{calcul biais}
       \N_0\left(\int_{-D(x,y)}^{D(x,y)}F(\cS,x,y,t)dt\right)=\int_{\R_+}\frac{2}{2\sqrt{2\pi a^3}}\N_0^{(a)}(D(x,y))\mathbf{Unic}^{(a)}(F(\cS_b^{(a)},x,y,\Delta))da.
   \end{equation}
   However, for every $a>0$, the scaling property of the Brownian sphere and \cite[Corollary 2]{DelmasMoments} gives 
   \[\N_0^{(a)}(D(x,y))=a^{1/4}\N_0^{(1)}(D(x,y))=3\frac{(2a)^{1/4}\Gamma(5/4)}{\sqrt{\pi}}.\] Therefore, the expression \eqref{calcul biais} is equal to 
   \[ \N_0\left(\int_{-D(x,y)}^{D(x,y)}F(\cS,x,y,t)dt\right)=3\frac{\Gamma(5/4)}{2^{1/4}\pi}\int_{\R_+}\frac{da}{a^{5/4}}\mathbf{Unic}^{(a)}(F(\cS_b^{(a)},x,y,\Delta)).\]
   Using the identity $\Gamma(x+1)=x\Gamma(x)$ and Proposition \ref{decomposition volume}, we obtain 
   \[\N_0\left(\int_{-D(x,y)}^{D(x,y)}F(\cS,x,y,t)dt\right)=3\cdot\mathbf{Unic}(F(\cS_b,x,y,\Delta)),\]
   which concludes the proof of the first statement. The second statement follows from the first one by taking a function $F$ that does not depend on the delay.
\end{proof}

\subsection{Coding unicycles with point measures}\label{coding unicycles with point measures}

A useful alternative to the random process description of Brownian surfaces is to use triples $\left(X,\cM,\cM'\right)$, where $X$ is a random path and $\cM$ and $\cM'$ are two Poisson point measures. In this section, we explain how to construct a unicycle and a surface from such a coding triple. This construction is very similar to the one in \cite{Spinerepresentation}, which deals with triples that encode non-compact random surfaces. 

For the rest of the section, everything will be deterministic. Consider a coding triple $\left(w,\cP,\cP'\right)$, where : 
\begin{itemize}
    \item $(w(t))_{0\leq t\leq\sigma}\in\mathfrak{W}_0$ and $w(0)=w(\sigma)$, 
    \item $\cP=\sum_{i\in I}\delta_{t_i,\omega_i}$ and $\cP'=\sum_{j\in J}\delta_{t_j,\omega_j}$ are two point measures on $[0,\sigma]\times\mathfrak{S}$ (the sets $I$ and $J$ are disjoint) such that for every $i\in I\cup J$, $\omega_i\in\mathfrak{S}_{w(t_i)}$ and $\sigma(\omega_i)>0$,
    \item The numbers $t_i,\,i\in I\cup J$ are distinct.
    \item the functions 
    \begin{equation}\label{processus explo}
        s\mapsto \beta_s:=\sum_{i\in I}\1_{\{t_i\leq s\}}\sigma(\omega_i)\quad\text{ and }\quad s\mapsto\beta'_s:=\sum_{j\in J}\1_{\{t_j\leq s\}}\sigma(\omega_j)
    \end{equation}
    take finite values, and are monotone increasing on $[0,\sigma]$
    \item for every $\varepsilon>0$, we have 
    \[\#\left\{i\in I\cup J,\sup_{0\leq s\leq \sigma(\omega_i)}\left|\widehat{W}_s(\omega_i)-w(t_i)\right|>\varepsilon\right\}<\infty\]
\end{itemize}
We mention that these conditions already appeared in \cite{Spinerepresentation}.
The unicycle $\mathfrak{U}$ associated to the coding triple $\left(w,\cP,\cP'\right)$ is obtained from the disjoint union 
\[[0,\sigma]\sqcup\left(\bigsqcup_{i\in I\cup J}\cT_{(\omega_i)}\right),\]
where we identify the point $0$ with $\sigma$, and the point $t_i\in[0,\sigma]$ with the root $\rho_{(\omega_i)}$ of $\cT_{(\omega_i)}$. By convention, we will consider that the trees $\cT_{(\omega_i)},\,i\in I$ form the \textbf{external face} $\mathfrak{U}_{ext}$ of $\mathfrak{U}$, whereas the trees $\cT_{(\omega_j)},\,j\in J$ form the \textbf{internal face} $\mathfrak{U}_{int}$. We define a metric $d_\mathfrak{U}$ on $\mathfrak{U}$ as follows. 
First, the restriction of $d_\mathfrak{U}$ to a tree $\cT_{\omega_i}$ is the metric $d_{\cT_{\omega_i}}$. Similarly, for $s,t\in[0,\sigma]$ with $s\leq t$, we have $d_{\mathfrak{U}}=\inf\left(|s-t|,\sigma-|s-t|\right)$. If $u\in[0,\sigma]$ and $v\in\cT_{\omega_i}$, we set $d_\mathfrak{U}(u,v)=\inf(|u-t_i|,\sigma-|u+t_i|)+d_{\omega_i}(\rho_{\omega_i},v)$. Finally, if $u\in\cT_{\omega_i}$ and $v\in\cT_{\omega_j}$, we set $d_\mathfrak{U}(u,v)=d_{\omega_i}(\rho_{\omega_i},u)+d_\mathfrak{U}(t_i,t_j)+d_{\omega_j}(\rho_{\omega_j},v)$. We root the metric space $\mathfrak{U}$ at $0$. It is also equipped with a volume measure $\mu$, which is the sum of the volume of the trees $\cT_{(\omega_i)},\,I\cup J$.

It will be convenient to define exploration processes of $\mathfrak{U}$, as the concatenation of the exploration processes of the trees $\cT_{(\omega_i)}$. More precisely, we will define a process $\cE$ to explore the external face of $\mathfrak{U}$, and a process $\cE'$ to explore the internal face. For $s\in[0,\sigma]$, consider $\beta_s$ and $\beta'_s$ as defined in \eqref{processus explo}.
Then, if we set $\Delta=\sum_{i\in I}\sigma(\omega_i)$, for every $t\in[0,\Delta]$, we define $\cE_t\in\mathfrak{U}$ as follows.
Observe that there exists a unique $s\in[0,\sigma]$ such that $\beta_{s-}\leq t\leq \beta_s$. Then, 
\begin{itemize}[label=\textbullet]
    \item Either we have $s=t_i$ for some $i\in I$, and we set $\cE_t=p_{\cT_{\omega_i}}(t-\beta_{t_i-})$,
    \item or there is no such $i$, and we set $\cE_t=s$.
\end{itemize}
We define $\cE$ in a very similar manner, by replacing $\beta$ by $\beta'$, and the set $I$ by $J$. 

These exploration processes allow us to define intervals in $\mathfrak{U}$. First, we make the convention that if $s>t$, the interval $[s,t]$ is defined by $[s,\sigma]\cup[0,t]$. Let $u,v\in\mathfrak{U}$ be two elements that lie on the same face. First, suppose that either $u$ or $v$ does not belong to $[0,\sigma]$.
Without loss of generality, suppose that they are in the external face. Then, there exists a smallest interval $[s,t]$ with $s,t\in[0,\Delta]$ such that $\cE_s=u$ and $\cE_t=v$, and we define 
\[[u,v]=\{\cE_r,\,r\in[s,t]\}.\]
Note that, generally, we do not have $[u,v]=[v,u]$.     Then, if $s,t\in [0,\sigma]$, let $a,b\in[0,\Delta]$ (resp. $a',b'\in[0,\Delta']$), we set
\[[s,t]_{ext}=\{\cE_r,\,r\in[a,b]\}\quad\text{and}\quad [s,t]_{int}=\{\cE'_r,\,r\in[a',b']\}.\]

Next, we assign labels to the points of $\mathfrak{U}$. First, for $s\in[0,\sigma]$, we set $\ell_s=w(s)$ (note that it is consistent with the fact that $0$ and $\sigma$ are identified, since $w(0)=w(\sigma)$). Then, if $u\in\cT_{(\omega_i)}$, we set $\ell_u=\ell_u(\omega_i)$. Note that, when the conditions stated at the beginning of the section are fulfilled, the function $u\mapsto\ell_u$ is continuous on $\mathfrak{U}$.

Using these labels, we define a function $D^\circ $ as follows. First, for $s,t\in[0,\sigma]$, we set
\[D^\circ (s,t)=\ell_t+\ell_s-2\max\left(\inf_{w\in[s,t]_{ext}}\ell_w,\inf_{w\in[t,s]_{ext}}\ell_w,\inf_{w\in[s,t]_{int}}\ell_w,\inf_{w\in[t,s]_{int}}\ell_w\right).\] Then, if $u,v\in\mathfrak{U}_{ext}$ or $u,v\in\mathfrak{U}_{int}$, and one of them does not belong to $[0,\sigma]$, we set
\[D^\circ(u,v)=\ell_u+\ell_v-2\max\left(\inf_{w\in[u,v]}\ell_w,\inf_{w\in[v,u]}\ell_w\right).\]
Finally, if $u,v$ are not in the previous cases, we set $D^\circ(u,v)=\infty$.

    Using $D^\circ$, we define a pseudo-distance $D$ on $\mathfrak{U}$ by 
    \begin{equation*}        D(u,v)=\inf_{u_0,u_1,...,u_n}\sum_{i=1}^n D^\circ(u_{i-1},u_i),
    \end{equation*}
    where the infimum is taken over all the choices of $n\in\N$ and sequences $u_0,u_1,...,u_n$ of $\mathfrak{U}$ such that $u_0=u$ and $u_n=v$.
    Then, we define a metric space $\cS_{w,\cP,\cP'}$ by $\mathfrak{U}/\{D=0\}$, equipped with the distance $D $, and rooted at the equivalence class of $0$. Once again, this metric space naturally comes with a measure $\mathrm{Vol}$, which is the pushforward of the measure $\mu$ on $\mathfrak{U}$ by the canonical projection. 

\subsection{Controlling the delay}

The goal of this section is to construct the free biased Brownian sphere conditionally on the delay and the distance between the distinguished points.

First, we give an alternative definition of the free biased Brownian sphere with a coding triple. Consider a coding triple $\left(\e,\cM,\overline{\cM}\right)$, where
\begin{itemize}[label=\textbullet]
    \item $\e$ is distributed under Ito's excursion measure of Brownian motion,
    \item given $\e$, $\cM$ and $\overline{\cM}$ are two independent Poisson point measures on $[0,\sigma]\times\mathfrak{S}$, with intensity 
    \begin{equation}\label{intensity poisson}
        2\1_{[0,\sigma]}(t)dt\N_{\e_t}(dW).
    \end{equation}
\end{itemize}

We denote by $\U$ the measure associated to this triple. Note that this is not a probability measure, but a $\sigma$-finite measure. Note that under $\U$, this triple satisfies a.e. the conditions of Section \ref{coding unicycles with point measures}. As previously, a.e. there exists a unique $u_*\in\mathfrak{U}_{int}$ and a unique $\overline{u}_*\in\mathfrak{U}_{ext}$ such that 
\[\ell_{u_*}=\inf_{w\in\mathfrak{U}_{int}}\ell_w\quad\text{ and }\quad\ell_{\overline{u}_*}=\inf_{w\in\mathfrak{U}_{ext}}\ell_w.\]
Therefore, the metric space $\cS_{\e,\cM,\overline{\cM}}$ is naturally equipped with two distinguished points $x_*$ and $\overline{x}_*$, which are the images of $u_*$ and $\overline{u}_*$ by the canonical projection. We also set $\ell_*:=\ell_{u_*}$ and $\overline{\ell}_*:=\ell_{\overline{u}_*}$.

Our first result on this model is to view it as another construction of the biased Brownian sphere.

\begin{proposition}
    Under $\U$, the metric space $\cS_{\e,\cM,\overline{\cM}}$ is distributed as the free biased Brownian sphere $\cS_b$ under $\textbf{Unic}$.
\end{proposition}
\begin{proof}
    This follows from Itō's excursion theory of Brownian motion. Indeed, the excursions of $X|_{I_+}$ and $X|_{I_-}$ together with their labels (under \textbf{Unic})  can be represented by two independent Poisson point measures $\cM$, $\overline{\cM}$ with intensity \eqref{intensity poisson}. The result is obtained by comparing the constructions of Sections \ref{construction} and \ref{coding unicycles with point measures}.
\end{proof}

In view of that  result, we still denote the metric space obtained under $\U$ by $\cS_b$. This Poissonian representation of the free biased Brownian sphere is more suitable for explicit computations, as we will see in this section.

First, we recall a useful lemma, which will be used in the next proof. This lemma follows from Itô's formula (see \cite[Lemma 8.1]{rierathese} for a proof).

\begin{lemme}\label{Calcul espérance}
    Let $f:\R_+\rightarrow\R_+$ be a continuous function, and suppose that there exists a twice differentiable bounded positive function $H_f$ such that 
    \[H_f(x)''=2f(x)H(x).\]
    Then, we have
    \begin{equation}
        \E_h\left[\exp\left(-\int_0^{T_0}f(B_t)dt\right)\right]=\frac{H_f(h)}{H_f(0)}
    \end{equation} and 
    \begin{equation}
        \mathbf{n}\left(1-\exp\left(-\int_0^{\sigma}f(\mathbf{e}_t)dt\right)\right)=-\frac{H'_f(0)}{2H_f(0)}.
    \end{equation}
\end{lemme}

The following proposition gives a very simple formula for the distribution of both $\ell_*$ and $\overline{\ell}_*$, which is the counterpart of \eqref{inf} for the measure $\U$.

\begin{proposition}[Control of both minima]\label{Control min}
    For every $a,b>0$, we have
    \[\U(\ell_*<-a,\,\overline{\ell}_*<-b)=\frac{1}{a+b}.\]
\end{proposition}
\begin{proof}
    According to the description of the measure $\U$, we have 
    \begin{align*}
         \U\left(\ell_*<-a,\,\overline{\ell}_*<-b\right)&=\U\left(\exists i\in I,j\in J,\,(W_i)_*<-a\text{ and }(W_j)_*<-b\right)\\
         &=\n\left(\left(1-\exp\left(-3\int_0^\sigma\frac{1}{(\e_t+a)^2}dt\right)\right)\left(1-\exp\left(-3\int_0^\sigma\frac{1}{(\e_t+b)^2}dt\right)\right)\right)   
    \end{align*}       
  where the second line follows from the independence between $\cM$ and $\overline{\cM}$ given $\n$. This last expression can be reformulated as 
  \begin{multline*}
          \n\left(1-\exp\left(-3\int_0^\sigma\frac{1}{(\e_t+a)^2}dt\right)\right)+\n\left(1-\exp\left(-3\int_0^\sigma\frac{1}{(\e_t+b)^2}dt\right)\right)\\-\n\left(1-\exp\left(-3\int_0^\sigma\frac{1}{(\e_t+a)^2}+\frac{1}{(\e_t+b)^2}dt\right)\right).
  \end{multline*}
We can compute the first two terms with Lemma \ref{Calcul espérance}, by taking $H(x)=\frac{1}{(x+c)^2}$. Similarly, the last integral can be computed by taking $H(x)=\frac{1}{(x+a)^2}-\frac{1}{(x+b)^2}$ (if $a<b$). This gives
\begin{align*} 
\U(\ell_*<-a,\overline{\ell}_*<-b)&=\frac{1}{a}+\frac{1}{b}-\frac{\frac{1}{a^3}-\frac{1}{b^3}}{\frac{1}{a^2}-\frac{1}{b^2}}\\
&=\frac{1}{a+b}.
\end{align*}
\end{proof}

We also give the Laplace transform of the delayed Voronoï cells under $\U$ (which could have also been deduced from \ref{volume cells}). Indeed, just like in Section \ref{construction}, it is not hard to see that the $\left(\overline{\ell}_*-\ell_*\right)$-delayed Voronoï cells of $x_*$ and $\overline{x}_*$ in $\cS$ are sets 
$p_{\cS}\left(\mathfrak{U}_{int}\right)$ and $p_{\cS}\left(\mathfrak{U}_{ext}\right)$.
Therefore, we set 
\[\sigma_1=\mathrm{Vol}\left(p_{\cS}\left(\mathfrak{U}_{int}\right)\right)\quad\text{ and }\quad\sigma_2=\mathrm{Vol}\left(p_{\cS}\left(\mathfrak{U}_{ext}\right)\right).\]

\begin{proposition}[Laplace transform of the volumes]
    For every $\lambda_1,\lambda_2>0$, we have 
    \[\U\left(1-\exp(-\lambda_1\sigma_1-\lambda_2\sigma_2)\right)=2^{-1/4}\left(\sqrt{\lambda_1}+\sqrt{\lambda_2}\right)^{1/2}.\]
\end{proposition}
\begin{proof}
    By the definition of the measure $\n$, we have 
    \begin{align*}
        \U\left(1-\exp(-\lambda_1\sigma_1-\lambda_2\sigma_2)\right)&=\int_0^\infty\frac{dx}{2\sqrt{2\pi x^3}}\E\left[1-\exp(-\lambda_1 T_x-\lambda_2 T'_x)\right]\\
        &=\int_0^\infty\frac{dx}{2\sqrt{2\pi x^3}}\left(1-\E\left[\exp(-\lambda_1 T_x\right]\E\left[\exp(-\lambda_2 T'_x)\right]\right),
    \end{align*}
    where $T_x$ and $T'_x$ are the first hitting time of $-x$ by two independent Brownian motions starting at $0$. Using the well known formula 
    \[\E\left[\exp(-\lambda T_x)\right]=\exp(-x\sqrt{2\lambda}),\]
    an integration by part gives
    \begin{align*}
          \U\left(1-\exp(-\lambda_1\sigma_1-\lambda_2\sigma_2)\right)&=\int_0^\infty\frac{1}{2\sqrt{2\pi x^3}}\left(1-\exp(-x(\sqrt{2\lambda_1}+\sqrt{2\lambda_2})\right)dx\\
          &=\int_0^\infty\frac{\sqrt{2\lambda_1}+\sqrt{2\lambda_2}}{\sqrt{2\pi x}}\left(\exp(-x(\sqrt{2\lambda_1}+\sqrt{2\lambda_2})\right)dx\\
    \end{align*}
    Finally, using some classical identities, we obtain 
    \[\U\left(1-\exp(-\lambda_1\sigma_1-\lambda_2\sigma_2)\right)=2^{-1/4}\left(\sqrt{\lambda_1}+\sqrt{\lambda_2}\right)^{1/2},\]
    which concludes the proof.
\end{proof}

Before stating our disintegration result, we need to construct a family of probability measures $\left(\U^{(x,y)}\right)_{x,y\geq 0}$. As we will show in Proposition \ref{Decompo minimum}, this probability measure corresponds to the measure $\U$ ``conditioned to have $\ell_*=-x$ and $\overline{\ell}_*=-y$''. 

Fix $x,y\geq 0$. First, consider a random variable $\tilde{\e}^{(x,y)}$ defined by the following formula
\begin{equation}\label{prob measure}
    \E[F(\tilde{\e}^{(x,y)})]=18(x+y)^3\n\left(F(\e)\left(\int_0^\sigma\int_0^\sigma\frac{dsdt}{(\e_s+x)^3(\e_t+y)^3}\right)\exp\left(-3\int_0^\sigma\left(\frac{1}{(\e_u+x)^2}+\frac{1}{(\e_u+y)^2}\right)du\right)\right).
\end{equation}
Note that this is a well-defined probability measure. Indeed, by Proposition \ref{Control min}, we have 
\begin{align*}
    &18(x+y)^3\n\left(\left(\int_0^\sigma\int_0^\sigma\frac{dsdt}{(\e_s+x)^3(\e_t+y)^3}\right)\exp\left(-3\int_0^\sigma\left(\frac{1}{(\e_u+x)^2}+\frac{1}{(\e_u+y)^2}\right)du\right)\right)\\
    &=\frac{(x+y)^3}{2}\frac{\partial^2}{\partial x\partial y}\n\left(\left(1-\exp\left(-3\int_0^\sigma\frac{1}{(\e_t+x)^2}dt\right)\right)\left(1-\exp\left(-3\int_0^\sigma\frac{1}{(\e_t+y)^2}dt\right)\right)\right)   \\
    &=\frac{(x+y)^3}{2}\frac{\partial^2 \U(\ell_*<-x,\overline{\ell}_*<-y)}{\partial x\partial y}\\
    &=1.
\end{align*}
Then, given $\tilde{\e}^{(x,y)}$, consider two random independent Poisson point measures $\cM_{>-x}$ and $\overline{\cM}_{>-y}$ with respective intensities 
\[2\1_{\{t\in[0,\sigma]\}}\1_{\{W_*>-x\}}dt\N_{\tilde{\e}^{(x,y)}_t}(dW)\quad\text{ and }\quad2\1_{\{t\in[0,\sigma]\}}\1_{\{W_*>-y\}}dt\N_{\tilde{\e}^{(x,y)}_t}(dW).\]
Moreover, given $\tilde{\e}^{(x,y)}$, consider two random point measures $\delta_{t_x,\omega^{(x)}}$ and $\delta_{\overline{t}_y,\overline{\omega}^{(y)}}$ made of a single atom, independent of each other and of $\cM_{>-x},\,\overline{\cM}_{>-y}$, such that for every continuous bounded function $F$,
\begin{align*}    
\E\left[F\left(t_x,\omega^{(x)},\overline{t}_y,\overline{\omega}^{(x)}\right)\bigg|\,\tilde{\e}^{(x,y)}\right]&=\frac{1}{\left(\int_0^\sigma\int_0^\sigma\frac{dsdt}{(\e_s+x)^3(\e_t+y)^3}\right)}\bigg(\int_{[0,\sigma]\times[0,\sigma]}dsdt\\&\int_{\mathfrak{S}\times\mathfrak{S}} \N_{\e_s}(dW|W_*=-x)\N_{\e_s}(dW'|W'_*=-y)\frac{F(s,W,t,W')}{(\e_s+x)^3(\e_t+y)^3}\bigg).
\end{align*}
Finally, set
\[\cM_{-x}:=\cM_{>-x}+\delta_{t_x,\omega^{(x)}}\quad\text{and}\quad\overline{\cM}_{-y}:=\overline{\cM}_{>-y}+\delta_{\overline{t}_y,\overline{\omega}^{(y)}}.\]
\begin{definition}\label{conditioned measure}
    For every $x,y>0$, the probability measure $\U^{(x,y)}$ is defined by the formula 
\[\U^{(x,y)}\left(F\left(\e,\cM,\overline{\cM}\right)\right)=\E\left[F\left(\tilde{\e}^{(x,y)},\cM_{-x},\overline{\cM}_{-y}\right)\right].\]
\end{definition}
The random metric space associated to the coding triple $\left(\e,\cM,\overline{\cM}\right)$ is denoted by $\left(\cS^{(x,y)},D^{(x,y)}\right)$. As for the measure $\U$, these metric spaces are equipped with two distinguished points $x_*$ and $\overline{x}_*$. Let us mention that these surfaces satisfy the following scaling property : for every $x,y>0$ and $\lambda>0$, we have 
\begin{equation}\label{scaling conditioned}
    \left(\cS^{(x,y)},\lambda\cdot D^{(x,y)}\right)\stackrel{(d)}{=}  \left(\cS^{(\lambda x,\lambda y)}, D^{(\lambda x,\lambda y)}\right).
\end{equation}
This property follows from the scaling properties of the processes in the coding triple. 
\begin{proposition}[Disintegration with respect to both minima]\label{Decompo minimum}
    Let $F$ and $G$ be measurable non-negative functions. Then
    \[\U\left(F(-\ell_*,-\overline{\ell}_*)G(\e,\cM,\overline{\cM})\right)=\int_{\R_+\times \R_+}F(x,y)\frac{2}{(x+y)^3}\U^{(x,y)}\left(G(\e,\cM,\overline{\cM})\right)dxdy.\]
\end{proposition}
\begin{proof}
For every $F$ and $G$ measurable non-negative functions, we have :
    \begin{align*}
        \U\left(F(-\ell_*,-\overline{\ell}_*)G(\e,\cM,\overline{\cM})\right)&
        =\n\left(\Pi_\e\left(F(-\ell_*,-\overline{\ell}_*)G(\e,\cM,\overline{\cM})\right)\right)\\
       &=\n\left(\Pi_\e\left(\sum_{i\in I,j\in J}F(-W_*^{(i)},-W_*^{(j)})\1_{\{W_*^{(i)}=\ell_*,W_*^{(j)}=\overline{\ell}_*\}}G(\e,\cM,\overline{\cM})\right)\right)
    \end{align*}
where under the probability measure $\Pi_\e$, $\cM$ and $\overline{\cM}$ are two independent Poisson point measure, with intensities given by \eqref{intensity poisson}. By Palm's formula, this last term is equal to 
\begin{align*}
    &4\n\bigg(\int_0^\sigma\int_0^\sigma \int_{\mathfrak{S}\times\mathfrak{S}}dsdt\N_{\e_s}(dW)\N_{\e_t}(dW')F(-W_*,-W'_*)\\
    &\qquad\qquad\qquad\qquad\qquad\qquad\times\Pi_\e\left(G(\e,\cM+\delta_{s,W},\overline{\cM}+\delta_{t,W'})\1_{\{\inf W_*^{(i)}>\ell_*,\inf W_*^{(j)}>\overline{\ell}_*\}}\right)\bigg)\\
    &=4\n\bigg(\int_0^\sigma\int_0^\sigma \int_{\mathfrak{S}\times\mathfrak{S}}dsdt\N_{\e_s}(dW)\N_{\e_t}(dW')F(-W_*,-W'_*){\Pi}^{(\ell_*,\overline{\ell}_*)}_\e\left(G(\e,{\cM}+\delta_{s,W},{\overline{\cM}}+\delta_{t,W'})\right)\\
    &\qquad\qquad\qquad\qquad\qquad\qquad\times\Pi_\e\left(\inf W_*^{(i)}>\ell_*,\inf W_*^{(j)}>\overline{\ell}_*\right)\bigg)
\end{align*}

where, conditionally on $\e$ and under ${\Pi}^{(\ell_*,\overline{\ell}_*)}_\e$, ${\cM}$ and ${\overline{\cM}}$ are two independent Poisson point measures with respective intensities
\[2\1_{[0,\sigma]}\1_{W_*>\ell_*}ds\,\N_{\e_s}(dW)\]
and 
\[2\1_{[0,\sigma]}\1_{W_*>\overline{\ell}_*}ds\,\N_{\e_s}(dW).\]
Then, using the fact that 
\[\N_0(F(W))=\int_0^\infty\frac{3}{x^3}\N_0(F(W)|W_*=-x)dx,\]
we obtain 
\begin{multline}
 \U\left(F(-\ell_*,-\overline{\ell}_*)G(\e,\cM,\overline{\cM})\right)=36\n\Bigg(\int_{\R_+\times\R_+}dsdt\int_{[0,\sigma]\times[0,\sigma]}\dx\dy\\\int_{\mathfrak{S}\times\mathfrak{S}}\frac{\N_{\e_s}(dW|W_*=-x)\N_{\e_t}(dW'|W'_*=-y)}{(\e_t+x)^3(\e_s+y)^3}F(x,y)
 \Phi_{G,\e,x,y}(s,W,t,W')\Bigg)
 \end{multline}
where 
\begin{equation}
    \Phi_{G,\e,x,y}(s,W,t,W')={\Pi}^{(-x,-y)}_\e\left(G(\e,{\cM}+\delta_{s,W},{\overline{\cM}}+\delta_{t,W'})\right)\exp\left(-3\int_0^\sigma\frac{1}{(\e_t+x)^2}+\frac{1}{(\e_t+y )^2}dt\right).
\end{equation} Finally, using Fubini theorem, we have
\[\U\left(F(-\ell_*,-\overline{\ell}_*)G(\e,\cM,\overline{\cM})\right)=\int_{\R_+\times\R_+}dxdyF(x,y)\n\left(\Psi(G,\e,x,y)\right),\]
where \[\Psi(G,\e,x,y)=\int_0^\sigma\int_0^\sigma\int_{\mathfrak{S}\times\mathfrak{S}}dsdt\frac{36\N_{\e_t}(dW|W_*=-x)\N_{\e_s}(dW'|W'_*=-y)}{(\e_t+x)^3(\e_s+y)^3}\Phi_{G,\e,x,y}(s,W,t,W').\]
Now, by comparing this formula with the definition of the measure $\U^{(x,y)}$, we get

\[\U\left(F(-\ell_*,-\overline{\ell}_*)G(\e,\cM,\overline{\cM})\right)=\int_{\R_+\times \R_+}F(x,y)\frac{2}{(x+y)^3}\U^{(x,y)}\left(G(\e,\cM,\overline{\cM})\right)dxdy,\] which is the desired result. 
\end{proof}

\subsection{Link with the unbiased Brownian sphere}

The purpose of this section is to relate the random metric spaces $\cS^{(x,y)}$ to the Brownian sphere \textit{without bias}. More precisely, we will prove that the random surfaces constructed under the probability measures $\U^{(x,y)}$ and $\N_0(\cdot\,|\,W_*=-x-y)$ have the same distribution. This result yields a new construction of the Brownian sphere with random volume and two points at a fixed distance.

\begin{theorem}\label{Representation libre}
    For every $a>0$, $0\leq t \leq a$ and every non-negative function $F$, 
    \[\N_0\left(F(\cS,x,y)\,|\, W_*=-a\right)=\U^{(t,a-t)}\left(F(\cS,x_*,\overline{x}_*)\right).\]
\end{theorem}
\begin{proof}
    By Propositions \ref{link free model} and \ref{Decompo minimum}, for every non-negative functions $F,G$ and $H$, we have 
    \[\N_0\left(\int_{-D(x,y)}^{D(x,y)}F(\cS,x,y)G(t)H(D(x,y))dt\right)=\int_{\R_+\times \R_+}\frac{6}{(b+b')^3}\U^{(b,b')}\left(F(\cS,x_*,\overline{x}_*)G(b-b')H(b+b')\right)dbdb'.\]
    By disintegrating the measure $\N_0$ according to $W_*$, and after rearranging the term on the right-hand side, we obtain
    \begin{align*}
       & \int_0^\infty\frac{3}{a^3}\left(\int_{-a}^{a}\N_0(F(\cS,x,y)\,|\,W_*=-a)G(t)H(a)dt\right)da\\&=\int_{0}^\infty\frac{6}{a^3}\left(\int_{0}^a\U^{(t,a-t)}\left(F(\cS,x_*,\overline{x}_*)\right)G(a-2t)H(a)dt\right)da\\
        &=\int_{0}^\infty\frac{3}{a^3}\left(\int_{-a}^a\U^{((a-t)/2,(a+t)/2)}\left(F(\cS,x_*,\overline{x}_*)\right)G(t)H(a)dt\right)da.
    \end{align*}     
     This means that for a.e $a\geq0$ and $0\leq t\leq a$, for every non-negative function $F$, we have 
     \[\N_0\left(F(\cS,x,y)\,|\, W_*=-a\right)=\U^{(t,a-t)}\left(F(\cS,x_*,\overline{x}_*)\right).\]
    In this situation, one can often rely on a scaling argument to obtain the equality for every parameter. However, since we have two parameters and that every change of scale affects both of them in the same way, this method does not apply in our situation. Instead, we will use a continuity argument.

    We will show that for every $a,b>0$ and $(a_n,b_n)_{n\in\N}\in \R_+^\N$ such that $a_n\rightarrow a$ and $b_n\rightarrow b$, it holds that 
    \begin{equation}\label{continuite}
        \cS^{(a_n,b_n)}\xrightarrow[n\rightarrow \infty]{(d)}\cS^{(a,b)}.
    \end{equation}
    Note that even though it is clear that for every continuous function, we have 
    \[\U^{(a_n,b_n)}\left(F(\e,\cM,\overline{\cM})\right)\xrightarrow[n\rightarrow\infty]{}\U^{(a,b)}\left(F(\e,\cM,\overline{\cM})\right),\]
    this does not directly imply that the associated surfaces also converge in distribution.

    Fix $a,b>0$. We will show that for every $\varepsilon>0$, there exists $\delta>0$ such that for every $a',b'>0$ with $|a-a'|+|b-b'|\leq\delta$, we can couple $\cS^{(a',b')}$ and $\cS^{(a,b)}$ so that with probability at least $1-\varepsilon$, there exists a bijection $\Phi: \cS^{(a,b)}\rightarrow \cS^{(a',b')}$ such that for every $x,y\in\cS^{(a,b)}$, we have 
    \[(1-\varepsilon)D^{(a',b')}(\Phi(x),\Phi(y))\leq D^{(a,b)}(x,y)\leq (1+\varepsilon)D^{(a',b')}(\Phi(x),\Phi(y)).\] Since this result implies \eqref{continuite}, this will complete the proof.

Fix $\varepsilon>0$, and consider the processes $A(a,b):=\left(\tilde{\e}^{(a,b)},\cM_{>-a}+\delta_{t_a,\omega^{(a)}},\overline{\cM}_{>-b}+\delta_{\overline{t}_b,\overline{\omega}^{(b)}}\right)$ that encode $\cS^{(a,b)}$. Using the explicit densities that define these random variables, one can see that there exists $\delta>0$ such that for every $a',b'>0$ with $|a-a'|+|b-b'|\leq\delta$, we can couple $\left(\tilde{\e}^{(a,b)},\cM_{>-a},t_a,\overline{\cM}_{>-b},\overline{t}_b\right)$ and $\left(\tilde{\e}^{(a',b')},\cM_{a'},t_{a'},\overline{\cM}_{b'},\overline{t}_{b'}\right)$ in such a way that these random variables are equal with probability at least $1-\varepsilon$ (we omit the details, but we mention that similar proofs will be done in Section \ref{local structure}). We call this event $\cH_\delta$. On this event, note that given $\left(\tilde{\e}^{(a,b)},\cM_{>-a},t_a,\overline{\cM}_{>-b},\overline{t}_b\right)$, the law of $\omega^{(a)}$ and $\overline{\omega}^{(b)}$ are respectively 
\[\N_{\tilde{\e}^{(a,b)}_{t_a}}(\cdot\,|\,W_*=-a)\quad\text{ and }\quad\N_{\tilde{\e}^{(a,b)}_{\overline{t}_b}}(\cdot\,|\,W_*=-b).\]
Set 
\[g_a(a')=\left(\frac{a'+\tilde{\e}^{(a,b)}_{t_a}}{a+\tilde{\e}^{(a,b)}_{t_a}}\right)^2\quad\text{ and }\quad g_b(b')=\left(\frac{b'+\tilde{\e}^{(a,b)}_{\overline{t}_b}}{b+\tilde{\e}^{(a,b)}_{\overline{t}_b}}\right)^2.\]
Then, using the scaling properties of the measure $\N$, we see that $\left(\omega^{(a')},\overline{\omega}^{(b')}\right)$
have the same distribution as 
\begin{equation*}\label{dilate}
    \left(\Theta_{g_a(a')}\left(\omega^{(a)}-\tilde{\e}^{(a,b)}_{t_a}\right)+\tilde{\e}^{(a,b)}_{t_a},\Theta_{g_b(b')}\left(\overline{\omega}^{(b)}-\tilde{\e}^{(a,b)}_{\overline{t}_b}\right)+\tilde{\e}^{(a,b)}_{\overline{t}_b}\right).
\end{equation*}
Moreover, we can suppose that $1-\varepsilon\leq g_a(a')\leq 1+\varepsilon$ and $1-\varepsilon\leq g_b(b')\leq 1+\varepsilon$ (we can still reduce $\delta$ if it is not the case).
Hence, we can use this formula to obtain a coupling of $\left(\omega^{(a)},\overline{\omega}^{(b)}\right)$ and $\left(\omega^{(a')},\overline{\omega}^{(b')}\right)$. More precisely, conditionally on $\cH_\delta$, we can couple these random variables in such a way that almost surely, $\left(\omega^{(a')},\overline{\omega}^{(b')}\right)$ is equal to \eqref{dilate}. Together with the previous coupling, we have constructed a coupling of $A(a,b)$ and $A(a',b')$. Let $\mathfrak{U}_{a,b}$ and $\mathfrak{U}_{a',b'}$ be the random unicycles obtained from this coupling. By construction, on $\cH_\delta$, we have a bijection $\Phi:\mathfrak{U}_{a,b}\rightarrow\mathfrak{U}_{a',b'}$ which is the identity map everywhere except on the trees encoded by $\left(\omega^{(a)},\overline{\omega}^{(b)}\right)$ and $\left(\omega^{(a')},\overline{\omega}^{(b')}\right)$, where it acts as a dilatation. Moreover, it is clear that $\Phi$ preserves intervals, and that on $\cH_\varepsilon$, for every $u\in\mathfrak{U}_{a,b}$, we have
\[(1-\varepsilon)\ell_{\Phi(u)}\leq\ell_u\leq(1+\varepsilon)\ell_{\Phi(u)}.\] Since $\Phi$ preserves intervals, this implies that for every $u,v\in\mathfrak{U}_{a,b}$
\[(1-\varepsilon)D^{(a',b')}(\Phi(x),\Phi(y))\leq D^{(a,b)}(x,y)\leq (1+\varepsilon)D^{(a',b')}(\Phi(x),\Phi(y)),\]
which concludes the proof.
\end{proof}

As a direct consequence, we obtain the following counterpart of Corollary \ref{distance frontiere} for the (non-biased) Brownian sphere.

\begin{corollary}\label{dist boundary}
    For any $\alpha>0$, consider a Brownian Sphere with two distinguished points $(\cS,x,y)$ conditioned on $D(x,y)=\alpha$. Then, for any $\beta\in(-\alpha,\alpha)$, the distances from $x$ and $y$ to the boundary of their $\beta$-delayed Voronoï cells evolve, up to deterministic additive constants, as the random variable $\tilde{\e}^{(\frac{\alpha+\beta}{2},\frac{\alpha-\beta}{2})}$.
\end{corollary}

\begin{proof}
    This readily follows from the same arguments as in the proof of Corollary \ref{distance frontiere}, together with Definition \ref{conditioned measure} and Theorem \ref{Representation libre}.  
\end{proof}

\section{Application : a new construction of the bigeodesic Brownian plane}\label{section bigeodesique}

In this section, we use Corollary \ref{Representation libre} to obtain a new construction of the bigeodesic Brownian plane. 

Recall that the bigeodesic Brownian plane is the local limit of the Brownian sphere around a point of its unique geodesic between two distinguished points $x_0$ and $x_*$ (or around a point of any ``typical geodesic''). More precisely, by \cite[Theorem 5.11]{Bigeodesicbrownianplane}, we have the following convergence
\begin{equation}   (\cS,\varepsilon^{-1}D,\Gamma(1))\xrightarrow[\varepsilon\rightarrow 0]{(d)}(\overline{\mathcal{BP}},\overline{D}_\infty,\overline{\rho}_\infty)
\end{equation}
    in distribution for the local Gromov-Hausdorff topology, where the left-hand side is distributed under $\N_0(\cdot\,|\,W_*=-2)$. We mention that the choices of $\Gamma(1)$ and $W_*=-2$ are totally arbitrary, and the same result holds with $\Gamma(s)$ and $W_*=-t$ as long as $0<s<t$. This random space has an explicit construction in terms of labelled trees, but we do not need it in this paper. However, by Theorem \ref{Representation libre}, the same convergence holds if we replace the measure $\N_0(\cdot\,|\,W_*=-2)$ by $\U^{(1,1)}$. Our goal is to study $\cS$ (under $\U^{(1,1)}$) around $\Gamma(1)$ to obtain a new construction of the limiting object. The proof is composed of three steps :
    \begin{itemize}[label=\textbullet]
        \item First, we present a new random surface $\cS_\infty$, which is our candidate for the new construction of the bigeodesic Brownian plane.
        \item Then, we prove that we can couple the processes that encode $\cS$ and $\cS_\infty$ in such a way that they are locally equal with high probability.
        \item Finally, we show that this coupling induces a coupling of balls in $\cS$ and $\cS_\infty$, which implies the local Gromov-Hausdorff-Prokhorov convergence.
    \end{itemize}

\subsection{The construction}\label{section construction bigeo}

In this section, we present the new construction of the bigeodesic Brownian plane. This construction is essentially the ``infinite volume version'' of the construction given in Section \ref{coding unicycles with point measures}. The equivalence with the definition used in \cite{Bigeodesicbrownianplane} will be proved in Theorem \ref{equivalence definition}. 

Consider a random process $(X_t)_{t\in\R}$ such that $(X_t)_{t\geq0}$ and $(X_{-t})_{t\geq0}$ are two independent Bessel processes of dimension $3$ started from $0$. Then, conditionally on $X$, let $\cM_\infty$ and $\overline{\cM}_\infty$ be two independent Poisson point measures on $\R\times\cS$ with intensity 
\[2dt\N_{X_t}(dW).\]
Let $I$ and $J$ be two disjoint sets indexing the atoms of these measures, so that we write 
\[\cM_\infty=\sum_{i\in I}\delta_{t_i,\omega_i}\quad\text{ and }\quad\overline{\cM}_\infty=\sum_{j\in J}\delta_{t_j,\omega_J}.\]
Consider the random space $\mathfrak{U}_\infty$ obtained from the disjoint union 
\[\R\sqcup\left(\bigsqcup_{i\in I}\cT_{\omega_i}\right)\sqcup\left(\bigsqcup_{j\in J}\cT_{\omega_j}\right),\]
where we identify the root $\rho_{\omega_i}$ of the tree $\cT_{\omega_i}$ with $t_i\in\R$. We define a metric $d_\mathfrak{U}$ on $\mathfrak{U}$ as follows. 
First,  the restriction of $d_\mathfrak{U}$ to $\R$ is the Euclidean metric, and the restriction of $d_\mathfrak{U}$ to a tree $\cT_{\omega_i}$ is the metric $d_{\cT_{\omega_i}}$. If $u\in\R$ and $v\in\cT_{\omega_i}$, we set $d_\mathfrak{U}(u,v)=|u-t_i|+d_{\omega_i}(\rho_{\omega_i},v)$. Finally, if $u\in\cT_{\omega_i}$ and $v\in\cT_{\omega_j}$, we set $d_\mathfrak{U}(u,v)=d_{\omega_i}(\rho_{\omega_i},u)+|t_i-t_j|+d_{\omega_j}(\rho_{\omega_j},v)$. The metric space $\mathfrak{U}_\infty$ is also equipped with a volume measure $\overline{\mu}$, which is the sum of the volume measure of the trees.

We also assign labels to the points of $\mathfrak{U}$. If $u\in\R$, we set $\ell_u=X_u$. Otherwise, if $u\in\cT_{\omega_i}$, we set $\ell_u=X_{t_i}+\ell_u(\omega_i)$. As in Section \ref{coding unicycles with point measures}, we also define two exploration processes $(\cE_t)_{t\in\R}$ and $(\overline{\cE}_t)_{t\in\R}$. Once again, these exploration processes allow us to define intervals in $\mathfrak{U}_\infty$. When $s,t\in\R\subset\mathfrak{U}_\infty$, we use the notations $[s,t]$ and $\overline{[s,t]}$ for the different intervals induced by $\cE$ and $\overline{\cE}$.
We also define a function $D^\circ_\infty$ on $\mathfrak{U}_\infty$ exactly as we did in Section \ref{coding unicycles with point measures}. Finally, we define a pseudo-distance $D_\infty$ on $\mathfrak{U}$ by 
    \begin{equation*}        D_\infty(u,v)=\inf_{u_0,u_1,...,u_n}\sum_{i=1}^nD_\infty^\circ(u_{i-1},u_i),
    \end{equation*}
    where the infimum is taken over all the choices of $n\in\N$ and sequences $u_0,u_1,...,u_n$ of $\mathfrak{U}_\infty$ such that $u_0=u$ and $u_n=v$.
    Then, we define a random metric space $\cS_\infty$ by $\mathfrak{U}_\infty/\{D_\infty=0\}$, equipped with the distance $D_\infty$, and rooted at the equivalence class of $0$. This space also comes with a volume measure $\overline{\mathrm{Vol}}$, which is the pushforward of the volume $\overline{\mu}$ by the canonical projection $p_{\cS_\infty}:\mathfrak{U}_\infty\rightarrow\cS_\infty$.  

    Note that this space is scale-invariant, meaning that for every $\lambda>0$, we have the equality in distribution 
    \[\left(\cS_\infty,\lambda\cdot D_\infty,\rho_\infty\right)=^{(d)}\left(\cS_\infty,D_\infty,\rho_\infty\right).\]
This property follows from the scale invariance properties of the processes that encode $\cS_\infty$. 

The rest of the section is devoted to the proof of the following result.

\begin{theorem}\label{equivalence definition}
        The random metric space $\left(\cS_\infty,D_\infty,p_{\cS_\infty}(0)\right)$ is distributed as the bigeodesic Brownian plane. 
    \end{theorem}
  
    Let us identify the unique bi-geodesic in this construction. For every $t\geq0$, we define 
    \[\tau_t=\inf\{s\geq0,\ell_{\cE_s}=-t\},\quad\tau_{-t}=\sup\{s\leq0,\ell_{\cE_s}=-t\}\]
    and 
    \[\overline{\tau}_t=\inf\{s\geq0,\ell_{\overline{\cE}_s}=-t\},\quad\overline{\tau}_{-t}=\sup\{s\leq0,\ell_{\overline{\cE}_s}=-t\}.\]
    Then, for every $t\in\R$, set
    \[\theta_t=\cE_{\tau_t}\quad\text{ and }\quad\overline{\theta}_t=\overline{\cE}_{\overline{\tau}_t}.\]
    These elements can be seen as the first hitting times of $-t$ in $\mathfrak{U}_\infty$, each one corresponding to one exploration direction of a side of $\mathfrak{U}$. In particular, $\theta_0=\overline{\theta}_0=0$. Note that for every $t\geq0$, $D_\infty^\circ(\theta_t,\theta_{-t})=0$ and $D_\infty^\circ(\overline{\theta}_t,\overline{\theta}_{-t})=0$.
    Finally, for every $t\in\R$, we define 
    \begin{equation*}
        \gamma_\infty(t)=
         \left\{\begin{array}{ll}
  p_{\cS_\infty}(\theta_t)\quad&\text{ if $t\geq0$},\\
     p_{\cS_\infty}(\overline{\theta}_t)\quad&\text{ if $t\leq0$}
    \end{array}\right.
    \end{equation*}
    Let us check that $\gamma_\infty$ is an infinite bigeodesic. First, if $s,t\geq0$, we have 
    \[D_\infty(\gamma_\infty(s),\gamma_\infty(t))\geq\left|\ell_{\gamma_\infty(s)}-\ell_{\gamma_\infty(t)}\right|=|s-t|.\]
    On the other hand, since $D_\infty\leq D^\circ_\infty$, we also have 
    \[D_\infty(\gamma_\infty(s),\gamma_\infty(t))\leq D_\infty^\circ(\theta_s,\theta_t)=|s-t|.\]
    Of course the same result holds if $s,t\leq0$. Then, if $s\leq0\leq t$, we have
\[D(\gamma_\infty(t),\gamma_\infty(s))\leq D(\gamma_\infty(t),\gamma_\infty(0))+D(\gamma_\infty(0),\gamma_\infty(s))=|t-s|.\]
Finally, as in Proposition \ref{frontier}, one can show that every continuous curve from a point $x\in p_{\cS_\infty}\left(\cM_\infty\right)$ to a point $y\in p_{\cS_\infty}\left(\overline{\cM}_\infty\right)$ must intersect the path $\left(p_{\cS_\infty}(s)\right)_{s\in\R}$. We omit the proof of this result, since it is similar to the one of Proposition \ref{frontier}, except that some additional care is needed since the space is not compact in this case (a similar result will be proved later, see Proposition \ref{a l'infini}).

Consequently, we also have 
\begin{align*}    
D(\gamma_\infty(t),\gamma_\infty(s))&=\inf_{h\in\R}\left\{D(\gamma_\infty(t),p_{\cS_\infty}(h))+D(p_{\cS_\infty}(s),\gamma_\infty(s))\right\}\\
&\geq\inf_{h\in\R}\{|X_h-t|+|X_h-s|\}\\
&=|t-s|,
\end{align*} which proves the claim. 

\subsection{Local structure of the unicycle}\label{local structure}

This section is devoted to the study of the unicycle $\mathfrak{U}$ (under $\U^{(1,1)}$) around some specific point. Recall that by Theorem \ref{Representation libre}, the random unicycle $\mathfrak{U}$ encodes a Brownian sphere with two points at distance $2$. Therefore, as mentioned in Section \ref{brownian sphere}, there exists a unique geodesic $\Gamma$ between $x_*$ and $\overline{x}_*$.

Since our goal is to study $\cS$ around $\Gamma(1)$, we need to understand which point of $\mathfrak{U}$ corresponds to $\Gamma(1)$. 

\begin{proposition}
    We have $p_\cS^{-1}(\Gamma(1))=\{0\}.$ 
\end{proposition}
\begin{proof}
    As explained at the end of Section \ref{construction}, for every $s\in [0,\sigma]$, $x= p_{\cS}({s}),$ we have 
\begin{equation}\label{cycle distance}
        D(x,x_*)=D(x,\overline{x}_*)=\e_s+1.
\end{equation}
Furthermore, as in Proposition \ref{frontier}, we know that $\Gamma$ must intersect $p_{\cS}([0,\sigma])$. Therefore, since $D(x_*,\overline{x}_*)=2$, we deduce from \eqref{cycle distance} that 
\[\Gamma\cap p_{\cS}([0,\sigma])=\{\Gamma(1\})=\{p_{\cS}(0)\}.\]
Moreover, as in Proposition \ref{frontiere}, it is clear that $0$ is the only antecedent of $\Gamma(1)$ by $p_\cS^{-1}$, which concludes the proof.
\end{proof}

Our next goal is to compare the local behavior of $\mathfrak{U}$ around $0$ to the local behavior of $\mathfrak{U}_\infty$ around $0$. Recall the definition of the triple $(\e,\cM,\overline{\cM})$ under $\U^{(1,1)}$ from Definition \ref{conditioned measure}. To simplify notations, we set $\tilde{\e}:=\tilde{\e}^{(1,1)}$. 

We also define a truncated version of the random variable $\tilde{\e}$. For every $\varepsilon>0$, set  
\[T_\varepsilon=\inf\{t\geq0,\,\e_t=\varepsilon\}\quad\text{ and }\quad L_\varepsilon=\sup\{t\geq0,\,\e_t=\varepsilon\}.\]
For every $\varepsilon\geq 0$, we define a random variable $\tilde{\e}^{(\varepsilon)}$ by the formula 
\begin{equation}\label{excursion tronquee}
\E[F(\tilde{\e}^{(\varepsilon)})]=\frac{1}{Z_\varepsilon}\n\left(F(\e)\1_{\{\sup \e>\varepsilon\}}\left(\int_{T_{\varepsilon}}^{L_{\varepsilon}}\int_{T_{\varepsilon}}^{L_{\varepsilon}}\frac{dsdt}{(\e_s+1)^3(\e_t+1)^3}\right)\exp\left(-6\int_{T_{\varepsilon}}^{L_{\varepsilon}}\frac{du}{(\e_u+1)^2}\right)\right)
\end{equation}
where $Z_\varepsilon$ is a normalizing constant. As we will see in the following proof, these random variables are well-defined, i.e. $Z_\varepsilon<\infty$. Note that $\tilde{\e}^{(0)}$ has the same law as $\tilde{\e}$, and Definition \ref{conditioned measure} gives $Z_0=\frac{1}{144}$.

\begin{proposition}[Approximation in total variation distance]\label{dtv tronque}
    We have 
    \[d_{TV}\left(Law(\tilde{\e}),Law(\tilde{\e}^{(\varepsilon)})\right)\xrightarrow[\varepsilon\rightarrow0]{}0.\]
\end{proposition}
\begin{proof}
    For every $\varepsilon>0$, given $\e$, we define
    \[\phi(\varepsilon)=\1_{\{\sup \e>\varepsilon\}}\int_{T_{\varepsilon}}^{L_{\varepsilon}}\int_{T_{\varepsilon}}^{L_{\varepsilon}}\frac{dsdt}{(\e_s+1)^3(\e_s+1)^3}\quad\text{ and }\quad\psi(\varepsilon)=\1_{\{\sup \e>\varepsilon\}}\exp\left(-6\int_{T_{\varepsilon}}^{L_{\varepsilon}}\frac{du}{(\e_u+1)^2}\right).\]
    Since both laws are absolutely continuous with respect to Itô's excursion measure, we need to show that
    \begin{equation}\label{convergence absolue}
        \n\left(\left|Z_0^{-1}\phi(0)\psi(0)-Z_\varepsilon^{-1}\phi(\varepsilon)\psi(\varepsilon)\right|\right)\xrightarrow[n\rightarrow\infty]{}0.
    \end{equation}
    First, let us show that $Z_\varepsilon$ tends to $Z_0=\frac{1}{144}$. Note that we have the deterministic bound 
    \[\phi(\varepsilon)\psi(\varepsilon)\leq\phi(0),\]
    which is integrable with respect to $\n$. Indeed, using Bismut's decomposition of the Brownian excursion, we have 
    \begin{align*}
      \n( \phi(0))= \n\left(\int_0^\sigma\int_0^\sigma\frac{dsdt}{(\e_s+1)^3(\e_t+1)^3}\right)&=\int_{\R_+^3}\frac{dxdydz}{(x+y+1)^3(x+z+1)^3}\\
        &=\int_{\R_+}\frac{1}{4}\frac{dx}{(x+1)^4}\\
        &=\frac{1}{12}.
    \end{align*}
    Therefore, using the dominated convergence dominated and the fact that the equation \eqref{prob measure} defines a probability measure, we obtain 
    \begin{equation}\label{convergence partition}
        Z_\varepsilon=\n\left(\1_{\{\sup \e>\varepsilon\}}\left(\int_{T_{\varepsilon}}^{L_{\varepsilon}}\int_{T_{\varepsilon}}^{L_{\varepsilon}}\frac{dsdt}{(\e_s+1)^3(\e_t+1)^3}\right)\exp\left(-6\int_{T_{\varepsilon}}^{L_{\varepsilon}}\frac{du}{(\e_u+1)^2}\right)\right)\xrightarrow[\varepsilon\rightarrow0]{}144.
    \end{equation}    
   Let us prove \eqref{convergence absolue}. First, we have 
   \begin{equation}
       \left|Z_0^{-1}\phi(0)\psi(0)-Z_\varepsilon^{-1}\phi(\varepsilon)\psi(\varepsilon)\right|\leq \left(Z_0^{-1}+Z_\varepsilon^{-1}\right)\phi(0).
   \end{equation}
    Using \eqref{convergence partition}, we see that $\left(Z_0^{-1}+Z_\varepsilon^{-1}\right)$ is uniformly bounded, so that the term on the right-hand side is integrable with respect to $\n$. Therefore, we can use the dominated convergence theorem, which gives \eqref{convergence absolue}.
\end{proof}
The following lemma is probably known, and is similar to William's decomposition of the Bessel process of dimension 3 \cite[Theorem 3.11]{revuzyor}, but we could not find a precise reference.
\begin{lemme}[Decomposition of a Brownian excursion]\label{decompo trois bouts}
    Fix $a>0$. Then, under the probability measure $\n(\cdot\,|\,\sup\e>a)$, the excursion $\e$ can be decomposed in three independent pieces as follows : 
    \begin{itemize}[label=\textbullet]
        \item The path between $0$ and $T_a$ is a Bessel process $R^{(1)}$ of dimension $3$, starting at $0$ and stopped when it reaches $a$. 
        \item The path between $T_a$ and $L_a$ is a Bessel process $R^{(2)}$ of dimension $3$, starting at $a$ and stopped when it hits $a$ for the last time.
        \item The path between $
        L_a$ and $\sigma$ is the time reversal of a Bessel process $R^{(3)}$ of dimension $3$, starting at $0$ and stopped when it reaches $a$.
    \end{itemize}
\end{lemme}
\begin{proof}
    Fix $a>0$, and three non-negative functions $F,G,H$. First, by the Markov property of Ito's excursion measure, we have 
    \begin{align*}
        &\n\left(F((\e(s))_{0\leq s\leq T_a})G((\e(s+T_a))_{0\leq s\leq L_a-T_a})H((\e(s+L_a))_{0\leq s\leq\sigma-L_a})\,\big|\,\sup \e > a\right)\\
        &=\n\left(F((\e(s))_{0\leq s\leq T_a})\,\big|\,\sup \e > a\right)\E_{a}\left[G((B(s))_{0\leq s\leq L_a})H((B(s+L_a))_{0\leq s\leq T_0-L_a})\right],
    \end{align*}
    where under $\P_a$, $B$ is a Brownian motion starting at $a$ and stopped when it reaches $0$ for the first time. Then, using a particular case of William's time reversal theorem, we know that we can replace $B$ under $\P_a$ by the time-reversal of a Bessel process $R$ of dimension $3$, starting at $0$ and stopped when it reaches $a$ for the last time (which we denote by $L^{(3)}_a$). Observe that after these transformations $L_a$ corresponds to the first hitting time of $a$ by $R$, denoted by $T^{(3)}_a$. This gives 
    \begin{align*}
        &\n\left(F((\e(s))_{0\leq s\leq T_a})G((\e(s+T_a))_{0\leq s\leq L_a-T_a})H((\e(s+L_a))_{0\leq s\leq\sigma-L_a})\,\big|\,\sup \e > a\right)\\
        &=\n\left(F((\e(s))_{0\leq s\leq T_a})\,\big|\,\sup \e > a\right)\E_a\left[G((R(L_a^{(3)}-\cdot))_{0\leq s\leq T^{(3)}_a-T^{(3)}_a})H((R(T^{(3)}_a-\cdot))_{0\leq s\leq T^{(3)}_a})\right],\\
        &=\n\left(F((\e(s))_{0\leq s\leq T_a})\,\big|\,\sup \e > a\right)\E_a\left[G((R(L_a-\cdot))_{0\leq s\leq L_a-L_a})H((R(L_a-\cdot))_{0\leq s\leq L_a})\right]
    \end{align*}
    Finally, using the strong Markov property of $R$ at $T^{(3)}_a$, we obtain 
    \begin{align*}
       &\n\left(F((\e(s))_{0\leq s\leq T_a})G((\e(s+T_a))_{0\leq s\leq L_a-T_a})H((\e(s+L_a))_{0\leq s\leq\sigma-L_a})\,\big|\,\sup \e > a\right)\\
        &=\n\left(F((\e(s))_{0\leq s\leq T_a})\,\big|\,\sup \e > a\right)\E_{a}\left[G((R(L^{(3)}_a-\cdot))_{0\leq s\leq L^{(3)}_a})\right]\E_0\left[H((R(T^{(3)}_a-\cdot))_{0\leq s\leq T^{(3)}_a})\right].
    \end{align*}
    The result follows from the time-reversal invariance of Ito's excursion measure.
\end{proof}
This decomposition allows us to compare the local behavior of $\tilde{\e}$ near its endpoints to that of two independent Bessel processes. To simplify notations, we set 
\[\overline{T}_\varepsilon=\sigma-L_\varepsilon.\]
\begin{proposition}\label{couplage cycle}
    Let $(R,\overline{R})$ be two independent Bessel processes of dimension $3$ starting from $0$. Then, for every $a>0$, the total variation distance between 
    \[\left((R_t)_{0\leq t\leq T_a^{(3)}},(\overline{R}_t)_{0\leq t\leq \overline{T}_a^{(3)}})\right)\quad\text{ and }\quad\left((\varepsilon^{-1}\Tilde{\e}_{\varepsilon^2 t})_{0\leq t\leq \varepsilon^{-2}T_{\varepsilon a}},(\varepsilon^{-1}\Tilde{\e}_{\sigma-\varepsilon^2 t})_{0\leq t\leq \varepsilon^{-2}\overline{T}_{\varepsilon a}}\right)\]
    goes to $0$ as $\varepsilon$ goes to $0$.
\end{proposition}
\begin{proof}
   We give the proof for $a=1$. Fix $\varepsilon>0$.
By Proposition \ref{dtv tronque}, it is enough to prove the result with $\tilde{\e}^{({\varepsilon})}$ instead of $\tilde{\e}$. 
   Observe that the weight that appears in \eqref{excursion tronquee} is a function of $(\e^{(\varepsilon)}_s)_{T_{\varepsilon}\leq s\leq L_{\varepsilon}}$. However, by Proposition \ref{decompo trois bouts}, the law of $\left((\Tilde{\e}^{({\varepsilon})}_t)_{0\leq t\leq T_{\varepsilon}},(\Tilde{\e}^{({\varepsilon})}_{\sigma-t})_{0\leq t\leq \overline{T}_{\varepsilon})}\right)$ is exactly the law of 
$\left((R_t)_{0\leq t\leq T_{\varepsilon}^{(3)}},(\overline{R}_t)_{0\leq t\leq \overline{T}_{\varepsilon}^{(3)}})\right)$. The result follows from the scaling properties of Bessel processes. 
\end{proof}
Before stating our next result, we introduce some notations. In order to have more symmetry with respect to the element $0$, we will replace the excursion $\tilde{\e}$ on $[0,\sigma]$ by a function $\tilde{\e}'$, defined by
\[\tilde{\e}_t=\left\{
    \begin{array}{ll}
   \tilde{\e}_t\quad&\text{ if $t\in[0,\sigma/2]$} \\
     \tilde{\e}_{\sigma/2-t}\quad&\text{ if $t\in[-\sigma/2,0]$.}
    \end{array}\right. \]
Similarly, we replace the point measures $\cM_{-1}$ and $\overline{\cM}_{-1}$ of Definition \ref{conditioned measure} by two point measures $\cM_{-1}$ and $\overline{\cM}_{-1}$ on $[-\sigma/2,\sigma/2]\times\mathfrak{S}$ defined as follows. For $t\in[0,\sigma]$, set $f(t)=t$ if $t\in[0,\sigma/2]$, and $f(t)=t-\sigma$ otherwise. Then, if $I,J$ are sets indexing the atoms of $\cM_{-1}$ and $\overline{\cM}_{-1}$, set
\[\cM_{-1}'=\sum_{i\in I}\delta_{f(t_i),\omega_i},\quad\overline{\cM}_{-1}'=\sum_{j\in J}\delta_{f(t_j),\omega_j}.\]
Note that with these changes, we have 
\[\overline{T}_\varepsilon=-\inf\{t\geq0:\tilde{\e}_{-t}'=\varepsilon\}.\]
It is straightforward to check that these modifications do not change the law of the associated surface. Consequently, we will still denote by $\mathfrak{U}$ the random unicycle obtained from these processes, and keep the same notations for the processes.
Then, for every point measure $\cM$ on $\R\times\mathfrak{S}$ with 
\[\cM=\sum_{i\in I}\delta_{t_i,\omega_i},\] and for every $\varepsilon>0$, we define 
\[\Omega_\varepsilon(\cM)=\sum_{i\in I}\delta_{\varepsilon^{-2}t_i,\Theta_{\varepsilon^{-2}}(\omega_i)},\] where $\Theta$ is the scaling operator introduced in Section \ref{snake}. Finally, for every $a>0$, define 
\[T_a^{(\infty)}=\inf\{t\geq0:X_t=a\}\quad\text{and}\quad \overline{T}_a^{(\infty)}=-\inf\{t\geq0:X_{-t}=a\},\]
where $X$ is the process introduced at the beginning of Section \ref{section construction bigeo}.
\begin{proposition}\label{couplage mesure}
    For every $\delta>0$ and $a>0$, there exists $\varepsilon_0>0$ such that for every $0<\varepsilon<\varepsilon_0$, one can couple the point measures $(\cM_\infty,\overline{\cM}_\infty)|_{[-\overline{T}_a^{(\infty)},T_a^{(\infty)}]}$ and $(\cM_{-1},\overline{\cM}_{-1})|_{[-\overline{T}_{a\varepsilon},T_{a\varepsilon}]}$ in such a way that the equality
    \[\left(\cM_\infty|_{[-\overline{T}_a^{(\infty)},T_a^{(\infty)}]},\overline{\cM}_\infty|_{[-\overline{T}_a^{(\infty)},T_a^{(\infty)}]}\right)=\left(\Omega_\varepsilon(\cM_{-1}|_{[-\overline{T}_{a\varepsilon},T_{a\varepsilon}]}),\Omega_\varepsilon(\overline{\cM}_{-1}|_{[-\overline{T}_{a\varepsilon},T_{a\varepsilon}]})\right)\] holds with probability at least $1-\delta$. 
\end{proposition}
\begin{proof}
   Fix $\delta>0$. First, by Proposition \ref{couplage cycle}, we can couple the trajectories $\left((X_t)_{t\geq0},(X_{-t})_{t\geq0})\right)$ and $\left((\Tilde{\e}_t)_{0\leq t\leq \sigma},(\Tilde{\e}_{\sigma-t})_{0\leq t\leq \sigma}\right)$ so that the equality 
   \begin{equation}\label{egalite processes}
       \left((X_t)_{0\leq t\leq T_a^{(\infty)}},(X_{-t})_{0\leq t\leq \overline{T}_a^{(\infty)}})\right)=\left((\varepsilon^{-1}\Tilde{\e}_{\varepsilon^2 t})_{0\leq t\leq \varepsilon^{-2} T_{\varepsilon a}},(\varepsilon^{-1}\Tilde{\e}_{\sigma-\varepsilon^2 t})_{0\leq t\leq \varepsilon^{-2}\overline{T}_{\varepsilon a}}\right)
   \end{equation}
    holds with probability at least $1-\delta/2$. We denote this event by $E_\varepsilon$, and from now on, we work with this coupling of the trajectories. Observe that we can decompose our point measures in the following way.
   First, by standard properties of Poisson point measures, we have 
   \begin{equation*}       \cM_\infty=\cM_\infty^{>-1/\varepsilon}+\cM_\infty^{\leq-1/\varepsilon}\quad\text{ and }\quad\overline{\cM}_\infty=\overline{\cM}_\infty^{>-1/\varepsilon}+\overline{\cM}_\infty^{\leq-1/\varepsilon},
   \end{equation*}
   where, given $X$, $\cM_\infty^{>-1/\varepsilon}$ and $\overline{\cM}_\infty^{>-1/\varepsilon}$ are Poisson point measures with intensity 
   \[\1_{\{W_*>-1/\varepsilon\}}dt\N_{X_t}(dW),\]
   and where
   $\cM_\infty^{\leq-1/\varepsilon}$ and $\overline{\cM}_\infty^{\leq-1/\varepsilon}$ are Poisson point measures with intensity 
   \[\1_{\{W_*\leq-1/\varepsilon\}}dt\N_{X_t}(dW),\]
   all these measures being independent. On the other hand, as described in Definition \ref{conditioned measure}, we know that we have
  \[\cM_{-1}=\cM_{>-1}+\delta_{t_1,\omega^{(1)}}\quad\text{and}\quad\overline{\cM}_{-1}=\overline{\cM}_{>-1}+\delta_{\overline{t}_1,\overline{\omega}^{(1)}}.\]
   Using the scaling property of the measure $\N_0$, conditionally on \eqref{egalite processes}, it is easy to see that the measures $\left(\cM_\infty^{>-1/\varepsilon},\overline{\cM}^{>-1/\varepsilon}_\infty\right)|_{[-\overline{T}_{a},T_{a}]}$ and $\left(\Omega_\varepsilon(\cM_{>-1}|_{[-\overline{T}_{a\varepsilon},T_{a\varepsilon}]}),\Omega_\varepsilon(\overline{\cM}_{>-1}|_{[-\overline{T}_{a\varepsilon},T_{a\varepsilon}]})\right)$ have the same law. 
   
   Let $H_\varepsilon$ be the event that the point measures $\left(\cM_\infty^{\leq-1/\varepsilon},\overline{\cM}_\infty^{\leq-1/\varepsilon}\right)|_{[-\overline{T}_a^{(\infty)},T_a^{(\infty)}]}$ are empty and $t_1,\overline{t}_1\notin[-\overline{T}_{a\varepsilon},T_{a\varepsilon}]$. Then, conditionally on $H_\varepsilon$, we can couple the different random point measures in such a way that we have the equality  $(\cM_\infty,\overline{\cM}_\infty)|_{[-\overline{T}_{a},T_{a}]}=\left(\Omega_\varepsilon(\cM_{-1}),\Omega_\varepsilon(\overline{\cM}_{-1})\right)|_{[-\overline{T}_{a\varepsilon},T_{a\varepsilon}]}.$  For every $-\sigma/2\leq a\leq b\leq \sigma/2$, we define $\mathfrak{U}^{[a,b]}$ (or $\mathfrak{U}^{(a,b)}$) as the connected subset of $\mathfrak{U}$ composed of
\begin{equation}\label{subset}
    [a,b]\sqcup\left(\bigsqcup_{\substack{i\in I\cup J\\a\leq t_i\leq b}}\cT_{(\omega_i)}\right)\quad \text{or}\quad (a,b)\sqcup\left(\bigsqcup_{\substack{i\in I\cup J\\a\leq t_i\leq b}}\cT_{(\omega_i)}\right),
\end{equation}
   and we define $\mathfrak{U}_\infty^{[a,b]}$ and $\mathfrak{U}_\infty^{(a,b)}$ similarly.
   Note that we have 
   \[H_\varepsilon=E_\varepsilon\cap\left\{t_1,\overline{t}_1\notin[-\overline{T}_{a\varepsilon},T_{a\varepsilon}]\right\}\cap\left\{\inf_{u\in \mathfrak{U}^{[-\overline{T}_{a},T_{a}]}_\infty}\ell_u>-1/\varepsilon\right\}.\]
   Then, we have
   \[\P\left(\inf_{u\in \mathfrak{U}^{[-\overline{T}_{a},T_{a}]}_\infty}\ell_u>-1/\varepsilon\right)=\E\left[\exp\left(-2\int_{-\overline{T}_{a}}^{T_{a}}\N_{X_t}(W_*<-1/\varepsilon)dt\right)\right]=\E\left[\exp\left(-3\int_{-\overline{T}_{a}}^{T_{a}}\frac{1}{(X_t+1/\varepsilon)^2}dt\right)\right].\]
   Moreover, we also have 
   \[\P\left(t_1,\overline{t}_1\notin[-\overline{T}_{a\varepsilon},T_{a\varepsilon}]\right)=\frac{1}{8}\n\left(\left(\int_{T_{\varepsilon}}^{L_{\varepsilon}}\int_{T_{\varepsilon}}^{L_{\varepsilon}}\frac{dsdt}{(\e_s+1)^3(\e_t+1)^3}\right)\exp\left(-6\int_0^\sigma\frac{du}{(\e_u+1)^2}\right)\right).\]
   By monotone convergence, it is clear that these two probabilities converge to $1$ as $n$ goes to $\infty$. Therefore, by Proposition \ref{couplage cycle}, 
   \[\P(H_\varepsilon)\xrightarrow[\varepsilon\rightarrow 0]{}1,\] which concludes the proof. 
\end{proof} 

In particular, Proposition \ref{couplage cycle} and \ref{couplage mesure} imply that for every $\delta>0$, there exists $\varepsilon_0>0$ such that for every $0<\varepsilon<\varepsilon_0$, we can couple $\mathfrak{U}$ and $\mathfrak{U}_\infty$ in such a way that the two equalities 
\begin{equation*}
       \left((X_t)_{0\leq t\leq T_a^{(\infty)}},(X_{-t})_{0\leq t\leq \overline{T}_a^{(\infty)}})\right)=\left((\varepsilon^{-1}\Tilde{\e}_{\varepsilon^2 t})_{0\leq t\leq \varepsilon^{-2} T_{\varepsilon a}},\varepsilon^{-1}\Tilde{\e}_{\sigma-\varepsilon^2 t})_{0\leq t\leq (\varepsilon^{-2}\overline{T}_{\varepsilon a}}\right)
   \end{equation*}
   and
    \[\left(\cM_\infty|_{[-\overline{T}_a^{(\infty)},T_a^{(\infty)}]},\overline{\cM}_\infty|_{[-\overline{T}_a^{(\infty)},T_a^{(\infty)}]}\right)=\left(\Omega_\varepsilon(\cM_1|_{[-\overline{T}_{a\varepsilon},T_{a\varepsilon}]}),\Omega_\varepsilon(\overline{\cM}_1|_{[-\overline{T}_{a\varepsilon},T_{a\varepsilon}]})\right)\] hold simultaneously with probability at least $1-\delta$. To lighten notations, we will denote this event by
\[\left\{\mathfrak{U}^{[-\overline{T}_{a\varepsilon},T_{a\varepsilon}]}=\varepsilon\cdot\mathfrak{U}_\infty^{[-\overline{T}_{a},T_{a}]}\right\}.\]

\subsection{Localization lemma and convergence}

In this section, we prove that with high probability, the projection of the set $\mathfrak{U}^{[-\overline{T}_\varepsilon,T_\varepsilon]}$ (resp. $\mathfrak{U}_\infty^{[-\overline{T}_\varepsilon^{(\infty)},T_\varepsilon^{(\infty)}]}$) contains a small ball around $\Gamma(1)$ (resp. $\Gamma_\infty(0)$).  This will enable us to transfer results on $\mathfrak{U}$ and $\mathfrak{U}_\infty$ to results on the associated surfaces. We mention that this kind of result already appeared several times in the study of Brownian surfaces (see \cite[Lemma 5]{brownianplane} or \cite[Proposition 4.1]{Bigeodesicbrownianplane}).

\begin{proposition}\label{localisation}
   For every $\eta>0$, there exists $\delta>0$ such that for every $\varepsilon>0$ small enough, 
    \begin{align}
        \U^{(1,1)}\Big(B_{\delta\varepsilon}(\cS,\Gamma(1))&\subset p_{\cS}(\mathfrak{U}^{[-\overline{T}_\varepsilon,T_\varepsilon]})\Big)>1-\eta\\        \P\Big(B_{\delta\varepsilon}u(\cS_\infty,\Gamma_\infty(0))&\subset p_{\cS_\infty}(\mathfrak{U}_\infty^{[-\overline{T}_\varepsilon^{(\infty)},T_\varepsilon^{(\infty)}]})\Big)>1-\eta.
    \end{align}    
\end{proposition}
We will only prove the statement about $\cS_\infty$. The proof of the other inequality is similar, and can be done by adapting the proof of \cite[Proposition 4.1]{Bigeodesicbrownianplane}. Note that, using the scaling properties of $\mathfrak{U}_\infty$ and $\cS_\infty$, we just need to prove the result for $\varepsilon=1$. The proof will be based on the following lemma.

\begin{lemme}\label{a l'infini}
   Almost surely, there exists some (random) $\delta>0$ and $\varepsilon>0$ such that for every $u\notin \mathfrak{U}_\infty^{(-\delta,\delta)}$, we have 
   \begin{equation*}
       D_\infty\left(p_{\cS_\infty}(u),0\right)>\varepsilon.
   \end{equation*}
\end{lemme}
\begin{proof}
The idea is to show that almost surely, there exists a region such that every point that is far away of $0$ in $\mathfrak{U}_\infty$ must cross this region in order to reach $\Gamma_\infty(0)$ in $\cS_\infty$.

    Let $\ell_*=\inf_{u\in \mathfrak{U}_\infty^{[-\overline{T}_1^{(\infty)},T_1^{(\infty)}]}}\ell_u$ and $L=\sup\{t\geq0\,:\,X_t=1\}$. We introduce the sets 
    \[B=\{i\in I,\,(\omega_{i})_*<\ell_*-1\text{ and }t_{i}>L\},\quad\overline{B}=\{j\in J,\,(\omega_j)_*<\ell_*-1\text{ and }t_j>L\}.\] 
    Note that these sets are never empty. Indeed, using properties of Poisson point measures and Bessel processes, we have 
    \begin{align*}
        \P(B=\emptyset)&=\E\left[\exp\left(-2\int_L^\infty\N_{X_t}(W_*<\ell_*-1)dt\right)\right]\\
        &=\E\left[\exp\left(-3\int_L^\infty\frac{dt}{(X_t+1-\ell_*)^2}\right)\right]\\
        &=0.
    \end{align*}
 Finally, let $i_*$ (resp. $j_*$) be the unique element of $B$ (resp. $\overline{B}$), such that $t_{i_*}=\inf_{i\in B}t_i$ (resp. $t_{j_*}=\inf_{j\in \overline{B}}t_j$.

Then, fix $u\in\mathfrak{U}_\infty^{(L,\infty)}$, and suppose that $D_\infty(p_{\cS_\infty}(u),0)<1/2$. First, note that the bound \eqref{bound} implies that $u\notin [L,\infty)$. Therefore, without loss of generality, we can suppose that $u\in \cT_i$ for some $i\in I$ with $t_i>L$. By assumption, there exists a sequence $u_0,...,u_p\in\mathfrak{U}_\infty$ such that $u_0=u,\,u_p=0$ and
\begin{equation}\label{assumption}    
\sum_{i=1}^p D_\infty^\circ(u_{i-1},u_i)\leq 3/4.
\end{equation}
Let $\tau=\sup\{i\geq0,\,u_i\in\mathfrak{U}_\infty^{(t_{i_*},\infty)}\}$ and $\eta=\inf\{i>\tau,\,u_i\notin\mathfrak{U}_\infty^{(t_{i_*},\infty)}\}$. As previously, the bound \eqref{bound} implies that for every $0\leq i\leq\tau$, $u_i\notin [L,\infty)$, which means that $u_\tau\in\cT_i$ for some $i\in I$ with $t_i>t_{i_*}$. Observe that the assumption \eqref{assumption} implies that $\ell_{u_i}>\ell_*-1/2$ for every $0\leq i\leq p$, and that 
\[D_\infty^\circ(u_\tau,u_\eta)\geq\ell_{u_\tau}+\ell_{u_\eta}-2(\ell_*-1)\geq1.\]
Therefore, we know that $\eta\neq\tau+1$, and that for every $\tau<i<\eta$, $u_i\in\cT_{i_*}$. Then, let $\overline{u_0}$ (resp. $\overline{u}_1$) be the last (resp. the first) element of $\cT_{i_*}$ with label $\ell_{\overline{u}_\tau}=\inf_{v\in[u_{\tau+1},u_\tau]}\ell_v$ and $\ell_{\overline{u}_\eta}=\inf_{v\in[u_{\eta},u_{\eta-1}]}\ell_v$. These elements can be defined properly using the exploration process $\cE$, but we omit the details. However, these elements always exist. Indeed, the bound \eqref{bound} imply that 
\[D^\circ_\infty(u_\tau,u_{\tau+1})\leq3/4\quad\Longrightarrow\quad\inf_{v\in[u_{\tau+1},u_\tau]}\ell_v\geq-\frac{7}{8}\geq\ell_*-1.\]
Then, it is easy to check that 
\[D^\circ_\infty(u_\tau,u_{\tau+1})=D^\circ_\infty(u_\tau,\overline{u}_\tau)+D^\circ_\infty(\overline{u}_\tau,u_{\tau+1}),\] and that a similar identity holds if we replace $\tau$ by $\eta$. Therefore, we can add $\overline{u}_\tau$ and $\overline{u}_\eta$ to our sequence, since it does not change \eqref{assumption}. 

Now, we can consider the \textbf{Brownian slice} associated to the tree $\cT_{i_*}$, denoted by $\left(\tilde{\cS},\tilde{D}\right)$, which is a random metric space with two geodesic boundaries. Brownian slices have been introduced in \cite{uniqueness}, and have been used several times to prove results about a Brownian sphere (see \cite{Browniandisk,Geodesicstars,Bigeodesicbrownianplane,Isoperimetric}). We refer to these references for the definition of Brownian slices. 

Let $\gamma,\overline{\gamma}$ be the geodesic boundaries of $\tilde{\cS}$. These paths have length $\chi=\ell_{t_{i_*}}-(\omega_{i_*})_*>2$, and for every $s\in[0,\chi],$ $\gamma(s)$ (resp $\overline{\gamma}(s)$) corresponds to to the projection of the last (resp. the first) element of $\cT_{i_*}$ with label $\ell_{t_{i_*}}-s$. Then, by \cite[Theorem 11]{Browniandisk} or \cite[Corollary 4.9]{Bigeodesicbrownianplane}, we have 
\begin{align*}
    \sum_{i=1}^pD_\infty^\circ(u_{i-1},u_i)&\geq  D_\infty^\circ(\overline{u}_\tau,u_{\tau+1})+\sum_{i=\tau+1}^{\eta-2}D_\infty^\circ(u_i,u_{i-1})+D_\infty^\circ(u_{\eta-1},\overline{u}_\eta)\\
    &\geq \tilde{D}\left(\gamma(\ell_{\overline{u}_\tau}+1),(\overline{\gamma}(\ell_{\overline{u}_\eta}+1\right)\\
    &\geq\inf_{s,t\in[\frac{1}{2},\frac{15}{8}]}\tilde{D}\left(\gamma(s),\overline{\gamma}(t)\right)
\end{align*}
where the last line follows from the inequality $-\frac{7}{8}\leq\ell_{\overline{u}_\tau}\leq\frac{1}{2}$. Finally, since $\gamma$ and $\overline{\gamma}$ do not intersect each other except at their endpoints, a compactness argument gives
\[\inf_{s,t\in[\frac{1}{2},\frac{15}{8}]}\tilde{D}\left(\gamma(s),\overline{\gamma}(t)\right)>0\quad a.s.\]
The key observation is that this lower bound holds for every sequence satisfying \eqref{assumption}. Consequently, for every $u\in\mathfrak{U}_\infty^{(T,\infty)}$, we have
\[D_\infty(p_{\cS_\infty}(u),0)\geq\frac{1}{2}\wedge\inf_{s,t\in[\frac{1}{2},\frac{15}{8}]}\tilde{D}\left(\gamma(s),\overline{\gamma}(t)\right)>0.\]   By symmetry, this concludes the proof. 
\end{proof}

\begin{proof}[Proof of Proposition \ref{localisation}]
    We argue by contradiction. Suppose that, with a positive probability, $p_{\cS_\infty}\left(\mathfrak{U}_\infty^{[-\overline{T}_1^{(\infty)},T_1^{(\infty)}]}\right)$ is not a neighborhood of $\Gamma_\infty(0)$. On this event, there exists a sequence $(u_n)_{n\geq1}$ such that for every $n\geq1$, $u_n\notin \mathfrak{U}_\infty^{[-\overline{T}_1^{(\infty)},T_1^{(\infty)}]}$ and $D_\infty(p_{\cS_\infty}(u_n),p_{\cS_\infty}(0))<\frac{1}{n}$. By Lemma \ref{a l'infini}, this sequence is bounded in $\mathfrak{U}_\infty$. Therefore, up to extracting a subsequence, we can suppose that this sequence converges toward an element $u_\infty\notin \mathfrak{U}_\infty^{[-\overline{T}_1^{(\infty)},T_1^{(\infty)}]}$. Moreover, it must satisfy $D_\infty(p_{\cS_\infty}(u_\infty),p_{\cS_\infty}(0))=0$. By Proposition \ref{localisation}, this means that $D^\circ_\infty(p_{\cS_\infty}(u_\infty),p_{\cS_\infty}(0))=0$, which is not possible. Therefore, almost surely, $p_{\cS_\infty}\left(\mathfrak{U}_\infty^{[-\overline{T}_1^{(\infty)},T_1^{(\infty)}]}\right)$ is a neighborhood of $\Gamma_\infty(0)$. This means that 
    \[\P\left(\text{ there exists $\delta>0$ such that }B_\delta(\cS_\infty,\Gamma_\infty(0))\subset p_{\cS_\infty}\left(\mathfrak{U}_\infty^{[-\overline{T}_1^{(\infty)},T_1^{(\infty)}]}\right)\right)=1.\] 
    Hence, for every $\varepsilon>0$, there exists $\delta>0$ such that \[\P\left(B_\delta(\cS_\infty,\Gamma_\infty(0))\subset p_{\cS_\infty}\left(\mathfrak{U}_\infty^{[-\overline{T}_1^{(\infty)},T_1^{(\infty)}]}\right)\right)>1-\varepsilon.\] 
    A scaling argument concludes the proof. 
\end{proof}

For every $\varepsilon,\delta>0$, we let $\cF_{\varepsilon,\delta}$ be the intersection of the following events:
\begin{itemize}[label=\textbullet]
    \item $\left\{\mathfrak{U}^{[-\overline{T}_{2\varepsilon},T_{2\varepsilon}]}=\varepsilon\cdot\mathfrak{U}_\infty^{[-\overline{T}_2^{(\infty)},T_2^{(\infty)}]}\right\}$
    \item $\left\{B_{\delta\varepsilon}(\cS,\Gamma(1))\subset p_{\cS}(\mathfrak{U}^{[-\overline{T}_\varepsilon,T_\varepsilon]})\right\}\cap\left\{B_{2\delta\varepsilon}(\cS,\Gamma(1))\subset p_{\cS}(\mathfrak{U}^{[-\overline{T}_{2\varepsilon},T_{2\varepsilon}]})\right\}$,
    \item $\left\{B_{\delta}(\cS_\infty,\Gamma(0))\subset p_{\cS_\infty}(\mathfrak{U}_\infty^{[-\overline{T}_1^{(\infty)},T_1^{(\infty)}]})\right\}\cap\left\{B_{2\delta}(\cS_\infty,\Gamma(0))\subset p_{\cS_\infty}(\mathfrak{U}_\infty^{[-\overline{T}_2^{(\infty)},T_2^{(\infty)}]})\right\}$
    \item $\{\inf_{u\in\mathfrak{U}}\ell_u<\inf_{u\in \mathfrak{U}^{[-\overline{T}_\varepsilon,T_\varepsilon]}}\ell_u\}$.
\end{itemize}
Note that by Propositions \ref{couplage cycle}, \ref{couplage mesure}, \ref{localisation}, and by monotone convergence, for every $\eta>0$, we can choose $\delta>0$ such that for every $\varepsilon>0$ small enough,
\[\P(\cF_{\varepsilon,\delta})>1-\eta.\]
\begin{proposition}\label{correspondance distance}
    Suppose that $\cF_{\varepsilon,\delta}$ holds. Then, there exists a bijection $\Phi:\mathfrak{U}^{[-\overline{T}_\varepsilon,T_\varepsilon]}\rightarrow \mathfrak{U}_\infty^{[-\overline{T}_1^{(\infty)},T_1^{(\infty)}]}$ such that, for every $u,v\in \mathfrak{U}^{[-\overline{T}_\varepsilon,T_\varepsilon]}$, we have
    \[D(p_{\cS}(u),p_{\cS}(v))=\varepsilon D_\infty(p_{\cS\infty}(\Phi(u),p_{\cS\infty}\Phi(v)).\]
\end{proposition}
\begin{proof}
   Since we work on the event $\cF_{\varepsilon,\delta}$, we can associate to every element $u\in \mathfrak{U}^{[-\overline{T}_\varepsilon,T_\varepsilon]}$ the corresponding $\Phi(u)\in \mathfrak{U}_\infty^{[-\overline{T}_1^{(\infty)},T_1^{(\infty)}]}$, which yields a bijection. Furthermore, it is easy to see that $\Phi$ preserves intervals. Hence, for every $u,v\in \mathfrak{U}^{[-\overline{T}_\varepsilon,T_\varepsilon]}$ such that $[u,v]\subset \mathfrak{U}^{[-\overline{T}_\varepsilon,T_\varepsilon]}$, we have 
   \begin{align}\label{egalite distance2}
       D^\circ(u,v)&=\ell_u+\ell_v-2\inf_{w\in [u,v]}\ell_w\nonumber\\
       &=\varepsilon\left(\ell_{\Phi(u)}+\ell_{\Phi(v)}-2\inf_{w\in [\Phi(u),\Phi(v)]}\ell_w\right)\nonumber\\
       &=\varepsilon D^\circ_\infty(\Phi(u),\Phi(v)).
   \end{align}
   Then, note that 
    \begin{equation*}
        D(u,v)=\inf_{u_0,...,u_m}\sum_{i=0}^{m-1}D^\circ(u_i,u_{i+1}),
    \end{equation*}
    where the infimum is taken over all $p\in\N$ and sequences $u_0,...,u_m\in \mathfrak{U}^{[-\overline{T}_{2\varepsilon},T_{2\varepsilon}]}$ with $u_0=u$ and $u_m=v$. This readily follows from the fact that we work on the event $\left\{B_{2\delta\varepsilon}(\cS,\Gamma(1))\subset p_{\cS}(\mathfrak{U}^{[-\overline{T}_{2\varepsilon},T_{2\varepsilon}]})\right\}$. Similarly, for $u',v'\in \mathfrak{U}_\infty^{[-\overline{T}_1^{(\infty)},T_1^{(\infty)}]}$, we have 
    \begin{equation*}
        D(u,v)=\inf_{u_0,...,u_m}\sum_{i=0}^{m-1}D^\circ(u_i,u_{i+1}),
    \end{equation*}
    where the infimum is taken over all $p\in\N$ and sequences $u_0,...,u_m\in \mathfrak{U}_\infty^{[-\overline{T}_2^{(\infty)},T_2^{(\infty)}]}$ with $u_0=u'$ and $u_m=v'$. The result follows from these two formulas, together with \eqref{egalite distance2}.
\end{proof}
    
    We can finally prove the main result of this section.
    
    \begin{proof}[Proof of Theorem \ref{equivalence definition}]
        It is known \cite[Theorem 5.11]{Bigeodesicbrownianplane} that the bigeodesic Brownian plane is the local limit of the Brownian sphere under $\N_0(\cdot\,|\,W_*=-2)$ around $\Gamma(1)$. Therefore, by Theorem \ref{Representation libre} and uniqueness of the limit, we need to show that $$\left(\cS_\infty,D_\infty,p_{\cS_\infty}(0)\right)$$ is the local limit of the Brownian sphere (under $\U^{(1,1)}$) around $p_{\cS}(0)$.

        To do so, let us prove that for every $\eta>0$ and $\delta>0$, there exists $\varepsilon_0>0$ such that for every $0<\varepsilon<\varepsilon_0$, we can couple $\left(\cS,\cS_\infty\right)$ in such a way that is the equality
        \[B_\delta(\varepsilon^{-1}\cdot\cS,p_{\cS}(0))=B_\delta(\cS_\infty,\overline{\rho}_\infty)\]
        holds with probability at least $1-\eta$.

        We work on the event $\cF_{\varepsilon,\delta}$. Under this event, by Proposition \ref{localisation}, every $x\in B_\delta(\varepsilon^{-1}\cdot\cS,p_{\cS}(0))$ is of the form $p_{\cS}(u)$ with $u\in A_{\varepsilon}$. Then, set 
        \[\cI(p_{\cS}(u))=p_{\cS_\infty}(\Phi(u)).\]
        By Proposition \ref{correspondance distance}, this formula does not depend on the choice of $u$. Moreover, it is an isometry between $B_\delta(\varepsilon^{-1}\cdot\cS,p_{\cS}(0))$ and $B_\delta(\cS_\infty,\overline{\rho}_\infty)$, and it is easy to see that it satisfies $\mathcal{I}_*(\varepsilon^{-4}\mathrm{Vol})|_{B_\delta(\varepsilon^{-1}\cdot\cS,\Gamma_1)}=\overline{\mathrm{Vol}}|_{B_\delta(\cS_\infty,\overline{\rho}_\infty)}$. Since $\cI(\Gamma_1)=\overline{\rho}_\infty$, this concludes the proof. 
    \end{proof}

\subsection{Consequences}
    
Let us show how this new construction allows us to obtain a new invariance result for the bigeodesic Brownian plane. Recall that $\gamma_\infty$ is the only infinite bigeodesic in $\cS_\infty$. Let $\mathcal{HS}_\infty$ be the random metric space obtained when we replace the two independent Bessel processes in Section \ref{section construction bigeo} by a single one, and when we restrict $\cM_\infty$ and $\overline{\cM}_\infty$ to $[0,\infty)$. We omit the details, but this random space appears when we cut $\cS_\infty$ along $\gamma_\infty$ (see \cite[Section 6.3]{Bigeodesicbrownianplane}, where this space is called $\overline{\mathcal{HBP}}$). Conversely, consider two independent copies $\mathcal{HS}_\infty^{(1)},\mathcal{HS}_\infty^{(2)}$ with the law described above. It is easy to check that these spaces also have a unique bi-infinite geodesic $\gamma_\infty^{(1)},\gamma_\infty^{(2)}$, so we can see them as curve-decorated metric spaces. Then, the random curve-decorated metric space obtained by gluing $\mathcal{HS}_\infty^{(1)}$ and $\mathcal{HS}_\infty^{(2)}$ alongside $\gamma_\infty^{(1)}$ and $\gamma_\infty^{(2)}$. One may view $\cS_\infty$ as a curve decorated metric space. For every $t\in\R $, we define 
\[\overline{\gamma}^{(1)}_\infty(t)=\gamma^{(1)}_\infty(-t).\]
\begin{proposition}\label{Retournement}
     The curve-decorated metric space $(\mathcal{HS}^{(1)}_\infty,\overline{\gamma}^{(1)}_\infty)$ has the same distribution as $(\mathcal{HS}^{(1)}_\infty,{\gamma}^{(1)}_\infty)$.
\end{proposition}
\begin{proof}
      Let $\left(X^{(1)},\cM^{(1)}_\infty,\overline{\cM}^{(1)}_\infty\right)$ be the triple used to construct $\mathcal{HS}_\infty^{(1)}$. Then, let  $\overline{\mathcal{HS}}^{(1)}_\infty$ be the random surface obtained from the triple $(X^{(1)},\overline{\cM}^{(1)}_\infty,\cM^{(1)}_\infty)$ (we have switched the Poisson point measures). We can naturally couple $\mathcal{HS}^{(1)}_\infty$ and $\overline{\mathcal{HS}}^{(1)}_\infty$, and it is obvious that these two spaces have the same law. However, from the definitions of these spaces, we see that  $\gamma^{(1)}_\infty$ 
has the same role in $\mathcal{HS}_\infty^{(1)}$ as $\overline{\gamma}^{(1)}_\infty$ in $\overline{\mathcal{HS}}_\infty^{(1)}$.
\end{proof}
    As a direct consequence, we can recover 
    \cite[Theorem 6.3]{Bigeodesicbrownianplane}. This result was proved using an approximation sequence of the bigeodesic Brownian plane that had a symmetry with respect to the bigeodesic. However, it was not clear at all how to prove this result directly from the previous construction of the bigeodesic Brownian plane. On the opposite, this new proof only relies on the new construction of the limiting space.
\begin{corollary}
    The curve-decorated metric space $(\cS_\infty,\overline{\gamma}_\infty)$ has the same distribution as $(\cS_\infty,\gamma_\infty)$.
\end{corollary}
\begin{proof}
  Since the space $(\cS,\gamma_\infty)$ is obtained by gluing $(\mathcal{HS}_\infty^{(1)},\gamma_\infty^{(1)})$ with $(\mathcal{HS}_\infty^{(2)},\gamma_\infty^{(2)})$, it follows that $(\cS,\overline{\gamma}_\infty)$ is obtained by gluing $(\mathcal{HS}_\infty^{(1)},\overline{\gamma}_\infty^{(1)})$ with $(\mathcal{HS}_\infty^{(2)},\overline{\gamma}_\infty^{(2)})$. However, by Proposition \ref{Retournement}, all these spaces have the same law, which gives the result. 
\end{proof}

\section{Discussions and remarks}

\subsection{About local limits}

Instead of looking at the behavior of $\mathfrak{U}^{(1,1)}$ around the point where the minimal label on the cycle is reached we could have study $\mathfrak{U}^{(1,1)}$ near a uniform point on its cycle. Note that such points are not typical in the Brownian sphere, since they have at least two disjoint outgoing geodesics (such points are called $2$-geodesic stars). As we did in Section \ref{construction}, this means replacing the Brownian excursion in the definition of $\mathfrak{U}^{(1,1)}$ by a Brownian bridge (and $\e$ by $\mathbf{b}-\inf \mathbf{b}$ in the Radon-Nikodym derivative). Then, we can repeat all the work done in Section \ref{section bigeodesique}, to find that the local limit around a uniform point in $p_\cS([0,\sigma])$ (under $\U^{(1,1)}$) is an explicit random metric space, which construction is the same as is the one presented in section \ref{section construction bigeo}, except that we need to replace the two Bessel processes by two independent Brownian motions. Furthermore, since we saw that $p_\cS([0,\sigma])$ corresponds to the boundary of the Voronoï cells of $x_*$ and $\overline{x}_*$, such a space describes the local behavior of the Brownian sphere around the boundary of Voronoï cells.

On the other hand, one could look at the behavior of the Brownian sphere around a point which is in $p_{\cS}(\mathrm{Int}(\cT))$ (under $\N_0$), namely a point which is in the projection of the skeleton of $\cT$. Using the Bismut decomposition of a Brownian excursion, from a point of the spine, we locally observe that a two-sided Brownian motion and two forests obtained from Poisson point measures. Therefore, locally, the boundary of Voronoï cells looks the same as the skeleton of the tree.

\subsection{About Voronoï cells}

Several results about Voronoï cells in the Brownian sphere are already known. For instance, it was shown in \cite{Guitter_statistics} that the perimeter of the Voronoï cells between two vertices taken uniformly at random in a large quadrangulations has a scaling limit. Moreover, in \cite{Guitter_Proof_Chapuy}, it is proved that the volume of one of the two Voronoï cells in the standard Brownian sphere obtained by choosing two points according to the volume measure is a uniform random variable in $[0,1]$. Let us mention that this statement is a particular case of Chapuy's conjecture \cite{Chapuy_Conjecture}, which states that for every $k\geq 2$, the law of the volumes of the $k$ Voronoï cells obtained by choosing $k$ points according to the volume measure of a compact Brownian surface of fixed genus is the same as the law the volumes of a uniform $k$-division of the unit interval.

As mentioned in Section \ref{section free model}, we believe that using the new construction of this paper, it might be possible to recover the results of \cite{Guitter_statistics,Guitter_Proof_Chapuy} directly in the continuum. We also mention that a different approach was initiated in \cite[Chapter 8]{rierathese}, where the author proposed another decomposition of the Brownian sphere into three pieces, and which is also closely related to Voronoï cells.

Finally, let us mention that throughout this paper, we have only considered the Brownian sphere with a fixed delay (or a delay chosen uniformly at random). However, one may wonder what happens as the delay varies. For instance, consider a Brownian sphere with two distinguished points at distance $1$. For every $\delta\in[-1,1]$, let $\eta^{(\delta)}$ denote the topological boundary of $\Theta_\delta$ and $\overline{\Theta}_\delta$. More precisely, it is easy to check that for every $\delta\in(-1,1)$, $\Theta_\delta/\text{Int}(\overline{\Theta}_\delta)=\overline{\Theta}_\delta/\text{Int}(\overline{\Theta}_\delta)$, and we call this set $\eta^{(\delta)}$. By Proposition \ref{simple curve}, for every fixed $\delta$, almost surely, $\eta^{(\delta)}$ is a simple curve. However, it is not clear at all that this property holds \textit{simultaneously} for every $\delta\in(-1,1)$. In fact, we believe that almost surely, there is a countable set of \textbf{exceptional delays} at which $\eta^{(\delta)}$ is not a simple curve, and cuts out a macroscopic part of the Brownian sphere. We hope to investigate further these questions in the future.

\printbibliography
\end{document}